\documentclass[13pt]{article}
\usepackage{definitions}
\usepackage{arydshln}

\title{An Overview of GPU-based First-Order Methods for Linear Programming and Extensions}
\author{Haihao Lu\thanks{MIT, Sloan School of Management (haihao@mit.edu).} \and Jinwen Yang\thanks{University of Chicago, Department of Statistics (jinweny@uchicago.edu).}}
\date{}

\begin{document}

\maketitle

\begin{abstract}
The rapid progress in GPU computing has revolutionized many fields, yet its potential in mathematical programming, such as linear programming (LP), has only recently begun to be realized. This survey aims to provide a comprehensive overview of recent advancements in GPU-based first-order methods for LP, with a particular focus on the design and development of cuPDLP. We begin by presenting the design principles and algorithmic foundation of the primal-dual hybrid gradient (PDHG) method, which forms the core of the solver. Practical enhancements, such as adaptive restarts, preconditioning, Halpern-type acceleration and infeasibility detection, are discussed in detail, along with empirical comparisons against industrial-grade solvers, highlighting the scalability and efficiency of cuPDLP. We also provide a unified theoretical framework for understanding PDHG, covering both classical and recent results on sublinear and linear convergence under sharpness conditions. Finally, we extend the discussion to GPU-based optimization beyond LP, including quadratic, semidefinite, conic, and nonlinear programming. 
\end{abstract}

\tableofcontents

\section{Introduction}
Over the past decade, the scale and sophistication of computational infrastructure, thanks to the development of machine and deep learning, have advanced at an extraordinary pace. Enabled by rapid developments in both hardware and software, the deep learning field now routinely trains models with trillions of parameters. This progress has been driven by massively parallel accelerators such as GPUs and TPUs, coupled with highly optimized software frameworks like PyTorch and JAX, which emphasize first-order methods (FOMs) and large-scale data-parallel computation.

In contrast, the computational infrastructure for mathematical programming, for example, linear programs (LPs), has remained grounded in a mature but shared-memory CPU-centric paradigm. Industrial-grade optimization solvers build on decades of algorithmic innovation, primarily leveraging simplex and interior-point methods. These methods are tightly coupled with direct sparse matrix factorizations and are carefully engineered for shared-memory CPU architectures. As a result, they offer highly accurate and certifiably optimal solutions across a broad range of applications.

Despite such algorithmic maturity, the size of solvable LPs remains constrained by current shared-memory CPU technologies, struggling to handle LPs containing billions of variables even on high-end CPUs. This limitation starkly contrasts with the immense scale routinely achieved by modern deep learning models.\footnote{We acknowledge that ``solving'' typically denotes differing levels of accuracy in deep learning versus linear programming contexts. Nonetheless, a substantial disparity exists regarding the tractable problem sizes in these two fields.}
This growing scalability gap raises an essential question:

\vspace{0.1cm}
\begin{center}
\textit{Can we leverage GPUs and first-order methods to improve the scalability and efficiency of linear programming?}
\end{center}
\vspace{0.1cm}

This question is increasingly important in practice, driven by the expanding data collection techniques employed in large-scale applications across numerous domains, including engineering, finance, and operations. In these fields, scalable and efficient optimization tools are often vital. Developing GPU-based first-order method solvers for LPs thus represents an exciting opportunity to close this computational gap and extend the advantages of modern parallel hardware to classical optimization challenges.

In this survey, we provide an overview of the emerging field of GPU-based linear programming, with a particular focus on recent developments surrounding (cu)PDLP. We examine its algorithmic foundations, computational techniques, and theoretical properties, aiming to offer a cohesive understanding of this rapidly evolving area.

\subsection{Linear Programming}
Linear programming (LP), which refers to optimizing a linear function over a system of linear inequality constraints, is one of the most fundamental classes of mathematical programming. It is used in many arenas of the global economy, including transportation, telecommunications, production and operations scheduling, as well as in support of strategic decision-making~\cite{hazell1974competitive,delson1992linear,dahleh1994control,liu2008choice,charnes1959application,zhou2008linear}. LP is used to allocate resources, plan production, schedule workers, plan investment portfolios and formulate marketing (and military) strategies. The versatility and economic impact of linear optimization models in today’s industrial world is truly awesome~\cite{freund1994professor,sarker2007optimization}.

The geometry of linear programming is rooted in the structure of polyhedra, which are sets defined as the intersection of a finite number of linear inequalities. In LP, the feasible region, determined by the system of linear constraints, describes a polyhedron. Each inequality constraint defines a half-space, and the intersection of these half-spaces is a polyhedron. The vertices of the polyhedron are critical in LP because there is at least one optimal solution at a vertex. The level set of the linear objective function defines a hyperplane, and solving the LP involves moving this hyperplane to its furthest extent while remaining tangent to the polyhedron. This interplay between polyhedral geometry and optimization underpins the algorithmic approaches to solving LPs, such as simplex methods, ellipsoid methods, interior-point methods (also known as barrier methods) and decomposition methods. 

\begin{itemize}
    \item {\bf Simplex method.} The simplex method, introduced by George Dantzig in 1947~\cite{dantzig1948programming}, is widely regarded as the origin of optimization as a scientific discipline. The simplex method operates by moving along the vertices of the feasible polyhedron to iteratively improve the objective function. Despite its exponential worst-case complexity~\cite{klee1972good}, its strong practical performance earned it recognition as one of the Top 10 Algorithms of the 20th century~\cite{cipra2000best}. A key extension is the dual simplex method~\cite{dantzig1963linear}, which solves a slightly modified problem. To further improve efficiency of primal and dual simplex, the steepest edge variation of the simplex algorithm~\cite{forrest1992steepest} was developed as a pivoting rule that selects the direction yielding the greatest improvement per unit move. This heuristic often reduces iteration count by guiding the search more effectively. Together, these and many other advencements, including primal simplex, dual simplex, and steepest edge, remain central components of many modern LP solvers~\cite{koberstein2005dual,maros2002computational}.
    
    \item {\bf Ellipsoid method.} Introduced by Nemirovski and Yudin in the 1970s~\cite{yudin1976informational}, the ellipsoid method marked a major milestone in optimization theory. Its significance in linear programming was cemented by Khachiyan’s 1979 result showing that LPs could be solved in polynomial time~\cite{khachiyan1980polynomial}. Unlike the simplex method, which moves along vertices, the ellipsoid method iteratively shrinks an enclosing ellipsoid containing an optimal solution. This geometric approach established a foundational result in complexity theory~\cite{bland1981ellipsoid}.
    Despite its theoretical importance, the ellipsoid method saw limited practical use due to its slow convergence and high computational overhead compared to simplex method (and later interior-point methods). Nonetheless, it reshaped the landscape of optimization and helped catalyze the development of more efficient polynomial-time algorithms, such as interior-point methods.
    
   \item {\bf Interior-point methods.} Interior-point methods (IPMs) represent a major advancement in optimization, offering an alternative to vertex-based approaches like the simplex method~\cite{nesterov1994interior,wright1997primal,renegar2001mathematical}. Early ideas trace back to the 1960s, with Fiacco and McCormick’s barrier methods for nonlinear programming~\cite{fiacco1964sequential,fiacco1964computational}, and Dikin’s iterative interior-point algorithm in 1967~\cite{dikin1967iterative}. These methods, however, saw limited use due to computational challenges in handling large-scale problems and the inherent numerical instability of the methods.

    The breakthrough came in 1984 with Karmarkar’s projective scaling algorithm~\cite{karmarkar1984new}, a polynomial-time method that operates within the interior of the feasible region, offering a theoretically efficient and practically competitive alternative to simplex. This sparked widespread interest and led to the development of more powerful variants, most notably the primal-dual interior-point method~\cite{wright1997primal,mehrotra1992implementation}, known for its robust convergence and efficiency in simultaneously solving primal and dual formulations.

    Modern IPMs rely on Newton-like iterations to solve systems derived from the Karush-Kuhn-Tucker (KKT) conditions and have been extended to broader problem classes, including convex quadratic, conic, and semidefinite programming~\cite{vandenberghe2010cvxopt,ye1994nl,nesterov1998primal,nesterov1997self,andersen2003implementing,dahl2022primal,wachter2006implementation,toh1999sdpt3}. Their polynomial-time guarantees and strong empirical performance have made IPMs a cornerstone of modern solvers.

    \item {\bf Decomposition methods.} Decomposition techniques, such as Dantzig-Wolfe~\cite{dantzig1960decomposition} and Benders~\cite{bnnobrs1962partitioning}, have been pivotal in solving large-scale LPs with particular problem structure. Dantzig-Wolfe decomposition targets block-angular problems that are common in multi-stage planning and network flows by reformulating them into a master problem and independent subproblems, iteratively improving the solution via column generation. This approach underpins methods used in applications like vehicle routing and airline crew scheduling. In addition, Benders decomposition handles problems with complicating variables by separating them into a master problem and subproblem. Benders cuts, derived from solving the subproblem, are added iteratively to the master until convergence. Benders decomposition has been widely used in facility location, energy planning, and telecommunications. While effective for structured problems, both techniques suffer from slow convergence in later iterations and are less applicable to general-purpose LPs due to their reliance on specific structural assumptions.
\end{itemize}

\subsection{Comparison between CPU and GPU}
Despite their differences in theoretical foundations and practical applications, each of the methods for LP has been profoundly shaped by the architecture of the Central Processing Unit (CPU), which served as the dominant computational platform for decades and was virtually the only widely accessible and rapidly advancing computing hardware for a long time. As a general-purpose processor good at sequential execution, a CPU typically comprises a small number of powerful cores and complex control logic. This architecture excels at low-latency, instruction-by-instruction processing, making it particularly well-suited for algorithms with strong sequential characteristics.

The influence of this CPU-centric paradigm on LP algorithm design has been substantial. Classical methods such as the simplex algorithm are inherently sequential, progressing through a series of pivot steps where each iteration depends on the outcome of the previous one. Similarly, interior-point methods require the repeated solution of linear systems, a process that, while admitting partial parallelism, remains dominated by sequential operations. Consequently, algorithm designers prioritized reducing the number of iterations and maximizing the efficiency of core computational routines. This led to the development of highly optimized data structures for sparse matrices and finely tuned linear algebra kernels that leverage the CPU’s fast memory access patterns and cache performance. As a result, modern LP solvers have become remarkably efficient on single-threaded or moderately parallel CPU architectures, setting a high bar for performance that alternative platforms must strive to meet.

More recent advances in computational hardware have brought about a paradigm shift in high-performance computing, most notably through the rise of Graphics Processing Units (GPUs). Originally developed to accelerate rendering in computer graphics, GPUs have evolved into powerful general-purpose computing platforms. This transition has been accelerated by the unprecedented success of deep learning, where GPUs have proven indispensable for training and inference in large-scale neural networks~\cite{nvidiaai}. Their ability to perform massive numbers of floating-point operations in parallel has also made GPUs a cornerstone of modern machine learning, scientific simulation, and data analytics~\cite{nvidiahpc}.

At a low level, the design philosophy of GPUs differs fundamentally from that of CPUs. Whereas CPUs are optimized for low-latency execution of complex, sequential tasks, GPUs are optimized for high-throughput execution of simple, parallel tasks. A typical GPU contains thousands of lightweight cores organized into streaming multiprocessors, each capable of performing arithmetic operations simultaneously across vast arrays of data. This architecture is accompanied by a memory hierarchy designed to support data-parallel computation, though with higher latency and less flexible control flow than CPUs. As a result, while a CPU may excel at single-threaded tasks, a GPU can deliver orders-of-magnitude higher peak performance on workloads that can be expressed in parallel form.

\begin{table}[ht!]
\begin{tabular}{lcc}
\hline
% \vspace*{0.05cm}
& \multicolumn{1}{c}{\textbf{GPU}} & \multicolumn{1}{c}{\textbf{CPU}}\\
\hline
\textbf{Core Design}           & thousands of lightweight cores             & fewer, more powerful cores              \\
% & & \\
\textbf{Control and Execution} & shared control units, batch processing & dedicated control units, task switching \\
% & & \\
\textbf{Memory Design}     & huge throughput                        & low latency                             \\
% & & \\
\textbf{Processing Model}      & thousands of simple threads       & fewer, more complex threads             \\ 
% & & \\
\hline
\end{tabular}
\caption{Summary of major distinctions between GPUs and CPUs}
\label{tab:cpu-gpu}
\end{table}

Table \ref{tab:cpu-gpu} further highlights the distinctions between GPUs and CPUs, reflecting their specialized architectures. GPUs prioritize parallel execution, featuring thousands of smaller cores and shared control units for batch processing, making them ideal for high-throughput\footnote{Throughput refers to the amount of data processed per unit time.}  tasks. In contrast, CPUs have fewer, more powerful cores with dedicated control units, excelling in task switching and complex decision-making. Furthermore, GPUs maximize throughput with high-bandwidth memory\footnote{Memory bandwidth measures the peak rate at which data can be read from or written to memory per unit time, determining the maximum throughput for data transfer between memory and processing units.}, while CPUs prioritize low-latency\footnote{Memory latency refers to the delay between a memory access request and the availability of the requested data, impacting how quickly a processor can retrieve information.} access for sequential processing. GPUs handle thousands of lightweight threads efficiently, whereas CPUs manage fewer but more complex threads. Together, these differences define their respective strengths, with GPUs optimized for large-scale computations and CPUs better suited for logic-intensive tasks. 

These architectural features of GPUs naturally lead to a broader question:
\textit{Can GPUs be leveraged to accelerate and scale up linear programming—and more generally, mathematical programming?}
A natural starting point is to use GPUs to accelerate key computational components of linear programming solvers, namely, the basic linear algebra subroutines (BLAS). In large-scale mathematical programming, three such subroutines are especially prevalent: sparse matrix-vector multiplication (SpMV), commonly used in various first-order methods; sparse Cholesky factorization, widely used in interior-point methods for linear programming\footnote{{The experiment employs the CHOLMOD package for CPU-based factorization and the cuDSS package for GPU-based factorization. While both libraries provide robust and efficient sparse direct solvers, the use of a customized factorization routine tailored to the specific problem structure may lead to substantial performance gains.}}; and sparse $\operatorname{LDL^\top}$ factorization, frequently employed in Alternating Direction Method of Multipliers (ADMM) for LP.

Recent developments in GPU sparse linear algebra libraries, such as \href{https://developer.nvidia.com/cusparse}{cuSPARSE} and \href{https://docs.nvidia.com/cuda/cudss/index.html}{cuDSS} by NVIDIA, have enabled GPU acceleration for all three operations. However, the degree of acceleration varies considerably across these tasks. As illustrated in Figure~\ref{fig:spmv-chol}, GPU speedup for SpMV scales linearly with problem size, making it particularly well-suited for large-scale problems. In contrast, the speedups for Cholesky and $\operatorname{LDL^\top}$ factorizations depend heavily on the structure and sparsity of the problem instance, showing no consistent scaling behavior. Moreover, the overall performance gains are substantially higher for SpMV compared to Cholesky, while $\operatorname{LDL^\top}$ factorization on average shows little to no clear advantage on GPUs.

\begin{figure}[ht!]
	% \centering
    \hspace{-1.3cm}
	\begin{tabular}{c c c c c}
		& \includegraphics[width=0.33\textwidth]{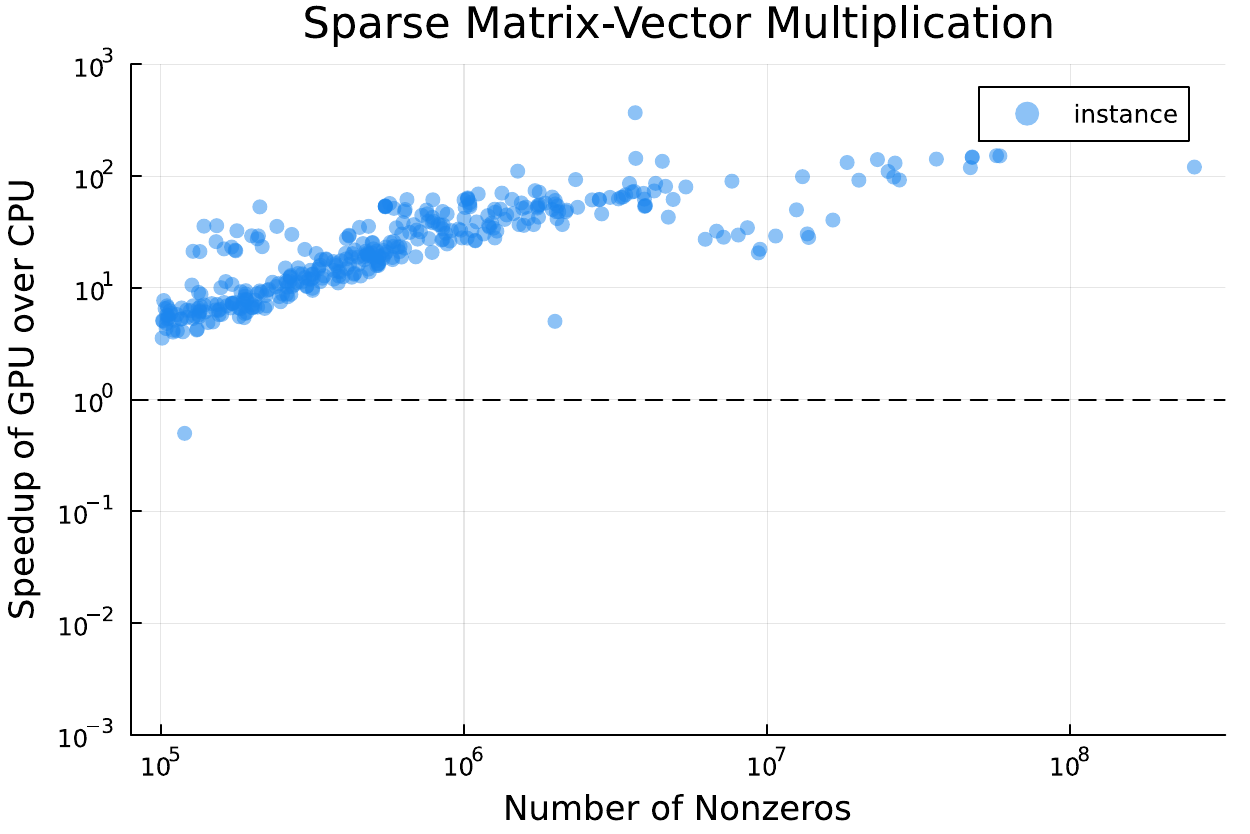}
        &
        & \includegraphics[width=0.33\textwidth]{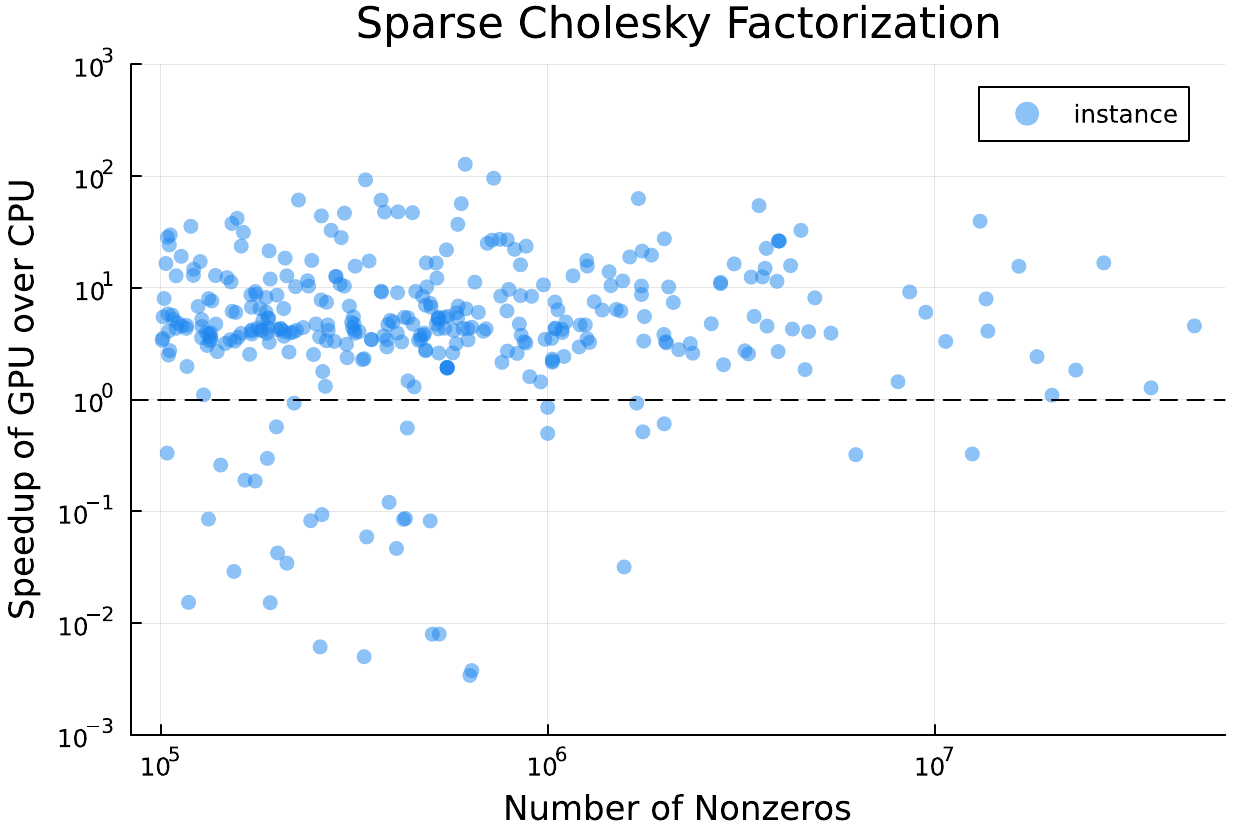} 
        & \includegraphics[width=0.33\textwidth]{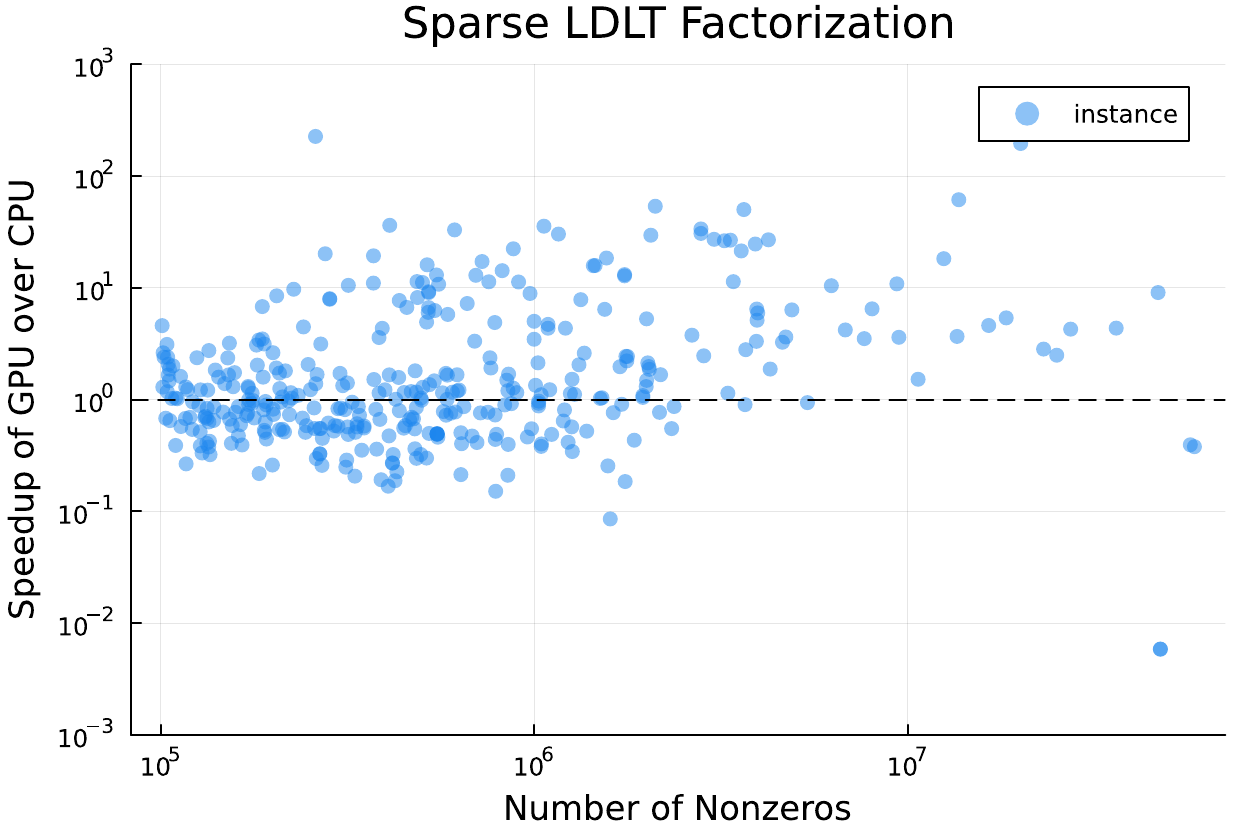}
	\end{tabular}
	\caption{Scatter plots showing the speedup in running time speedup of GPU versus CPU for SpMV, Cholesky factorization and  $\operatorname{LDL^\top}$ factorization with data of 383 MIPLIB instances (see Section \ref{sec:setup} for detail). {Particularly, SpMV computes $Ax$ with a random $x$, sparse Cholesky factorization computes Cholesky factor of matrix $AA^\top$, and sparse $\operatorname{LDL^\top}$ factorization computes $\operatorname{LDL^\top}$ factor of $\begin{pmatrix}
	    I & -A^\top\\-A & 0
	\end{pmatrix}$, where $A$ is the constraint matrix in the problem instance.} The $x$-axis in the plots represents the number of nonzero elements in the sparse matrices in the benchmark set, while the $y$-axis shows the speedup of GPU over CPU on a logarithmic scale. A horizontal dashed line at speedup equal 1 marks the threshold where GPU and CPU perform equally; points above this line indicate cases where GPU outperforms CPU, whereas points below indicate cases where CPU is faster.}
	\label{fig:spmv-chol}
\end{figure}

Furthermore, despite the low-level speedups achievable through GPU-accelerated sparse Cholesky factorization, the overall performance gains for fully GPU-based interior-point LP solvers remain modest. One key reason is that Cholesky factorization, while computationally intensive, does not consistently dominate the total runtime in modern solvers. For example, as observed by Gurobi~\cite{gurobinews}, substantial portions of computational effort are spent on other stages such as step computation and control logic, which are inherently sequential and difficult to parallelize. These components often remain bottlenecks even when linear algebra routines are accelerated.  As such, significant improvements in factorization speed do not always translate into significant improvements in overall LP solve times, unless accompanied by a more comprehensive redesign of the solver pipeline to expose and exploit parallelism throughout.

\begin{table}[ht!]
\centering
\begin{tabular}{cccc:c}
\hline
                 & \textbf{SpMV}    & \textbf{Cholesky Factorization} & \textbf{$\operatorname{\bf LDL^{\bf \top}}$ Factorization} & \textbf{Instances} \\ \hline
\textbf{Allocated Memory} & 0.971 MB & 481.067 MB  & 243.486 MB & 4.915 MB            \\ \hline
\end{tabular}
\caption{Geometric mean of additional memory usage of GPU-based SpMV, sparse Cholesky and $\operatorname{LDL^\top}$ factorizations on 383 MIPLIB instances (see Section \ref{sec:setup} for detail).}
\label{tab:mem}
\end{table}

In addition, GPUs face significant memory constraints that further restrict their utility for large-scale LPs. While GPUs offer superior floating-point throughput, their onboard memory is typically an order of magnitude smaller than that of high-end CPUs. This is especially limiting for interior-point methods, where Cholesky or $\operatorname{LDL^\top}$ factorizations can require one to two orders of magnitude more memory than the original problem data. Table~\ref{tab:mem} reports the geometric mean of additional memory requirements for SpMV, Cholesky factorization, and $\operatorname{LDL^\top}$ factorization across 383 MIPLIB instances. While SpMV imposes negligible overhead, both factorization methods demand substantially more memory, further constraining the applicability of GPU-based interior-point methods to very large LPs.

These challenges motivate the development of new linear programming methods that are more naturally aligned with GPU architectures. To fully leverage the computational potential of GPUs, such methods should satisfy two key criteria: (1) the majority of computational time should be concentrated in components that are highly parallelizable, with minimal reliance on sequential or control-heavy routines. Only under this condition can LP algorithms effectively exploit the massive throughput and memory bandwidth advantages offered by GPU hardware. (2) Due to limited device memory of GPUs, computational primitives with cheaper memory cost are favorable, particularly for large-scale instances.

\subsection{First-order Methods for LP}

A promising direction under this paradigm is to build LP algorithms around SpMV as a core primitive, which is one of the most memory-efficient and parallelizable operations available for large-scale optimization problems. On modern GPUs, carefully optimized SpMV routines can achieve substantial speedups over their CPU counterparts. Crucially, SpMV avoids the substantial memory costs associated with sparse matrix factorizations, as compared in Table \ref{tab:mem}, and it scales well with problem size. As shown in the leftmost panel of Figure \ref{fig:spmv-chol} for SpMV, GPU exhibits a clear advantage over CPU as the number of nonzeros increases. The speedup is consistently above 1 for most instances, reaching up to more than 100x for larger matrices. This trend aligns with expectations, as SpMV is highly parallelizable and benefits from GPU acceleration. Notably, the speedup gets larger as instance sizes grow, suggesting that GPUs gain more speedup over CPUs for larger-scale SpMV.

In this context, first-order methods (FOMs)~\cite{nesterov2018lectures,beck2017first}, which can often purely base on matrix-vector multiplication, emerge as a promising alternative. Unlike traditional LP algorithms that rely on matrix factorization, FOMs update their iterates using only gradient information, making matrix-vector multiplication the primary computational bottleneck. This makes SpMV an ideal foundation for FOMs tailored to GPU architectures. Unlike traditional approaches that involve costly factorization steps, FOMs based on SpMV are structurally simple and computationally transparent: they spend nearly all of their runtime on operations that are well-suited to GPUs, primarily repeated SpMV and vector updates. These operations are not only memory-efficient but also highly parallelizable, enabling solvers to take full advantage of the GPU’s throughput and bandwidth. The minimal reliance on sequential logic or control-heavy routines further enhances efficiency and implementation simplicity. As a result, SpMV-centric FOMs offer a compelling pathway toward LP solvers that are both efficient and highly compatible with modern GPU computing, particularly in applications where solving many large-scale LPs rapidly is critical.

Indeed, using FOMs for LP is not a new idea. As early as the 1950s, initial efforts were made to employ FOMs for solving LP problems. Notably, with the intuition to make big jumps rather than ``crawling along edges", \cite{brown1951computational} discussed steepest ascent to maximize the linear objective under linear inequality constraints. \cite{zoutendijk1960methods,zoutendijk1970some} pioneered the development of feasible direction methods for LP, where the iterates consistently move along a feasible descent direction. The subsequent development of steepest descent gravitational methods in~\cite{chang1989steepest} also falls within this category. Another early approach to first-order methods for solving LP is the utilization of projected gradient algorithms~\cite{rosen1961gradient,lemke1961constrained}. All these methods, however, still require solving linear systems to determine the appropriate direction for advancement. Solving the linear systems that arise during the update can be highly challenging for large instances. Moreover, the computation of projections onto polyhedral constraints can be arduous and even intractable for large-scale instances. Therefore, despite their potential, these FOMs are also not inherently well-suited for GPUs. 

Thus, the development of new FOM-based LP algorithms specifically tailored for GPUs is essential to fully harness their computational power. Recently, there has been significant progress in the creation of GPU-based LP solvers, improving solution time for LPs as problem size grows. These advancements aim to bridge the gap between the computational capabilities of GPUs and the requirements of LP algorithms, particularly,
\begin{itemize}
    \item PDLP~\cite{applegate2021practical,applegate2025pdlp} is a linear programming solver based on first-order methods, designed to be practical for large-scale problems. It builds upon the restarted primal-dual hybrid gradient (PDHG) method~\cite{applegate2023faster}, which avoids matrix factorizations in favor of simple iterative updates involving matrix-vector multiplications, with many heuristic improvements for practical performance. PDHG is known for its scalability and efficiency on large problems, and it avoids computationally sequential operations, making it well-suited for parallel hardware. With its focus on scalability and parallelizability, the CPU-based PDLP has paved the way for a new class of LP solvers aligned with modern computing architectures such as GPUs.

    \item cuPDLP~\cite{lu2023cupdlp,lu2023cupdlpc} is a GPU-based solver for large-scale linear programming that builds upon the restarted primal-dual hybrid gradient method~\cite{applegate2023faster}. As a GPU-based extension of PDLP~\cite{applegate2021practical,applegate2025pdlp}, cuPDLP inherits the theoretical strengths of PDHG while incorporating several heuristic improvements tailored for modern GPUs. By leveraging these techniques, cuPDLP scales efficiently to solve massive LP problems and achieves strong performance, highlighting the potential of first-order methods like PDHG to address the challenges of large-scale optimization in GPU-based environments.
    
    \item HPR-LP~\cite{chen2024hpr} is a recently developed GPU-based solver for linear programming that introduces a novel approach to tackling large-scale LP problems. It employs the Halpern Peaceman-Rachford (HPR) method with semi-proximal terms, leveraging the method's theoretical strengths, including robust convergence properties and enhanced efficiency for solving LP problems. To further improve practical performance, HPR-LP integrates adaptive techniques such as dynamic restarts and penalty parameter updates, ensuring better scalability and robustness. The solver demonstrates remarkable improvements in computational speed, making it a promising tool for addressing large-scale LPs.

    \item Another approach is the combination of the Alternating Direction Method of Multipliers (ADMM)~\cite{eckstein1992douglas,boyd2011distributed} with conjugate gradient (CG) methods. Conjugate gradient methods solve linear systems using only matrix-vector multiplications, which can be efficiently parallelized on GPUs. This approach has been adopted in quadratic programming (QP) solvers such as OSQP~\cite{stellato2020osqp} and conic programming solvers like SCS~\cite{o2016conic,o2021operator}, both of which utilize CG on GPUs to solve linear systems. However, a notable limitation of this approach is that each iteration of the ADMM algorithm typically requires tens to hundreds of conjugate gradient steps (i.e., several matrix-vector multiplications), which can result in significant computational overhead. Despite this, the integration of ADMM with CG remains a viable direction for leveraging GPUs in optimization.

\end{itemize}

\subsection{Other Mathematical Programming}

Beyond linear programming, there have been significant advancements in GPU-based solvers for other classes of optimization problems. These include PDQP~\cite{lu2023practical} and PDHCG~\cite{huang2024restarted}, which are designed for solving convex quadratic programming problems, leveraging GPU parallelism to handle the large-scale matrix-vector operations inherent in these tasks. Similarly, FOM-based solvers such as cuLoRADS~\cite{han2024accelerating}, cuHALLaR~\cite{aguirre2025cuhallar} and ALORA~\cite{ding2025new}, have emerged for semidefinite programming, utilizing low-rank factorization and GPU-optimized computations to solve massive problems that were previously intractable. For conic programming, PDCS~\cite{lin2025pdcs} extends cuPDLP to solve conic linear programming, while  CuClarabel~\cite{chen2024cuclarabel} incorporates GPU-specific techniques to accelerate the solution of problems involving second-order, semidefinite, and exponential cones. In the realm of nonlinear programming, MadNLP~\cite{shin2020graph,shin2023accelerating} has demonstrated the potential of GPU-based solvers by employing efficient differentiation and condensed-space interior-point methods tailored for GPUs.

These advances demonstrate the growing potential of GPU-based optimization solvers to handle increasingly complex and large-scale problems across various domains. As GPU architectures continue to evolve, and with the increasing demand for solving high-dimensional optimization problems in fields like machine learning, engineering, and operations research, we anticipate even more innovative developments in this area, further expanding the capabilities and efficiency of GPU-based solvers.

\subsection{Paper Organization}

This survey provides an overview of recent advancements in this new area of research. Section \ref{sec:design} outlines the step-by-step design of basic first-order methods for linear programming. In Section \ref{sec:pdlp}, we examine various practical enhancements in PDLP that improve convergence and provide a numerical comparison between cuPDLP and industrial-grade solvers. Section \ref{sec:theory} offers theoretical insights into PDHG for LP. In Section \ref{sec:beyondlp}, we explore recent developments in GPU-based solvers beyond LP. Finally, we conclude with a summary of key findings and open questions in the field.

\section{Designing FOMs for LP from First Principles}\label{sec:design}

In this section, we discuss how to design different FOMs for LP. We start by discussing the basic FOMs, including, projected gradient descent (PGD), gradient descent ascent (GDA), and proximal point method (PPM) to solve LP as well as their limitations. This leads to primal-dual hybrid gradient (PDHG), the base algorithm utilized in an LP solver PDLP, described in the next section. 

More specifically, we consider standard form LP of the form

\begin{equation}\label{eq:primal}
    \begin{aligned}[c]
    \min_{x\in \mathbb R^n}~~ &~ c^\top x \\
\text{s.t.}~~ &~ Ax=b \\
& ~ x\geq 0 \ ,
    \end{aligned}
\end{equation}
where $A\in\mathbb R^{m\times n}$, $b\in \mathbb R^m$ and $c\in \mathbb R^n$, and its dual problem

\begin{equation}\label{eq:dual}
    \begin{aligned}[c]
    \max_{y\in \mathbb R^m}~~ &~ b^\top y \\
\text{s.t.}~~ &~ A^\top y\leq c \ .
    \end{aligned}
\end{equation}
While the discussion in this section is on standard-form LP~\eqref{eq:primal}, most of the algorithms and their behaviors can be directly extended to other LP forms.

The most basic FOM for solving a constrained optimization problem, such as LP \eqref{eq:primal}, is perhaps the projected gradient descent (PGD), which updates the variable by a gradient descent step and then projects to the constraint set:
\begin{equation*}\label{alg:pgd}
    x^{k+1}=\text{proj}_{\{x\in\mathbb R_+^n|Ax=b\}}(x^k-\eta c)\ ,
\end{equation*}
where $\eta$ is the step-size of the algorithm.
Despite the nice theoretical guarantees of PGD for general constrained convex optimization problems~\cite{nesterov2018lectures,beck2017first}, computing the projection onto the constrained set (i.e., the intersection of an affine subspace and the positive orthant) involves solving a quadratic programming problem, which can be as hard as or even harder than solving the original LP, and thus PGD is not a practical algorithm. 

To disentangle linear constraints and (simple) nonnegativity of variables, a natural idea is to dualize the linear constraints $Ax=b$ and consider the primal-dual form of the problem \eqref{eq:minmax}:
\begin{equation}\label{eq:minmax}
    \min_{x}\max_y\; L(x,y):=c^\top x+\iota_{\mathbb R_+^n}(x)-y^\top Ax+b^\top y \ .
\end{equation}
Convex duality theory~\cite{boyd2004convex} shows that the saddle points to \eqref{eq:minmax} can recover the optimal solutions to the primal problem~\eqref{eq:primal} and the dual problem~\eqref{eq:dual}. 
For the primal-dual formulation of LP~\eqref{eq:minmax}, the most natural FOM  is perhaps the projected gradient descent-ascent method (GDA), which performs a gradient descent step in $x$ and a gradient ascent step in $y$, has iterated update:
\begin{equation*}\label{alg:gda}
    \begin{aligned}
        & x^{k+1}=\text{proj}_{\mathbb R^n_+}(x^k+\eta A^\top y^k-\eta c) \\ 
        & y^{k+1}= y^k-\sigma Ax^k+\sigma b\ ,
    \end{aligned}
\end{equation*}
where $\eta$ and $\sigma$ are the primal and the dual step-size, respectively.
The projection of GDA is onto the positive orthant for the primal variables and it is cheap to compute. Unfortunately, GDA does not converge to a saddle point of \eqref{eq:minmax}. For instance, Figure \ref{fig:pdhg-gda} plots the trajectory of GDA on a simple primal-dual form of LP
\begin{equation}\label{eq:bilinear}
    \min_{x\ge 0}\max_y\; (x-3)y \ ,
\end{equation}
where $(3,0)$ is the unique saddle point. As we can see, the GDA iterates diverge and spiral farther away from the saddle point. Thus, GDA is not an appropriate algorithm for \eqref{eq:minmax}.

\begin{figure}[ht!]
	\centering\includegraphics[scale=0.4]{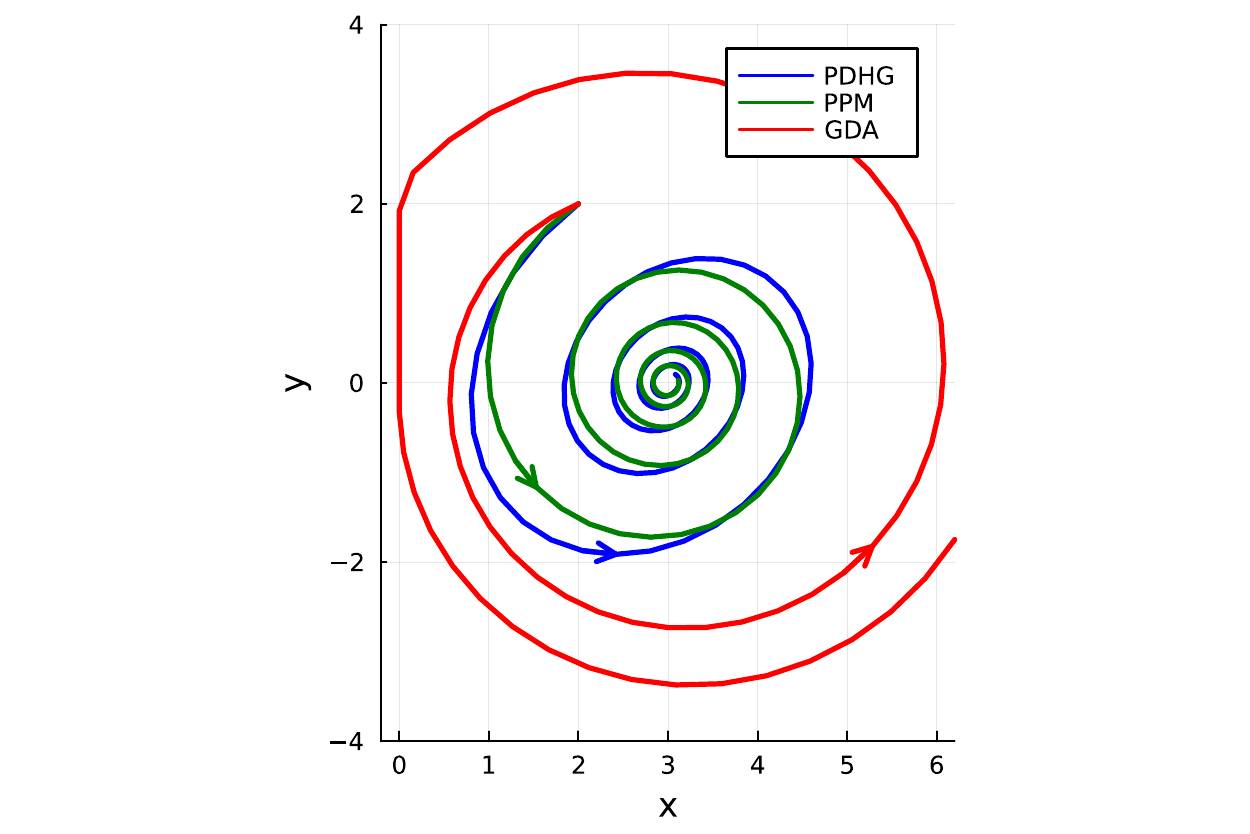}
	\caption{Trajectories of GDA, PPM and PDHG to solve a simple bilinear problem \eqref{eq:bilinear} with initial solution $(2, 2)$ and step-size $\eta=\sigma=0.2$.}
	\label{fig:pdhg-gda}
\end{figure}

Another classic candidate algorithm to solve such primal-dual problem is the proximal point method (PPM), proposed in the seminal work of  Rockafellar~\cite{rockafellar1976monotone}. The basic idea of PPM is to use the (sub-)gradient in the next iteration to update the step, in constrast to GDA where the (sub-)gradient in the current iterate is utilized, and its iterated update can be written as follows:
\begin{equation}\label{eq:ppm}
    (x^{k+1},y^{k+1})\leftarrow \arg\min_{x\geq 0}\max_y\;L(x,y)+\frac{1}{2\eta}\|x-x^k\|_2^2-\frac{1}{2\sigma}\|y-y^k\|_2^2\ .
\end{equation}
Unlike GDA, PPM exhibits nice convergence properties for solving primal-dual problems \cite{rockafellar1976monotone} (see Figure \ref{fig:pdhg-gda} for an example). However, its update rule is implicit and requires solving the subproblems arising in \eqref{eq:ppm}. This drawback makes PPM more of a conceptual rather than a practical algorithm.

To overcome these issues of GDA and PPM, we introduce primal-dual hybrid gradient method (PDHG, a.k.a Chambolle-Pock algorithm)~\cite{chambolle2011first,zhu2008efficient} and alternative direction method of multipliers. 
PDHG is a first-order method for convex-concave primal-dual problems originally motivated by applications in image processing. In the case of LP, the update rule is straightforward:
\begin{equation}\label{eq:pdhg}
    \begin{aligned}
        & x^{k+1}\leftarrow \text{proj}_{\mathbb R^n_+}(x^k+\eta A^\top y^k-\eta c) \\ 
        & y^{k+1}\leftarrow y^k-\sigma A(2x^{k+1}-x^k)+\sigma b\ ,
    \end{aligned}
\end{equation}
where $\eta$ is the primal step-size and $\sigma$ is the dual step-size. Similar to GDA, the algorithm alternates with the primal and the dual variables, and the difference is in the dual update, one utilizes the gradient at the extrapolated point $2x^{k+1}-x^k$. The extrapolation helps with the convergence of the algorithm, as we can see in Figure \ref{fig:pdhg-gda}. Indeed, one can show the PDHG is a preconditioned version of PPM (see the next section for more details), and thus share the nice convergence properties with PPM, but it does not require solving the implicit update~\ref{eq:ppm}. The computational bottleneck of PDHG is the matrix-vector multiplication (i.e., in $A^\top y$ and $Ax$). 

\begin{figure}[ht!]
	% \centering
    \hspace{0.7cm}
    \begin{subfigure}{0.5\textwidth}
        \includegraphics[width=1.0\textwidth]{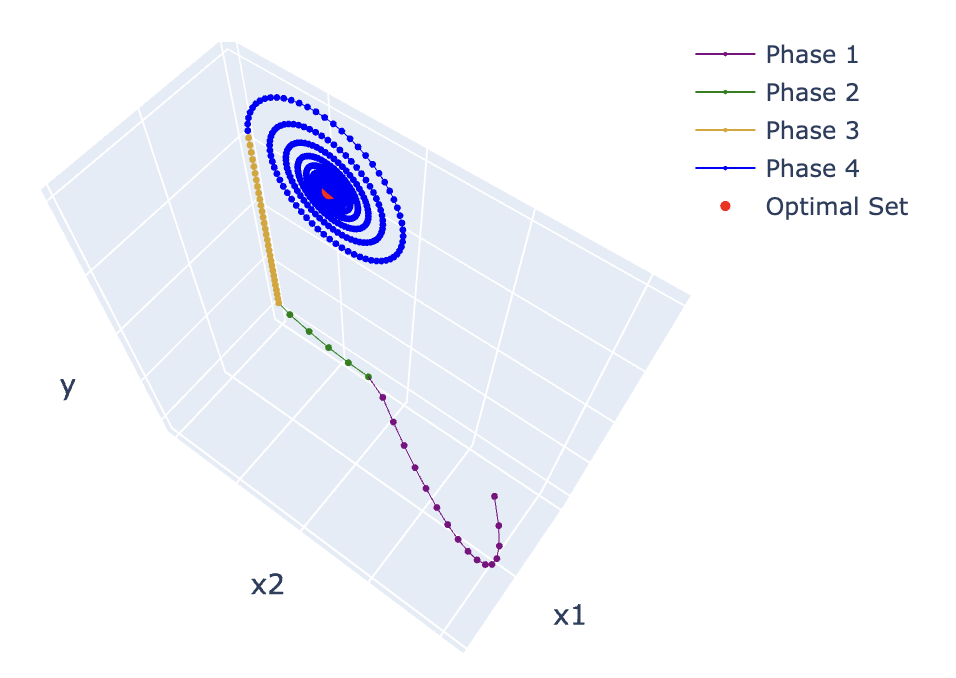}
        \caption{Primal-dual trajectory of PDHG for solving\\
            $\min\ 2x_1+3x_2\ \ $\\
        $\mathrm{s.t.}\ x_1+2x_2=1$\\
        $x_1\geq 0,x_2\geq 0$
        }
    \end{subfigure}
    \begin{subfigure}{0.45\textwidth}
        \includegraphics[width=1.0\textwidth]{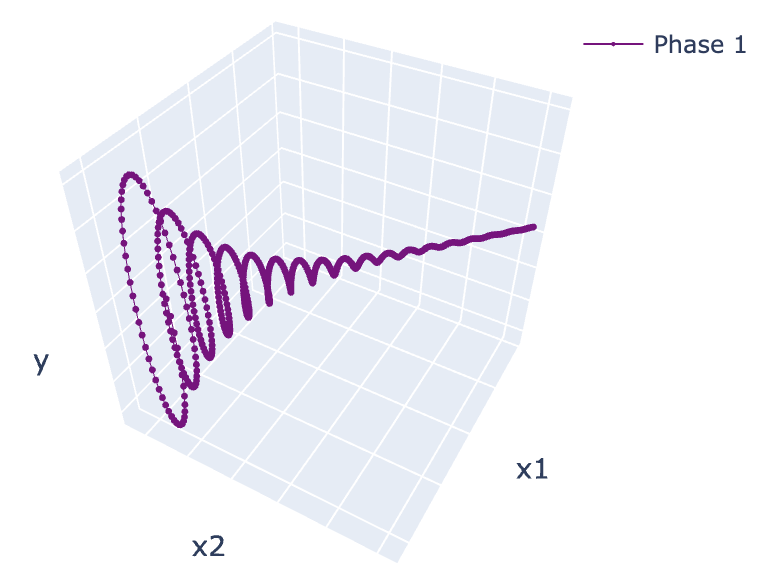}
        \caption{Primal-dual trajectory of PDHG for solving\\ 
        $\min\ 2x_1+3x_2\ \ $\\
        $\mathrm{s.t.}\ x_1+2x_2=1$\\
        $x_1\geq -10$}
    \end{subfigure}
	\caption{Primal-dual trajectory of PDHG for solving two toy LP instances.}
	\label{fig:pdhg-dynamics}
\end{figure}

At the end of this section, we discuss the geometry of PDHG iterates. Figure \ref{fig:pdhg-dynamics} plots the trajectory of PDHG iterates for two toy LP instances. Essentially, the dynamics of PDHG iterates can be partitioned into phases, where consecutive iterates with fixed active variables define each phase. For standard-form LP~\eqref{eq:primal}, active variables refer to non-zero primal variables.
To illustrate, the red point in Figure \ref{fig:pdhg-dynamics} (a) is the optimal solution to which the PDHG iterates are converging. Points with same color represents iterations within the same phase, i.e,, their active sets are unchanged: purple for Phase 1, green for Phase 2, yellow for Phase 3 and blue for Phase 4. Within each phase, the trajectory of PDHG follows a closed-form linear dynamical system that behaves like a spiral ray. More specifically, the dynamical system of PDHG within each phase has close-form solutions as:
\begin{equation*}
    z^{(k)}-z_v=W^k(z^{(0)}-z_v)+kv
\end{equation*}
where $z=(x,y)$ is the primal-dual solution pair, {$k$ is the number of steps in the phase}, $z_v$ is the spiral center (a fixed point of the iteration process), $W$ is a transformation matrix with some eigenvalues having modulus less than 1, leading to convergence in the rotation component, and $v$ is the ray direction, along which PDHG progresses towards optimality. The exact formula of the spiral center $z_v$, ray direction $v$ and transformation matrix $W$ has close-form solutions, as specified in~\cite{liu2024new}. Figure \ref{fig:pdhg-dynamics} (b) plots an LP instance with a long phase to demonstrate this spiral ray behavior. Indeed, every one of the four phases for the example in Figure \ref{fig:pdhg-dynamics} (a) corresponds to a short period of such a spiral ray behavior.

Particularly, this spiral ray dynamic consists of two orthogonal components~\cite{liu2024new}: 
\begin{itemize} 
\item Spiral-in rotation: The iterates spiral in towards a central point;
\item Forward Movement: The iterates advance along a specific ray direction.
\end{itemize}

It turns out that the spiral-in behavior improves primal and dual feasibility, and the forward movement direction improves optimality gap (see \cite{liu2024new} for a detailed explanation).

The phase transition occurs when the active set of the iterates changes, i.e., when an active variable hits the variable bound or when an inactive variable becomes active (see Figure \ref{fig:pdhg-dynamics} (a) for an example where four phases appear in the dynamic). This reveals a characteristic of the dynamics in the iterative updates of PDHG: it keeps following a spiral ray direction until the active set changes, which leads to another spiral ray direction. Such spiral ray behavior is common for first-order primal-dual algorithms to solve LP, which distinguishes them from classical LP methods such as the simplex and interior point methods.

\section{PDLP: a First-Order Primal-Dual Method for LP}\label{sec:pdlp}
PDLP is a recently developed, general-purpose solver for LP, with several implementation variants. In contrast to most classical LP solvers, which are based on the simplex or interior-point methods, PDLP is built on the primal-dual hybrid gradient (PDHG) method, a first-order primal-dual algorithm, augmented with a range of practical enhancements to improve performance. Although FOMs are traditionally viewed as incapable of delivering high-accuracy solutions and thus unsuitable for general-purpose LP solving, PDLP has shown strong empirical performance. In particular, its GPU-based variants show promising results, positioning PDLP as a potential “new horse in the race” among state-of-the-art LP solvers.

Since the introduction of PDLP, significant research progress has been made along this line. This section begins with an overview of various implementations and variants of PDLP in Section \ref{sec:variant}. The base algorithm and its practical enhancements are then detailed in Sections \ref{sec:base} and \ref{sec:heuristic}, respectively. Comprehensive numerical studies are presented in Section \ref{sec:numerical}, followed by applications in Section \ref{sec:application} and a discussion in Section \ref{sec:discuss}.

\subsection{Implementations and Variants of PDLP}\label{sec:variant}
In 2021, the first version of PDLP was introduced as a CPU-based Julia package \href{https://github.com/google-research/FirstOrderLp.jl}{FirstOrderLp.jl}~\cite{applegate2021practical} for research purposes. Subsequently, a multi-threading CPU-based C++ implementation of \href{https://github.com/google/or-tools/tree/stable/ortools/pdlp}{PDLP} was open-sourced through \href{https://github.com/google/or-tools/tree/stable}{Google OR-Tools} in 2022~\cite{applegate2025pdlp}. A GPU-optimized variant of PDLP, called \href{https://github.com/jinwen-yang/cuPDLP.jl}{cuPDLP.jl}~\cite{lu2023cupdlp} was released in November 2023, which demonstrates strong numerical performance by taking advantages of GPUs, and triggered a rapid adaption of this methodology in optimization solver industry. In December 2023, \href{https://www.copt.de/}{COPT} published an open-sourced GPU-based C re-implementation of cuPDLP.jl named \href{https://github.com/COPT-Public/cuPDLP-C}{cuPDLP-C}~\cite{lu2023cupdlpc}. In Feburary 2024, \href{https://www.copt.de/}{COPT} integrated cuPDLP-C to its commercial side. In March 2024, \href{https://github.com/ERGO-Code/HiGHS}{HiGHS} incorporated cuPDLP-C, and \href{https://www.nvidia.com/en-us/ai-data-science/products/cuopt/}{NVIDIA cuOpt} suite introduced their \href{https://developer.nvidia.com/blog/accelerate-large-linear-programming-problems-with-nvidia-cuopt/}{GPU implementation} of PDLP. In April 2024,  \href{https://www.fico.com/en/products/fico-xpress-optimization}{FICO Xpress} integrated PDLP into the solver, and \href{https://www.gurobi.com/}{Gurobi} announced to integrate PDLP in a future version in October 2024. In December 2024, a new implementation in JAX, called \href{https://github.com/MIT-Lu-Lab/MPAX}{MPAX}~\cite{lu2024mpax}, was introduced for deep learning applications, which supports auto-differentiation, batch solving and multiple GPUs. In March 2025, NVIDIA announced to  open-source their PDLP implementation.

These implementations, while based on the PDHG algorithm for LP, incorporate different enhancements that diverge from the original FirstOrderLp.jl implementation. For instance, Google OR-Tools employs a different standard form of LP compared to FirstOrderLp.jl, cuPDLP.jl introduces an alternative restarting scheme designed to optimize performance on GPUs, while MPAX implements a reflected Halpern variant~\cite{lu2024restarted} of the algorithm, etc. The details of industrial-grade implementations, however, are often proprietary, making it unclear which specific algorithmic variations they utilize. For the sake of clarity and alignment with GPU-based implementations, the rest of this section focuses on the algorithms, enhancements, and results presented in FirstOrderLp.jl,  cuPDLP.jl and MPAX, while acknowledging the details in other implementations.

\subsection{Base Algorithm}\label{sec:base}
Different from the standard form LP \eqref{eq:primal}, PDLP takes the following general LP form as input 
\begin{equation}\label{eq:lp-general}
    \begin{aligned}[c]
    \min_{x}~~ &~ c^\top x \\
\text{s.t.}~~ &~  Ax = b\ , \  \ Gx \geq h \\
&~ l \leq x \leq u\  ,
    \end{aligned}
\end{equation} 
where $G \in \mathbb R^{m_1\times n}$, $A \in \mathbb R^{m_2 \times n}$, $c \in \mathbb R^{n}$, $h \in \mathbb R^{m_1}$, $b \in \mathbb R^{m_2}$, $l \in (\mathbb R \cup \{ -\infty \})^{n}$, $u \in (\mathbb R \cup \{ \infty \})^{n}$. While \eqref{eq:lp-general} can, in principle, be converted into \eqref{eq:primal} by introducing auxiliary variables, such transformations not only increase the problem's size but may also degrade algorithmic convergence.
This general form is utilized in FirstOrderLp.jl, cuPDLP.jl, cuPDLP-C, and MPAX. However, the Google OR-Tools implementation employs a different standard LP formulation.

PDLP solves this LP by addressing its primal-dual form:
\begin{flalign}\label{eq:primal-dual}
\min_{x \in X} \max_{y \in Y}\ L(x,y) := c^\top x - y^\top K x + q^\top y \ ,
\end{flalign}
where $K^\top = \begin{pmatrix} G^\top, A^\top \end{pmatrix}$ and $q^\top := \begin{pmatrix}
h^\top,
b^\top
\end{pmatrix}$, $X := \{x \in \mathbb R^n : l \leq x \leq u \}$, and $Y := \{y \in \mathbb R^{m_1+m_2} : y_{1:m_1} \geq 0\}.$ By duality theory of convex optimization, we know that a saddle point solution $(x^*,y^*)$ to \eqref{eq:primal-dual} recovers an optimal primal-dual solution pair to \eqref{eq:lp-general}.

PDLP is an enhanced Primal-Dual Hybrid Gradient (PDHG) method for LP. Let $z=(x,y)$ represent the primal-dual pair, and let $z^{k+1}=\mathrm{PDHG}(z^k)$ denote a single PDHG iteration, defined by the following update rule:
\begin{equation}\label{eq:pdhg-general} 
    \begin{aligned}
        & x^{k+1}\leftarrow \text{proj}_{X}(x^k-\eta (c-K^\top y^k)) \\ & y^{k+1}\leftarrow \text{proj}_{Y}(y^k+\sigma (q-K(2x^{k+1}-x^k)))\ ,
    \end{aligned}
\end{equation}
where $\eta$ is the primal step-size and $\sigma$ is the dual step-size. The primal and the dual step-sizes are reparameterized in PDLP as
\begin{equation*}
    \tau = \eta/\omega,\; \sigma=\eta\omega\quad \text{with}\;\  \eta,\omega>0\ ,
\end{equation*}
where $\eta$ (called step-size) controls the scale of the step-sizes, and $\omega$ (called primal weight) balances the primal and the dual progress. 

Notice that the projection step onto $X$ and $Y$ is straightforward as $X$ and $Y$ are simple box constraints. An attractive feature of PDHG on LP is factorization-free, with its primary computational bottleneck being matrix-vector multiplications, $Kx$ and $K^\top y$.

\subsection{Algorithmic Enhancements}\label{sec:heuristic}
The numerical performance of vanilla PDHG~\eqref{eq:pdhg-general} for LP is insufficient to serve as the backbone of a modern solver. For example, vanilla PDHG can solve only 113 instances to a relative accuracy of $10^{-4}$ and 48 instances to 
$10^{-8}$  among the 383 MIPLIB benchmark instances within 100,000 iterations~\cite{applegate2021practical}. In contrast, the enhanced PDHG implementation in FirstOrderLp.jl significantly improves performance, solving 371 instances to $10^{-4}$ relative accuracy and 334 instances to $10^{-8}$ under identical conditions~\cite{applegate2021practical}. This demonstrates the importance of these enhancements, which are essential to enable PDHG to function as a mature and robust general-purpose solver. Building upon the initial enhancements introduced in FirstOrderLp.jl, further modifications have been proposed and implemented in Google OR-Tools implementation, cuPDLP.jl and MPAX. Below, we outline the key enhancements and their high-level intuitions incorporated into various PDLP variants. For a detailed explanation, we refer to the readers corresponding literature and codebase mentioned below.

\begin{itemize}
    \item {\bf Preconditioning.} Unlike second-order methods, the numerical performance of first-order methods, such as PDHG and PDLP, is highly influenced by the conditioning of the underlying problem. To address the potential ill-conditioning of the LP instance, PDLP incorporates a diagonal preconditioner to improve the condition number of the original problem. This involves rescaling the constraint matrix $K=(G,A)$ to $\tilde K=(\tilde G,\tilde A)=D_1KD_2$ where $D_1$ and $D_2$ are positive diagonal matrices. This rescaling ensures that the resulting matrix $\tilde{K}$ is "well-balanced," enhancing numerical stability.

    As a result, the preconditioning step produces a modified LP instance, where $A,G,c,b,h,u$ and $l$ in \eqref{eq:lp-general} are transformed into $\tilde G,\tilde A,\hat x=D_2^{-1}x,\tilde c=D_2c,(\tilde b,\tilde h)=D_1(b,h), \tilde u=D_2^{-1}u$ and $\tilde l=D_2^{-1}l$, respectively. In its default configuration, PDLP employs a combination of Ruiz rescaling~\cite{ruiz2001scaling} and the diagonal preconditioning technique proposed by Pock and Chambolle~\cite{pock2011diagonal}. A more detailed description and comparison among different preconditioning schemes can be found in \cite{applegate2021practical}.

    \item {\bf Average, Halpern and reflection.} In PDHG algorithm, it is known that the average iterate, i.e.,
    $$\bar{z}^k=\frac{1}{k} \sum_{i=1}^k z^i$$
    has a faster sublinear convergence rate than the last iteration of PDHG (see Section \ref{sec:unified} for more details).  Implementations such as FirstOrderLp.jl, OR-Tools, and cuPDLP.jl incorporate the average iterate into their update rules, using a weighted average with step-size as the weight, sometimes in combination with the last iterate. This approach is described in detail in \cite{applegate2021practical}.
    
    The Halpern scheme and reflection are recent enhancements proposed and studied in \cite{lu2024restarted, chen2024hpr}, and implemented in MPAX. Halpern method~\cite{halpern1967fixed,lieder2021convergence,diakonikolas2020halpern,park2022exact,kim2021accelerated} is a scheme originally designed to accelerate general operator splitting methods. Halpern PDHG anchors to the initial solution, and takes a weighted average between the PDHG step at the current iterate and the initial point. When applying to LP, the update rule for Halpern PDHG is given by: 
    \begin{equation}\label{eq:hpdhg}
        z^{k+1}%=\text{H-PDHG}(z^k;z^0):
        =\frac{k+1}{k+2}\mathrm{PDHG}(z^k)+\frac{1}{k+2}z^0 \ .
    \end{equation}
    where $\mathrm{PDHG}(z^k)$ represents a single PDHG iteration at current solution $z^k$, as defined in \eqref{eq:pdhg-general}. Halpern PDHG bears some resemblance to the use of averaged PDHG iterates and achieves the same improved convergence guarantees as the ergodic rate for PDHG (see \cite{lu2024restarted} for a detailed comparison), in contrast to the convergence guarantees for its last iteration. A key numerical advantage of Halpern iterates is that they eliminate the need to track and update the average iterate. Theoretically, Halpern iterates offer numerous benefits, including an accelerated linear convergence rate for infeasibility detection and enhanced two-stage convergence behavior—features that are difficult to achieve with averaged iterates (see \cite{lu2024restarted} for further details).

    Reflection~\cite{bauschke2019convex,ryu2022large} further enhances the Halpern PDHG method. The extension is straightforward: instead of applying the Halpern iteration directly to vanilla PDHG operator $\mathrm{PDHG}(\cdot)$, it is applied to its reflection $2\mathrm{PDHG}(\cdot)-I$. The update rule for Reflected Halpern PDHG is:
    \begin{equation}\label{eq:rhpdhg}
        z^{k+1}%=\text{H-PDHG}(z^k;z^0):
        =\frac{k+1}{k+2}\bigg(2\mathrm{PDHG}(z^k)-z^k\bigg)+\frac{1}{k+2}z^0 \ .
    \end{equation}
    
    This use of reflection effectively takes a longer step compared to vanilla Halpern PDHG, leading to faster convergence both theoretically and empirically. In fact, theoretical analysis shows that the reflection approach provides an improvement by a factor of 2 (see Section \ref{sec:unified} for details), which is corroborated by superior empirical performance.
    
    \item {\bf Restart.} Restarting is a key enhancement in PDLP for achieving high-accuracy solutions. In the average scheme, the algorithm restarts from the average iterate when a predefined condition is met. The underlying intuition is that after completing a spiral of PDHG iterates, the average of the iterates is typically close to the stationary point, making it advantageous to restart from this average solution (considering the example in Figure \ref{fig:pdhg-gda}).
    In the Halpern scheme, restarting involves resetting the anchor point to the current solution rather than the initial solution. The intuition here is that as the algorithm converges toward the optimal solution, continuing to reference a distant initial anchor point becomes counterproductive. Resetting the anchor point ensures the algorithm focuses on the vicinity of the optimal solution.
Restarting accelerates the theoretical convergence rate for both the average scheme and the Halpern scheme (see Section \ref{sec:optimal} for details).
    
    A key component of the restart mechanism is determining the appropriate time to restart. PDLP adopts an adaptive restarting strategy, where a restart candidate is evaluated at each iteration. This involves assessing various restart criteria to identify whether a constant-factor decay in a specific progress metric has occurred. If such decay is detected, a restart is triggered. Detailed discussion of this strategy can be found in~\cite{applegate2023faster,lu2023cupdlp}.
     
    Notably, different implementations of PDLP employ distinct progress metrics for determining restarts. For example, FirstOrderLp.jl and the Google OR-Tools implementation use a metric known as the normalized duality gap at the average iterates~\cite{applegate2021practical,applegate2023faster}, which is evaluated via a trust-region solver. However, this trust-region solver operates sequentially, making it unsuitable for GPUs. To address this, the GPU-based cuPDLP.jl introduces a new restart scheme leveraging the KKT error as the progress metric, ensuring compatibility with GPU architectures. Conversely, the MPAX implementation uses fixed-point residual at the current Halpern iterates, as this metric aligns naturally with the Halpern scheme~\cite{lu2024restarted}.

    \item {\bf Adaptive step-size.}  The step-size suggested by theoretical considerations, specifically $1/\|A\|_2$, often proves overly conservative in practical applications. To address this, PDLP incorporates a heuristic line search to adaptively determine a suitable step-size that satisfies the condition:
    \begin{equation}\label{eq:step-size}
            \eta\leq \frac{\|z^{k+1}-z^k\|_{\omega}^2}{2(y^{k+1}-y^k)^\top K(x^{k+1}-x^k)}\ ,
    \end{equation}
    where $\|z\|_{\omega}:=\sqrt{\omega\|x\|_2^2+\frac{\|y\|_2^2}{\omega}}$ and $\omega$ represents the current primal weight. Additional details regarding this adaptive step-size rule can be found in~\cite{applegate2021practical}. The step-size rule \eqref{eq:step-size} is inspired by keeping the matrix $P$ (defined later in \eqref{eq:generic}) along the two consecutive iterates direction positive semidefinite. Empirical evidence from numerical experiments in~\cite{applegate2021practical} demonstrates the consistent efficacy of this approach in improving practical performance compared to other established step-size rules, despite breaking the theoretical convergence guarantees.
    
    \item {\bf Primal weight.} Adjusting the primal weight $\omega$ is intended to balance the primal and dual spaces through a heuristic approach. This update occurs infrequently, typically during restart events. In PDLP, the primal weight is updated at the beginning of each new epoch. The rationale is to determine the primal weight $\omega^n$ such that it approximately equalizes the distance to optimality in both the primal and dual domains, i.e., $\|(x^{n,k}-x^*,0)\|_{\omega^n}\approx \|(0,y^{n,k}-y^*)\|_{\omega^n}$. 
    To further stabilize the updates and reduce oscillations, PDLP applies exponential smoothing with a parameter $\theta \in [0,1]$. Additional details about this adaptive weight adjustment strategy can be found in~\cite{applegate2021practical}. However, our numerical experiences indicate that primal weight is a crucial hyper-parameter for the performance of the algorithm, and the proposed adaptive primal weight rule is not always reliable as expected. Per-instance tuning of the primal weight has the potential to yield significant performance enhancements.

    \item {\bf Feasibility polishing.} In certain LP applications, approximately optimal solutions that are feasible are often sufficient. Feasibility polishing is a recent heuristic designed to enable first-order methods to find solutions with minimal feasibility violations and moderate duality gaps~\cite{applegate2025pdlp}. This approach involves applying PDLP to solve the primal/dual feasibility problems, initialized from a near-optimal solution. The key insights underlying feasibility polishing are: (i) PDLP converges faster for feasibility problems than for optimality problems, and (ii) when provided with a good initial solution, PDLP reliably converges to a nearby optimal solution.

    \item {\bf Infeasibility detection.}
    The ability to detect infeasibility/unboundedness in LP instances is a critical feature of LP solvers. In practice, PDLP employs periodic checks to determine whether the difference between iterates, $z^{k+1} - z^k$, or the normalized iterates, $\frac{1}{k}(z^k - z^0)$, can provide infeasibility certificates. For a more detailed discussion on the methods and efficacy of infeasibility detection in PDHG, we refer readers to Section \ref{sec:infeas}.
   
\end{itemize}

\subsection{Numerical Performances}\label{sec:numerical}

In this section, we summarize the numerical experiments of cuPDLP.jl and $\mathrm{r^2HPDHG}$ that were presented in \cite{lu2023cupdlp} and \cite{lu2024restarted} to illustrate the numerical performance of GPU-based LP algorithms compared to CPU-based solvers. In particular, we present comparisons of GPU-based cuPDLP.jl and $\mathrm{r^2HPDHG}$ with the CPU implementations of PDLP~\cite{applegate2021practical} and three methods implemented in the commercial solver Gurobi, i.e., primal simplex method, dual simplex method, and interior-point method. Section \ref{sec:setup} describes specific experiment setup. The numerical results compared with Gurobi and CPU-based PDLP are presented in Section \ref{sec:result-gurobi} and Section \ref{sec:result-pdlp} respectively.

\subsubsection{Experiment Setup}\label{sec:setup}

{\bf Benchmark datasets.} The \texttt{MIP Relaxations} benchmark dataset is curated from the MIPLIB 2017 collection~\cite{gleixner2021miplib}, a well-established repository of mixed-integer linear programming problems. For benchmarking purposes, the root-node LP relaxation of instances from MIPLIB 2017 is extracted to form the LP benchmark set. A total of 383 instances are selected based on specific criteria, following a similar selection process used in the experiments of CPU-based PDLP~\cite{applegate2021practical}:
\begin{itemize}
    \item Not tagged as numerically unstable
    \item Not tagged as infeasible
    \item Not tagged as having indicator constraints
    \item Finite optimal objective (if known)
    \item The constraint matrix has a number of nonzeros greater than 100,000
    \item Zero is not an optimal solution to the LP relaxation.
\end{itemize}
\texttt{MIP Relaxations} is further split into three classes based on the number of nonzeros (nnz) in the constraint matrix\footnote{{In the comparison, we only include LP instances with at least 100K nonzeros, as for small problems factorization is relatively cheap and there is no need to use FOMs in such cases.}}, as shown in Table \ref{tab:miplib-size}.
\begin{table}[ht!]
\centering
\begin{tabular}{cccc}
\hline
                    & \textbf{Small}                   & \textbf{Medium}                     & \textbf{Large}                   \\ \hline
\textbf{Number of nonzeros}  & 100K -  1M & 1M - 10M & \textgreater 10M \\
\textbf{Number of instances} & 269                     & 94                         & 20                      \\ \hline
\end{tabular}
\caption{Scales of instances in \texttt{MIP Relaxations}.}
\label{tab:miplib-size}
\end{table}

{\bf Software.} The \href{https://github.com/jinwen-yang/cuPDLP.jl}{cuPDLP.jl} and $\mathrm{r^2HPDHG}$ solvers are implemented as Julia~\cite{bezanson2017julia} modules, utilizing \href{https://github.com/JuliaGPU/CUDA.jl}{CUDA.jl}~\cite{besard2018effective} as the interface for executing computations on NVIDIA CUDA GPUs. To evaluate their performances, cuPDLP.jl and $\mathrm{r^2HPDHG}$ are compared against five LP solvers: the Julia implementation of PDLP (\href{https://github.com/google-research/FirstOrderLp.jl}{FirstOrderLp.jl}), the C++ implementation of PDLP with both single-threaded and multi-threaded configurations (available in \href{https://github.com/google/or-tools}{Google OR-Tools}), and Gurobi’s implementations of the primal simplex, dual simplex, and barrier methods. Crossover is disabled in the Gurobi experiments, and all Gurobi runs are conducted using 16 threads. To ensure a fair comparison, the running time of cuPDLP.jl, $\mathrm{r^2HPDHG}$ and FirstOrderLp.jl is measured after pre-compilation in Julia.

{\bf Computing environment.} The experiments are conducted on an NVIDIA H100-PCIe-80GB GPU with CUDA 12.3 for running cuPDLP.jl and $\mathrm{r^2HPDHG}$, and an Intel Xeon Gold 6248R CPU (3.00 GHz) with 160GB RAM and 16 threads for executing CPU-based solvers. The software environment includes Julia 1.9.2 and Gurobi 11.0, the latter of which was released in November 2023. Table \ref{tab:flop} presents a comparison of the theoretical peak double-precision FLOPS (floating-point operations per second) for the CPU and GPU used in the experiments\footnote{{While we utilized the best CPUs available to us in these experiments, we acknowledge that CPUs are substantially less expensive than NVIDIA H100 GPUs. It is also possible that a higher-end CPU could further improve the performance of CPU-based solvers.}}.

\begin{table}[ht!]
\centering
{\small
\begin{tabular}{ccc}
\hline
                                      & \textbf{CPU (16 threads)}      & \textbf{GPU}                \\ \hline
\textbf{Processor}                    & Intel Xeon Gold 6248R CPU 3.00GHz\tablefootnote{See \href{https://ark.intel.com/content/www/us/en/ark/products/199351/intel-xeon-gold-6248r-processor-35-75m-cache-3-00-ghz.html}{processor specifications}.} & NVIDIA H100-PCIe\tablefootnote{H100 has 114 SMs and 7296 FP64 cores. See \href{https://resources.nvidia.com/en-us-tensor-core/nvidia-tensor-core-gpu-datasheet}{H100 datasheet} and \href{https://resources.nvidia.com/en-us-tensor-core/gtc22-whitepaper-hopper}{H100 whitepaper} for more a detailed description of H100 GPU. } \\
\textbf{Theoretical peak (FP64)} & 256 GFLOPS                         & 26 TFLOPS \\ 
\textbf{Maximum memory bandwidth} & 137.48 GB/sec  & 2 TB/sec                 \\ \hline
\end{tabular}
}
\caption{Comparison of CPU and GPU specifications.}
\label{tab:flop}
\end{table}

{\bf Initialization.} PDLP, cuPDLP.jl and $\mathrm{r^2HPDHG}$ use all-zero vectors as the initial starting points.

{\bf Optimality termination criteria.} PDLP, cuPDLP.jl and $\mathrm{r^2HPDHG}$ terminate when the relative KKT error is no greater than the termination tolerance $\epsilon\in(0,\infty)$:
\begin{equation*}
    \begin{aligned}
        |q^\top y+l^\top \lambda^+-u^\top \lambda^--c^\top x|&\leq \epsilon(1+|q^\top y+l^\top\lambda^+-u^\top \lambda^-|+|c^\top x|)\\
        \left\|\begin{pmatrix}
            Ax-b \\ [h-Gx]^+
        \end{pmatrix}\right\|_2&\leq \epsilon(1+\|q\|_2)\\
        \|c-K^\top y-\lambda\|_2&\leq \epsilon(1+\|c\|_2) \ .
    \end{aligned}
\end{equation*}
The termination criteria are evaluated based on the original LP instance rather than the preconditioned version, ensuring that the stopping conditions remain unaffected by the preconditioning process.
$\epsilon=10^{-4}$ is used for moderately accurate solutions and $\epsilon=10^{-8}$ for high-quality solutions. Parameters \texttt{FeasibilityTol}, \texttt{OptimalityTol} and \texttt{BarConvTol} (for barrier methods) of Gurobi are also set to $10^{-4}$ and $10^{-8}$ tolerances\footnote{It is important to note that Gurobi barrier, Gurobi simplex, and PDLP employ different termination criteria. A key distinction is that PDLP uses a relative tolerance, while Gurobi applies an absolute tolerance in its termination test, resulting in a stricter stopping condition. Despite this difference, empirical observations indicate that PDLP with high-accuracy settings ($10^{-8}$ relative tolerance) consistently produces high-quality solutions.}.

{\bf Time limit.} A time limit of 3600 seconds is imposed on instances with small-sized and medium-sized instances. A time limit of 18000 seconds is imposed for large instances.

{\bf Shifted geometric mean.} Shifted geometric mean of solve time is reported to measure the performance of solvers on a certain collection of problems. More precisely, shifted geometric mean is defined as $\left(\prod_{i=1}^n (t_i+\Delta)\right)^{1/n}-\Delta$ where $t_i$ is the solve time for the $i$-th instance. The shift is set to $\Delta=10$ and denoted as SGM10. If the instance is unsolved, the solve time is always set to the corresponding time limit. 

{\bf Presolve.} 
The comparison with Gurobi is conducted on two sets of instances, namely, the original instances without any presolve step, as well as the instances after Gurobi presolve step.

\subsubsection{Comparison with Gurobi}\label{sec:result-gurobi}

Table \ref{tab:miplib-1e-4-no-presolve} - \ref{tab:miplib-1e-8-with-presolve} present a comparison between cuPDLP.jl, $\mathrm{r^2HDPHG}$ and the commercial LP solver Gurobi. Particularly, Table \ref{tab:miplib-1e-4-no-presolve} and Table \ref{tab:miplib-1e-4-with-presolve} present moderate accuracy results (i.e., $10^{-4}$ relative KKT error), while Table \ref{tab:miplib-1e-8-no-presolve} and Table \ref{tab:miplib-1e-8-with-presolve} present the high accuracy results (i.e., $10^{-8}$ relative KKT error); Table \ref{tab:miplib-1e-4-no-presolve} and Table \ref{tab:miplib-1e-8-no-presolve} present the results on solving the original problems, while Table \ref{tab:miplib-1e-4-with-presolve} and Table \ref{tab:miplib-1e-8-with-presolve} present the results on LP instances after Gurobi presolve. 

Presolve refers to the process of simplifying an optimization problem before applying the main solver, typically by removing redundancies, tightening variable bounds, and identifying infeasibilities or fixed variables. There has been an ongoing debate about whether comparing algorithm performance with and without presolve is more appropriate. On one hand, presolve can substantially reduce problem size and mitigate ill-conditioning, benefiting all classes of algorithms. On the other hand, presolve often involves considerable engineering effort and is generally treated as a black-box process. In commercial solvers, presolve routines are primarily designed to enhance the performance of interior-point and simplex methods, and the majority of performance gains often stem from these traditional algorithms. Designing effective presolve schemes tailored for FOMs remains an open research question.
From an algorithmic comparison standpoint, disabling presolve may lead to fairer evaluations by isolating the core algorithmic behavior. However, from a standpoint of solver comparison, presolve is critical and can lead to substantial runtime improvements.  To provide a comprehensive view, we here report experimental results both with and without presolve enabled.

The tables yield several noteworthy observations:
\begin{itemize}
    \item With Gurobi presolve, cuPDLP.jl can solve 99.5\% instances to medium accuracy and 97.4\% instances to high accuracy within the time limit, demonstrating its reliability for solving real-world LP. Furthermore, $\mathrm{r^2HPDHG}$ can solve 99.0\% instances to medium accuracy and 96.6\% instances to high accuracy.
    \item In the case of moderate accuracy ($\epsilon=10^{-4}$), cuPDLP.jl and $\mathrm{r^2HPDHG}$ exhibit strong performance in terms of solved count and solve time, regardless of whether to use presolve. 
    For medium-sized and large-sized instances without presolve, $\mathrm{r^2HPDHG}$ establishes an advantage over Gurobi, achieving a 5.7x speed-up on medium problems with 5 more instances solved and a 4.6x speed-up on large instances with one additional solved instance, respectively. With Gurobi presolve, cuPDLP.jl is able to solve 381 out of 383 instances and the solve time is comparable to the best of the three Gurobi methods (i.e., barrier methods).
    \item In the case of high accuracy ($\epsilon=10^{-8}$), cuPDLP.jl and $\mathrm{r^2HPDHG}$ have strong performance, compared to Gurobi primal and dual simplex method, though it is inferior to the Gurobi barrier method. 
    \item Gurobi presolve can improve the performance of all Gurobi methods as well as the performance of cuPDLP.jl/$\mathrm{r^2HPDHG}$. The effect of presolve is more significant for Gurobi methods. This is expected because Gurobi presolve is designed to speed up Gurobi methods.
\end{itemize}
To summarize, these observations affirm that both cuPDLP.jl and $\mathrm{r^2HPDHG}$ attain strong performance in \texttt{MIP Relaxations} benchmark dataset. This demonstrates that a first-order-method-based LP solver on GPU can be a ``new horse in the race'' with strong implementations of simplex and barrier methods, even in obtaining high-accuracy solutions.

\begin{table}[h!]
\centering
{\small
\begin{tabular}{ccccccccc}
\hline
\multirow{2}{*}{}                                   & \multicolumn{2}{c}{\begin{tabular}[c]{@{}c@{}}\textbf{Small (269)} \\ (1-hour limit)\end{tabular}} & \multicolumn{2}{c}{\begin{tabular}[c]{@{}c@{}}\textbf{Medium (94)}\\ (1-hour limit)\end{tabular}}    & \multicolumn{2}{c}{\begin{tabular}[c]{@{}c@{}}\textbf{\textbf{Large (20)}}\\ (5-hour limit)\end{tabular}}  & \multicolumn{2}{c}{\textbf{Total (383)}}       \\
                                                    & \textbf{Count} & \textbf{Time} & \textbf{Count} & \textbf{Time} & \textbf{Count} & \textbf{Time} & \textbf{Count} & \textbf{Time} \\ \hline
\multicolumn{1}{c}{\textbf{cuPDLP.jl}}     & 266                   & 8.61               & 92                    & 14.80               & \cellcolor{LightCyan}19                    & 111.19 &377 &12.02             \\
\multicolumn{1}{c}{{{$\mathrm{\bf r^2HPDHG}$}}}     & {267}                   & {6.61}               & \cellcolor{LightCyan}{93}                    & \cellcolor{LightCyan}{7.84}              & \cellcolor{LightCyan}{19}                    & \cellcolor{LightCyan}{90.81} & \cellcolor{LightCyan}379 & \cellcolor{LightCyan}8.58\\
\multicolumn{1}{c}{\textbf{Primal simplex (Gurobi)}} & \cellcolor{LightCyan}268                   & 12.56                & 69                    & 188.81              & 11                    & 3145.49   &348 &39.81           \\
\multicolumn{1}{c}{\textbf{Dual simplex (Gurobi)}}   & \cellcolor{LightCyan}268                   & 8.75                & 84                    & 66.67               & 15                    & 591.63    &367 &21.75           \\
\multicolumn{1}{c}{\textbf{Barrier (Gurobi)}}        & \cellcolor{LightCyan}268                   & \cellcolor{LightCyan}5.30                & 88                    & 45.01               & 18                    & 415.78    &374 &14.92           \\ \hline
\end{tabular}
}
\caption{Solve time in seconds and SGM10 of different solvers on instances of \texttt{MIP Relaxations} with tolerance $10^{-4}$: cuPDLP.jl/$\mathrm{r^2HPDHG}$ versus Gurobi without presolve.}
\label{tab:miplib-1e-4-no-presolve}
\end{table}

\begin{table}[h!]
\centering
{\small
\begin{tabular}{ccccccccc}
\hline
\multirow{2}{*}{}                                   & \multicolumn{2}{c}{\begin{tabular}[c]{@{}c@{}}\textbf{Small (269)} \\ (1-hour limit)\end{tabular}} & \multicolumn{2}{c}{\begin{tabular}[c]{@{}c@{}}\textbf{Medium (94)}\\ (1-hour limit)\end{tabular}}    & \multicolumn{2}{c}{\begin{tabular}[c]{@{}c@{}}\textbf{\textbf{Large (20)}}\\ (5-hour limit)\end{tabular}} & \multicolumn{2}{c}{\textbf{Total (383)}}          \\
                                                    & \textbf{Count} & \textbf{Time}  & \textbf{Count} & \textbf{Time}  & \textbf{Count} & \textbf{Time} & \textbf{Count} & \textbf{Time}  \\ \hline
\multicolumn{1}{c}{\textbf{cuPDLP.jl}}     & 261                   & 23.47                & 86                    & 40.69                & 16                    & 421.40  &363 &32.35             \\
\multicolumn{1}{c}{{{$\mathrm{\bf r^2HPDHG}$}}}     & {260}                   & {19.13}               & {87}                    & \cellcolor{LightCyan}{28.35}              & {16}                    & \cellcolor{LightCyan}{229.47} & 363 & 24.79  \\
\multicolumn{1}{c}{\textbf{Primal simplex (Gurobi)}} & \cellcolor{LightCyan}268                   & 12.43                 & 74                    & 157.59               & 13                    & 2180.23        &355 &36.68      \\
\multicolumn{1}{c}{\textbf{Dual simplex (Gurobi)}}   & \cellcolor{LightCyan}268                   & 8.00                 & 83                    & 59.93                & 15                    & 687.17       &366 &20.40      \\
\multicolumn{1}{c}{\textbf{Barrier (Gurobi)}}        & 267                   & \cellcolor{LightCyan}6.24                 & \cellcolor{LightCyan}88                    & 48.62                & \cellcolor{LightCyan}18                    & 438.69      &\cellcolor{LightCyan}373 &\cellcolor{LightCyan}16.46        \\ \hline
\multicolumn{1}{l}{}                                & \multicolumn{1}{l}{}  & \multicolumn{1}{l}{} & \multicolumn{1}{l}{}  & \multicolumn{1}{l}{} & \multicolumn{1}{l}{}  & \multicolumn{1}{l}{}
\end{tabular}
}
\caption{Solve time in seconds and SGM10 of different solvers on instances of \texttt{MIP Relaxations} with tolerance $10^{-8}$: cuPDLP.jl/$\mathrm{r^2HPDHG}$ versus Gurobi without presolve.\tablefootnote{Gurobi primal and dual simplex methods may perform better for obtaining high-accuracy solution than medium-accuracy solution. This is a \href{https://support.gurobi.com/hc/en-us/community/posts/21790169951249-Simplex-method-does-not-stop-at-optimal-solution}{known effect} due to the different trajectory paths when setting different tolerance levels.}}
\label{tab:miplib-1e-8-no-presolve}
\end{table}

\begin{table}[h!]
\centering
{\small
\begin{tabular}{ccccccccc}
\hline
\multirow{2}{*}{}                                   & \multicolumn{2}{c}{\begin{tabular}[c]{@{}c@{}}\textbf{Small (269)} \\ (1-hour limit)\end{tabular}} & \multicolumn{2}{c}{\begin{tabular}[c]{@{}c@{}}\textbf{Medium (94)}\\ (1-hour limit)\end{tabular}}    & \multicolumn{2}{c}{\begin{tabular}[c]{@{}c@{}}\textbf{\textbf{Large (20)}}\\ (5-hour limit)\end{tabular}}  & \multicolumn{2}{c}{\textbf{Total (383)}}       \\
                                                    & \textbf{Count} & \textbf{Time} & \textbf{Count} & \textbf{Time} & \textbf{Count} & \textbf{Time} & \textbf{Count} & \textbf{Time} \\ \hline
\multicolumn{1}{c}{\textbf{cuPDLP.jl}}     & \cellcolor{LightCyan}269                   & 5.35               & 93                    & 10.31               & 19                    & 33.93 &381 &7.37             \\
\multicolumn{1}{c}{{$\mathrm{\bf r^2HPDHG}$}}     & {267}                   & {3.95}                & \cellcolor{LightCyan}{94}                    & \cellcolor{LightCyan}{6.45}                & {19}                    & \cellcolor{LightCyan}{17.13}    & 380 & 5.04             \\
\multicolumn{1}{c}{\textbf{Primal simplex (Gurobi)}} & \cellcolor{LightCyan}269                   & 5.67                & 71                    & 121.23              & 19                    & 297.59   &359 &20.84          \\
\multicolumn{1}{c}{\textbf{Dual simplex (Gurobi)}}   & 268                   & 4.17                & 86                    & 37.56               & 19                    & 179.49    &373 &11.84           \\
\multicolumn{1}{c}{\textbf{Barrier (Gurobi)}}        & \cellcolor{LightCyan}269                   & \cellcolor{LightCyan}1.21                & \cellcolor{LightCyan}94                    & 15.32               & \cellcolor{LightCyan}20                    & 30.70    &\cellcolor{LightCyan}383 &\cellcolor{LightCyan}4.65           \\ \hline
\end{tabular}
}
\caption{Solve time in seconds and SGM10 of different solvers on instances of \texttt{MIP Relaxations} with tolerance $10^{-4}$: cuPDLP.jl/$\mathrm{r^2HPDHG}$ versus Gurobi with presolve.}
\label{tab:miplib-1e-4-with-presolve}
\end{table}

\begin{figure}[ht!]
	% \centering
    \hspace{-1cm}
	\begin{tabular}{c c c c}
		& \includegraphics[width=0.5\textwidth]{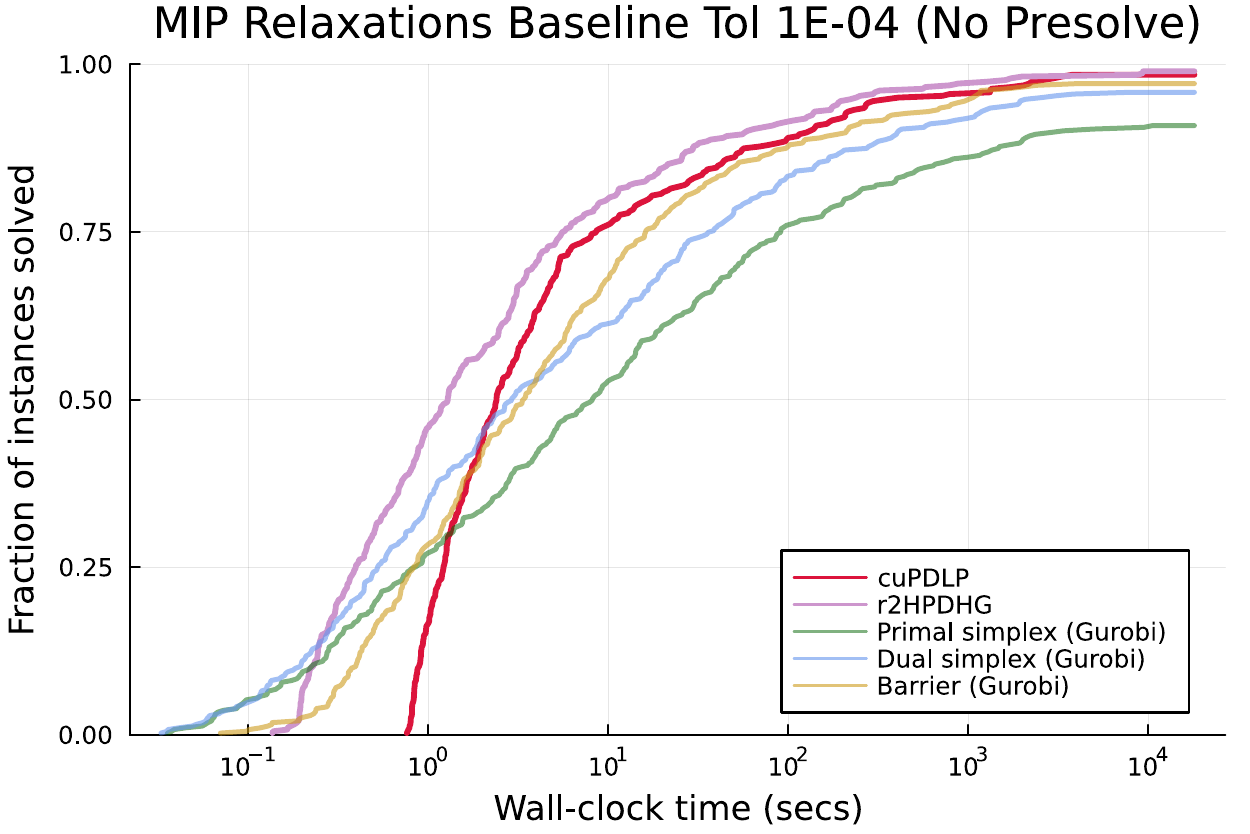}
        & \includegraphics[width=0.5\textwidth]{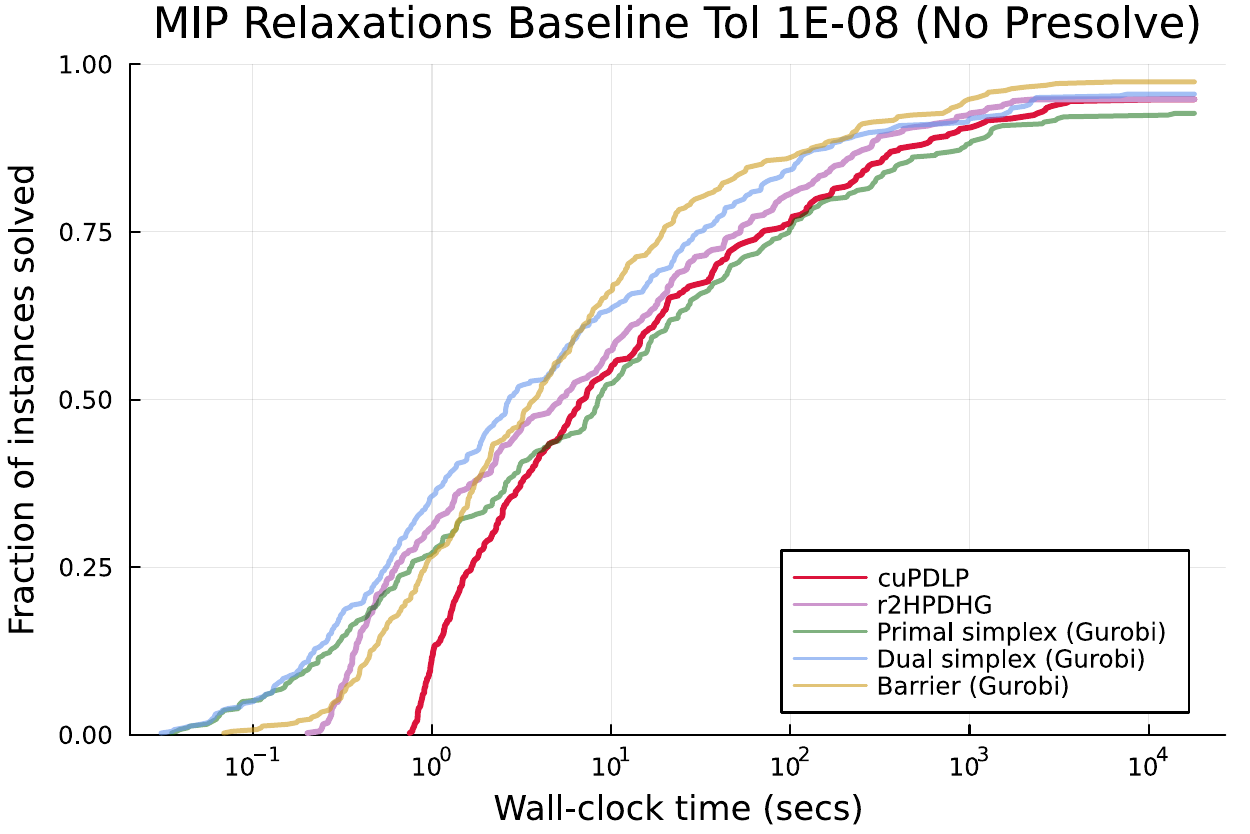}
	\end{tabular}
	\caption{Number of instances solved for \texttt{MIP Relaxations} under moderate accuracy (left) and high accuracy (right): cuPDLP.jl versus Gurobi without presolve.}
	\label{fig:performance-no-presolve}
\end{figure}

\begin{table}[h!]
\centering
{\small
\begin{tabular}{ccccccccc}
\hline
\multirow{2}{*}{}                                   & \multicolumn{2}{c}{\begin{tabular}[c]{@{}c@{}}\textbf{Small (269)} \\ (1-hour limit)\end{tabular}} & \multicolumn{2}{c}{\begin{tabular}[c]{@{}c@{}}\textbf{Medium (94)}\\ (1-hour limit)\end{tabular}}    & \multicolumn{2}{c}{\begin{tabular}[c]{@{}c@{}}\textbf{\textbf{Large (20)}}\\ (5-hour limit)\end{tabular}} & \multicolumn{2}{c}{\textbf{Total (383)}}          \\
                                                    & \textbf{Count} & \textbf{Time}  & \textbf{Count} & \textbf{Time}  & \textbf{Count} & \textbf{Time} & \textbf{Count} & \textbf{Time}  \\ \hline
\multicolumn{1}{c}{\textbf{cuPDLP.jl}}     & 264                   & 17.53                & 90                    & 30.05                & 19                    & 81.07  &373 &22.13             \\
\multicolumn{1}{c}{{$\mathrm{\bf r^2HPDHG}$}}     & {261}                   & {15.24}                & {90}                    & {21.67}                & {19}                    & {56.19}  & 370 & 18.07           \\
\multicolumn{1}{c}{\textbf{Primal simplex (Gurobi)}} & \cellcolor{LightCyan}269                   & 5.19                & 75                    & 100.03               & 18                    & 171.72        &362 &18.11      \\
\multicolumn{1}{c}{\textbf{Dual simplex (Gurobi)}}   & 268                   & 3.53                 & 89                    & 27.17               & 19                    & 121.94        &376 &9.53      \\
\multicolumn{1}{c}{\textbf{Barrier (Gurobi)}}        & \cellcolor{LightCyan}269                   & \cellcolor{LightCyan}1.34                 & \cellcolor{LightCyan}94                    & \cellcolor{LightCyan}16.85                & \cellcolor{LightCyan}20                    & \cellcolor{LightCyan}33.48       &\cellcolor{LightCyan}383 &\cellcolor{LightCyan}5.03        \\ \hline
\multicolumn{1}{l}{}                                & \multicolumn{1}{l}{}  & \multicolumn{1}{l}{} & \multicolumn{1}{l}{}  & \multicolumn{1}{l}{} & \multicolumn{1}{l}{}  & \multicolumn{1}{l}{}
\end{tabular}
}
\caption{Solve time in seconds and SGM10 of different solvers on instances of \texttt{MIP Relaxations} with tolerance $10^{-8}$: cuPDLP.jl/$\mathrm{r^2HPDHG}$ versus Gurobi with presolve.}
\label{tab:miplib-1e-8-with-presolve}
\end{table}

\begin{figure}[ht!]
	% \centering
    \hspace{-1cm}
	\begin{tabular}{c c c c}
		& \includegraphics[width=0.5\textwidth]{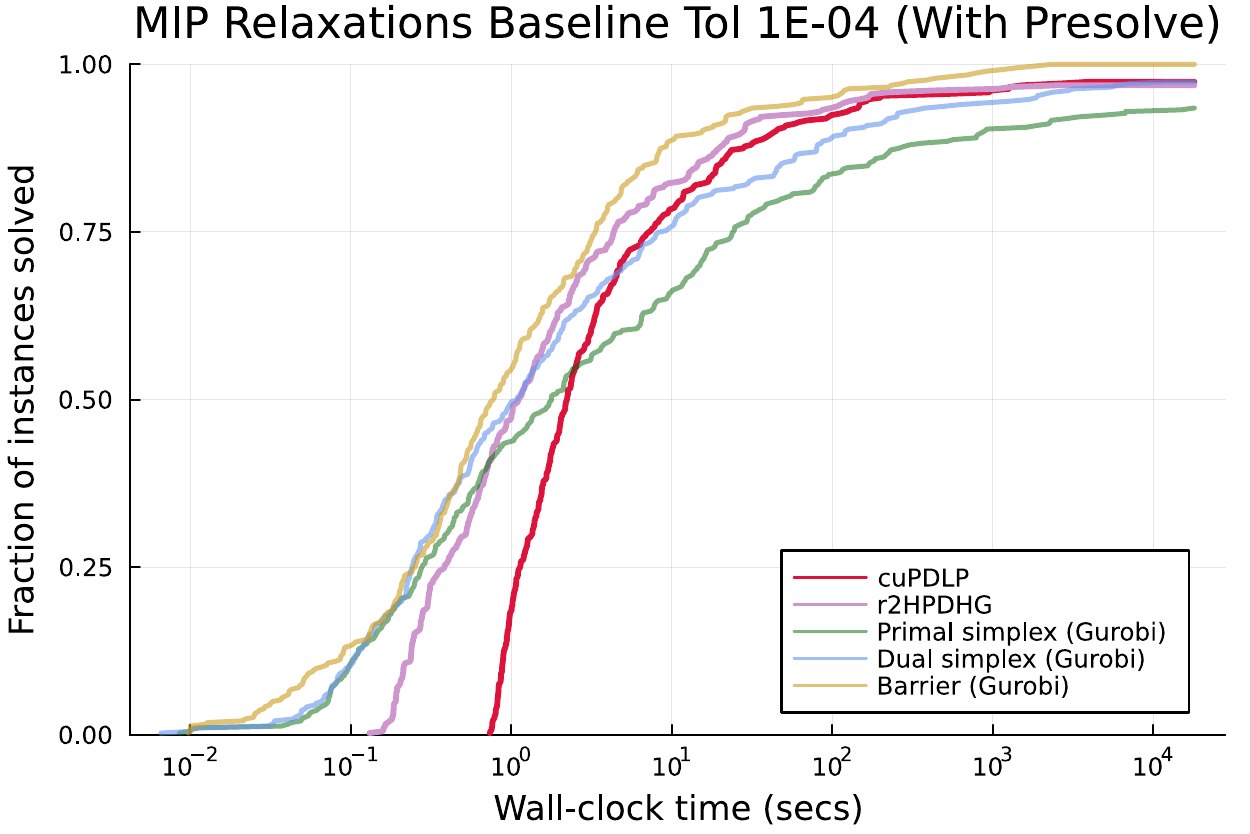}
        & \includegraphics[width=0.5\textwidth]{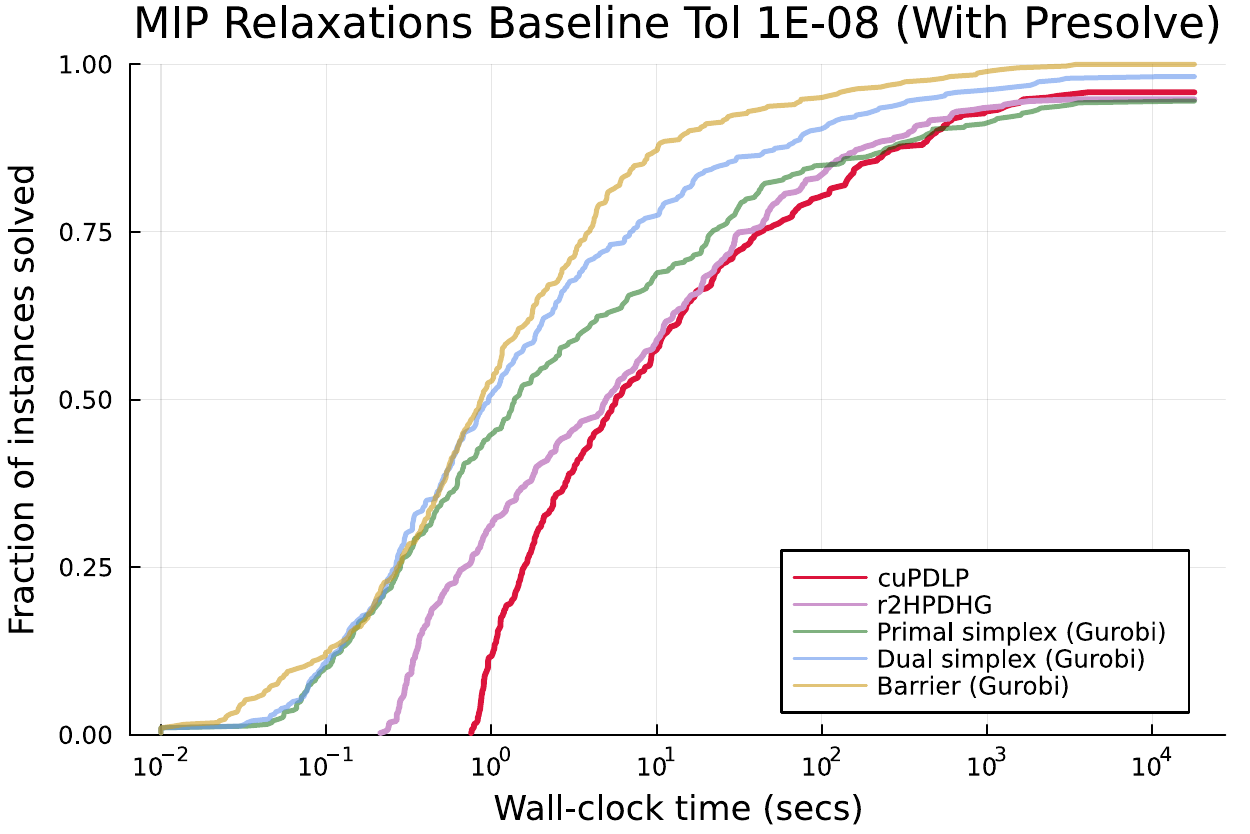}
	\end{tabular}
	\caption{Number of instances solved for \texttt{MIP Relaxations} under moderate accuracy (left) and high accuracy (right): cuPDLP.jl versus Gurobi with presolve.}
	\label{fig:performance-with-presolve}
\end{figure}

Figure \ref{fig:performance-no-presolve} and Figure \ref{fig:performance-with-presolve} show the number of solved instances of cuPDLP.jl/$\mathrm{r^2HPDHG}$ and three methods in Gurobi on \texttt{MIP Relaxations} in a given time. The y-axes display the fraction of solved instances, and the x-axes display the wall-clock time in seconds. As shown in the left panel, when seeking solutions with moderate accuracy ($\epsilon=10^{-4}$), $\mathrm{r^2HPDHG}$ has strong numerical performances, compared with Gurobi barrier. Both cuPDLP.jl and $\mathrm{r^2HPDHG}$ eventually have better performances on \texttt{MIP Relaxation} than all three methods of Gurobi without presolve. In addition, we can see the performance of cuPDLP.jl and $\mathrm{r^2HPDHG}$ for high-quality solution ($\epsilon=10^{-8}$), as shown in the right panel, is still comparable to Gurobi. An observation is that the number of instances Gurobi can solve for a given running time does not differ much for moderate and high accuracy; conversely, such difference is more apparent for cuPDLP.jl and $\mathrm{r^2HPDHG}$. This is a feature of a first-order-method-based solver. Another interesting fact is that Gurobi solves about 35\% of instances within one second, which is exactly the power of Gurobi. 

\subsubsection{Comparison with CPU-based PDLP}\label{sec:result-pdlp}
\begin{table}[ht!]
\centering
\begin{tabular}{cccccccccc}
\hline
\multirow{2}{*}{}                                   & \multicolumn{2}{c}{\begin{tabular}[c]{@{}c@{}}\textbf{Small (269)} \\ (1-hour limit)\end{tabular}} & \multicolumn{2}{c}{\begin{tabular}[c]{@{}c@{}}\textbf{Medium (94)}\\ (1-hour limit)\end{tabular}}    & \multicolumn{2}{c}{\begin{tabular}[c]{@{}c@{}}\textbf{\textbf{Large (20)}}\\ (5-hour limit)\end{tabular}}   & \multicolumn{2}{c}{\textbf{Total (383)}}       \\
                                                                                                   & \textbf{Count} & \textbf{Time} & \textbf{Count} & \textbf{Time} & \textbf{Count} & \textbf{Time}  & \textbf{Count} & \textbf{Time} \\ \hline
\multicolumn{1}{c}{\textbf{cuPDLP.jl}}     & 266                   & 8.61               & 92                    & 14.80               & 19                    & 111.19 &377 &12.02             \\
\multicolumn{1}{c}{{{$\mathrm{\bf r^2HPDHG}$}}}     & \cellcolor{LightCyan}{267}                   & \cellcolor{LightCyan}{6.61}               & \cellcolor{LightCyan}{93}                    & \cellcolor{LightCyan}{7.84}              & \cellcolor{LightCyan}{19}                    & \cellcolor{LightCyan}{90.81} & \cellcolor{LightCyan}379 & \cellcolor{LightCyan}8.58\\
\multicolumn{1}{c}{\textbf{\begin{tabular}[c]{@{}c@{}}FirstOrderLp.jl\end{tabular}}}               & 253                   & 35.94               & 82                    & 155.67              & 12                    & 2002.21  & 347 & 66.67           \\
\multicolumn{1}{c}{\textbf{\begin{tabular}[c]{@{}c@{}}PDLP (1 thread)\end{tabular}}} & 256                   & 22.69               & 85                    & 98.38             & 15                    & 1622.91   &    356 & 43.81       \\
\multicolumn{1}{c}{\textbf{\begin{tabular}[c]{@{}c@{}}PDLP (4 threads)\end{tabular}}}      & 260                   & 24.03               & 91                    & 42.94               & 15                    & 736.20     & 366 & 34.57   
\\
\multicolumn{1}{c}{\textbf{\begin{tabular}[c]{@{}c@{}}PDLP (16 threads)\end{tabular}}}      & 238                   & 104.72               & 84                    & 142.79               & 15                    & 946.24     & 337 & 127.49       \\ \hline
\end{tabular}
\caption{Solve time in seconds and SGM10 of different solvers on instances of \texttt{MIP Relaxations} with tolerance $10^{-4}$: cuPDLP.jl/$\mathrm{r^2HPDHG}$ versus PDLP.}
\label{tab:miplib-1e-4-pdlp}
\end{table}

\begin{table}[ht!]
\centering
\begin{tabular}{cccccccccc}
\hline
\multirow{2}{*}{}                                   & \multicolumn{2}{c}{\begin{tabular}[c]{@{}c@{}}\textbf{Small (269)} \\ (1-hour limit)\end{tabular}} & \multicolumn{2}{c}{\begin{tabular}[c]{@{}c@{}}\textbf{Medium (94)}\\ (1-hour limit)\end{tabular}}    & \multicolumn{2}{c}{\begin{tabular}[c]{@{}c@{}}\textbf{\textbf{Large (20)}}\\ (5-hour limit)\end{tabular}}   & \multicolumn{2}{c}{\textbf{Total (383)}}       \\
                                                                                                   & \textbf{Count} & \textbf{Time} & \textbf{Count} & \textbf{Time} & \textbf{Count} & \textbf{Time}  & \textbf{Count} & \textbf{Time} \\ \hline
\multicolumn{1}{c}{\textbf{cuPDLP.jl}}     & \cellcolor{LightCyan}261                   & 23.47                & 86                    & 40.69                & \cellcolor{LightCyan}16                    & 421.40  &\cellcolor{LightCyan}363 &32.35             \\
\multicolumn{1}{c}{{{$\mathrm{\bf r^2HPDHG}$}}}     & {260}                   & \cellcolor{LightCyan}{19.13}               & \cellcolor{LightCyan}{87}                    & \cellcolor{LightCyan}{28.35}              & \cellcolor{LightCyan}{16}                    & \cellcolor{LightCyan}{229.47} & \cellcolor{LightCyan}363 & \cellcolor{LightCyan}24.79  \\
\multicolumn{1}{c}{\textbf{FirstOrderLp.jl}}       & 235                   & 91.14                & 68                    & 389.34               & 9                     & 3552.50    &312 &160.63          \\
\multicolumn{1}{c}{\textbf{\begin{tabular}[c]{@{}c@{}}PDLP (1 thread)\end{tabular}}} & 250                   & 49.31               & 73                    & 259.04             & 12                    & 3818.42   &    335 & 96.86       \\
\multicolumn{1}{c}{\textbf{\begin{tabular}[c]{@{}c@{}}PDLP (4 threads)\end{tabular}}}      & 245                   & 54.19               & 81                   & 136.16               & 14                    & 1789.54     & 340 & 83.49         \\ 
\multicolumn{1}{c}{\textbf{\begin{tabular}[c]{@{}c@{}}PDLP (16 threads)\end{tabular}}}      & 214                   & 248.34               & 69                   & 403.17               & 14                    & 2475.57     & 297 & 316.27        \\ 
\hline
\end{tabular}
\caption{Solve time in seconds and SGM10 of different solvers on instances of \texttt{MIP Relaxations} with tolerance $10^{-8}$: cuPDLP.jl/$\mathrm{r^2HPDHG}$ versus PDLP.}
\label{tab:miplib-1e-8-pdlp}
\end{table}

Tables \ref{tab:miplib-1e-4-pdlp} and \ref{tab:miplib-1e-8-pdlp} present the performance comparison of cuPDLP.jl/$\mathrm{r^2HPDHG}$ and its CPU implementations on \texttt{MIP Relaxations} with tolerances of $10^{-4}$ and $10^{-8}$, respectively. The two tables demonstrate that GPU can significantly speed up PDLP, in particular for large instances:
\begin{itemize}
    \item 
    For moderate accuracy (Table \ref{tab:miplib-1e-4-pdlp}), $\mathrm{r^2HPDHG}$ demonstrates a 5x speed-up for small instances, a 20x speed-up for medium instances, and a 22x speed-up for large instances, compared to FirstOrderLp.jl, a CPU implementation of PDLP in Julia. When comparing $\mathrm{r^2HPDHG}$ with C++ implementation PDLP with multithreading support, it still exhibits a significant speedup for all sizes of problems. The significant speed-up can also be observed for high accuracy (Table \ref{tab:miplib-1e-8-pdlp}).
    \item In terms of solved count, cuPDLP.jl and $\mathrm{r^2HPDHG}$ solves significantly more instances regardless of the scales. In particular, comparing with FirstOrderLp.jl under tolerance $\epsilon=10^{-4}$, $\mathrm{r^2HPDHG}$ solves 14 more small-sized instances, 11 more medium-sized instances and 7 more large problems, with in total 32 more instances solved. Compared to PDLP with the best of 1 thread, 4 threads or 16 threads, $\mathrm{r^2HPDHG}$ can solve 7 more small-sized instances and 4 more large-sized instances. The improvement is also remarkable when looking at results for high accuracy $\epsilon=10^{-8}$.
\end{itemize}
In summary, the GPU-implemented cuPDLP.jl and $\mathrm{r^2HPDHG}$ consistently outperform the CPU-implemented PDLP in numerical performance on \texttt{MIP Relaxations}, with cuPDLP.jl and $\mathrm{r^2HPDHG}$ demonstrating even more pronounced advantages on instances of medium to large scale.

\begin{figure}[ht!]
	% \centering
    \hspace{-1cm}
	\begin{tabular}{c c c c}
		& \includegraphics[width=0.5\textwidth]{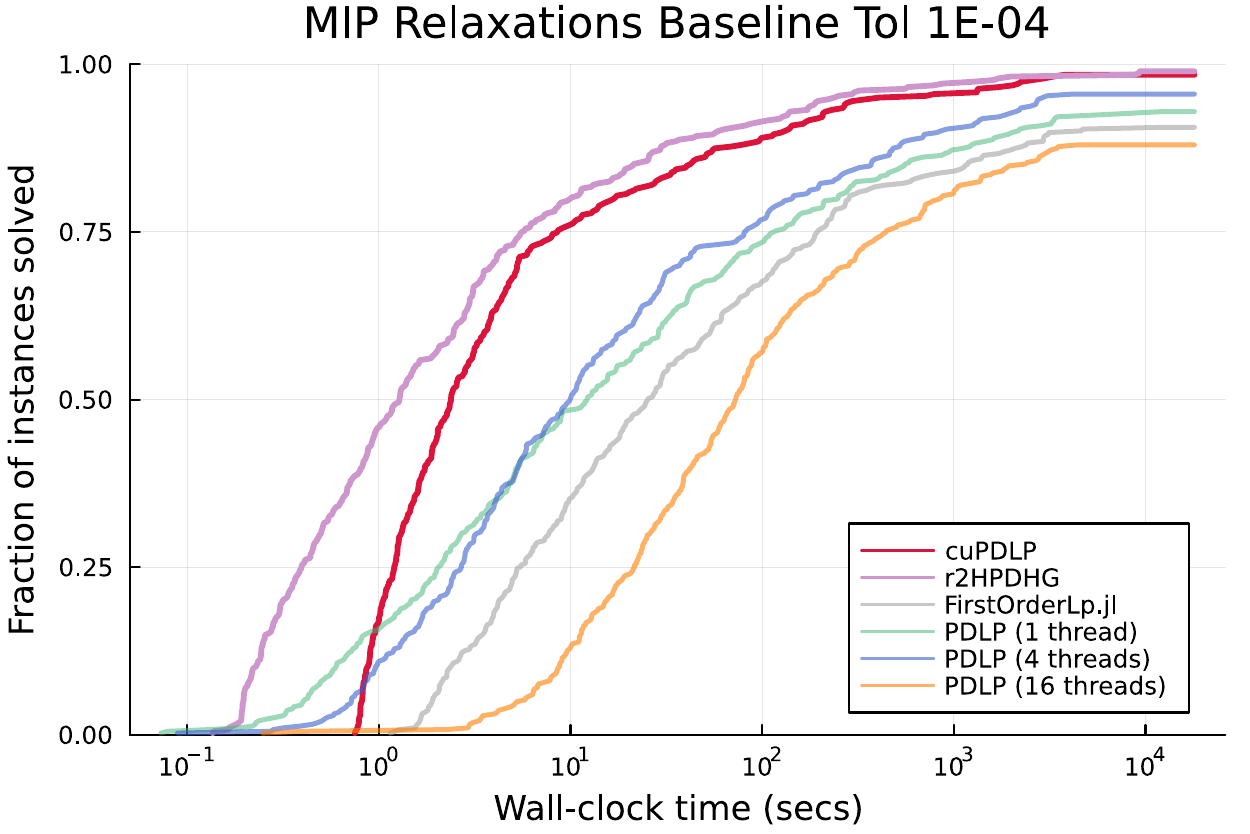}
        & \includegraphics[width=0.5\textwidth]{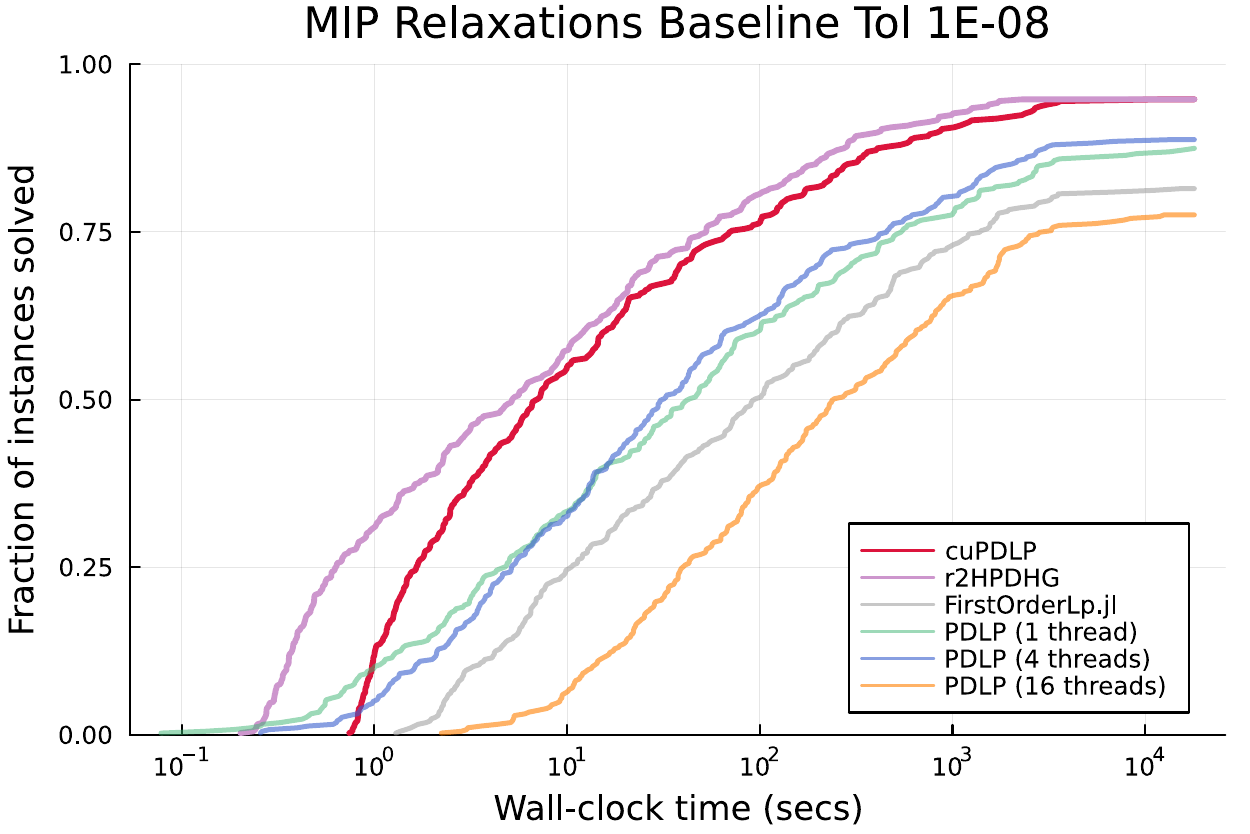}
	\end{tabular}
	\caption{Number of instances solved for \texttt{MIP Relaxations} under moderate accuracy (left) and high accuracy (right): cuPDLP.jl versus PDLP.}
	\label{fig:performance-pdlp}
\end{figure}

Similar to Figure \ref{fig:performance-no-presolve} and \ref{fig:performance-with-presolve}, Figure \ref{fig:performance-pdlp} demonstrates the number of solved instances of cuPDLP.jl/$\mathrm{r^2HPDHG}$ and three CPU implementations of PDLP on \texttt{MIP Relaxations} in a given time. As shown in both panels, when seeking solutions with moderate accuracy ($\epsilon=10^{-4}$) and high accuracy ($\epsilon=10^{-8}$), cuPDLP.jl and $\mathrm{r^2HPDHG}$ have a clear superior performance to all CPU versions of PDLP.

\subsubsection{Additional Comparison from Industry}
As (cu)PDLP has been incorporated into both commercial and open-source optimization solvers, including those developed by Gurobi, COPT, FICO Xpress, HiGHS, Google, and NVIDIA, it has increasingly become the subject of performance evaluations and comparisons across various organizations. Below, we summarize key insights from these discussions.

\begin{itemize}

    \item \textbf{COPT~\cite{coptlink}}: Cardinal Optimizer (COPT) has released an open-source implementation, cuPDLP-C~\cite{lu2023cupdlpc}, which is a C-based version of cuPDLP.jl. According to Mittelmann's benchmark, cuPDLP-C was 2.3× slower than the leading CPU-based solver for obtaining a primal-dual feasible solution (as of February 3, 2025), and 1.84× slower for obtaining an optimal basic solution (as of April 4, 2025), within a reasonable factor. At the meantime, COPT reports numerous instances where cuPDLP-C exhibits significant advantages, particularly in problems like PageRank, unit commitment, supply chain optimization, and quadratic assignment problem. For example, the classic zib03 instance can be solved by cuPDLP-C in under 15 minutes, compared to 16 hours using a CPU-based interior-point method.

    \item \textbf{NVIDIA}: NVIDIA has developed a GPU-based implementation of cuPDLP and plans to open-source it. In a technical blog~\cite{nvidianews}, NVIDIA compares its implementation (cuOpt) against a state-of-the-art CPU-based LP solver on Mittelmann's benchmark set at a tolerance of $10^{-4}$. Among the instances where both solvers converged with a correct objective value, cuOpt was faster on 60\% of the cases and more than 10× faster on 20\%. The largest observed speedup was 5000× on a large-scale multi-commodity flow problem.

    \item \textbf{Gurobi~\cite{gurobilink}}: Gurobi conducted comparisons between GPU-based PDLP and both CPU- and GPU-accelerated barrier methods. Their main conclusion is that performance strongly depends on the model characteristics. While GPU-based PDLP shows promise on certain instances, CPU-based solvers continue to outperform it on the majority of their internal test cases. Gurobi also emphasizes the importance of solver tolerance and the role of crossover procedures in the overall performance and solution usability. 
\end{itemize}

\subsection{Applications}\label{sec:application}
While PDLP is still a relatively new methodology for LP, it has already been applied to a variety of domains, demonstrating practical advantages over existing algorithms. Below is a subset of notable applications:

\begin{itemize} 
\item \cite{googleblog} describes the use of PDLP at Google, particularly in network traffic routing within Google's data centers, leading to significant resource savings. Additionally, PDLP aids in solving large-scale integer two-layer multi-commodity flow problems in the shipping industry, improving the efficiency of container placement and routing. PDLP has also been employed to compute lower bounds for large-scale TSP instances, effectively handling problems with up to 12 billion non-zero entries in the constraint matrix. 
\item \cite{coptlink} describes a real-world application of cuPDLP-C for a security-constrained unit commitment (SCUC) problem, which is a crucial optimization model used in electricity production planning. By combining cuPDLP-C with the simplex and barrier methods in COPT, a 20\% average reduction is achieved in SCUC solution time, enabling faster, more efficient electricity production planning.
\item \cite{lu2023optimizing} highlights how PDLP enables the design and solving of large-scale targeted marketing policies that maximize the total revenue with budget and fairness constraints, whereas commercial solvers fail to solve this class of linear programming problems with moderate size.
\item \cite{de2024power} demonstrates that cuPDLP can significantly improve the computational performance of LP relaxations for $K$-means clustering compared to existing commercial solvers, facilitating large-scale, tighter $K$-means formulations. In particular, the GPU implementation of PDLP provides substantial speedup compared to its CPU counterparts. 
\item \cite{lu2024power} showcases the power of linear programming in sponsored listing rankings, with empirical evidence from an online marketplace where PDLP exhibits strong numerical performance in the overall planning step. 
\item {\cite{kempke2025low} demonstrates the potential of using PDLP within a fix-and-propagate heuristic framework for large-scale unit commitment problems, which are mixed-integer programming. The approach solves unit-commitment instances with up to 243 million non-zeros and 8 million variables, achieving feasible solutions within hours, where commercial solvers fail to produce any in two days.}
\end{itemize}

\subsection{Discussion}\label{sec:discuss}
In this section, we discuss the advantages and limitations of PDLP, compared to the traditional simplex and barrier methods for LP. Notice that PDLP is a rather new methodology, and the future improvements and understanding of the algorithm can likely lead to further speedup.

PDLP offers several key advantages that make it a compelling choice for solving linear programming problems. First, it is highly memory-efficient, as it avoids factorization, enabling it to handle large-scale problems that are intractable for factorization-based methods. Second, its design is well-suited for modern computational hardware, offering superior GPU speedup and efficient scalability. Finally, PDLP excels at quickly obtaining approximate solutions with minimal computational overhead, making it an attractive option in scenarios where high accuracy is not immediately necessary. By comparison, the computational cost of simplex and barrier methods for approximate solutions is nearly identical to that required for high-accuracy solutions. This makes PDLP especially well-suited to machine learning and data science applications, where approximate solutions are often sufficient. Collectively, these features demonstrate PDLP's efficiency and adaptability in modern optimization workflows.

Despite its benefits, PDLP has notable drawbacks that limit its applicability in certain contexts. Achieving high-accuracy solutions with PDLP remains challenging, as first-order methods generally struggle to meet the extreme precision standards required in some applications. Furthermore, PDLP is sensitive to hyper-parameter tuning, such as step-size, primal-weight and preconditioning parameters, which can significantly affect performance. For some problem instances, PDLP may require a substantial number of iterations to converge, leading to increased computational costs compared to factorization-based methods. Furthermore, if the factorization is cheap, then an interior-point solver often beats PDLP even if PDLP does not require that many iterations.  Additionally, the underlying problem structures that influence the ease or difficulty of solving an instance with PDLP are not yet fully understood, complicating its practical deployment. These limitations demonstrate the need for further research and development to make PDLP a more robust and versatile solver.

{
Another important consideration is the crossover procedure, which converts an approximately optimal solution obtained by PDLP into an exact vertex solution of the LP. This step is often essential in applications that require interpretable or sparse solutions, and is particularly valuable when PDLP is used as a subroutine within integer programming. While mature crossover techniques developed for IPMs can be adapted to work with PDLP solutions, their runtime performance is often highly sensitive to the quality of the approximate solution provided by PDLP. Designing an efficient and scalable crossover mechanism specifically tailored for first-order method solvers like PDLP remains an open and important challenge, one that represents a promising direction for future research.
}

Lastly, we comment that the development of PDLP (and first-order methods for mathematical programming more generally) is in a very early stage, similar to the 1990s for the development of interior-point methods. There can be significant potential to further speed up and scale up these methodologies with more substantial numerical development and theoretical understanding.

\section{Theoretical Understandings on PDHG for LP}\label{sec:theory}

In this section, we examine the theoretical properties of PDHG for linear programming (LP). 

For simplicity and cleanness of the exposition, in this section, we discuss the theoretical results of PDHG for solving the standard form of LP~\eqref{eq:primal}. Almost all of these theoretical results can be extended to the general form~\eqref{eq:lp-general}. Furthermore, we assume the primal $\eta$ and the dual step-sizes $\sigma$ are the same throughout this section, i.e., $\sigma=\eta<\frac{1}{\|A\|_2}$. This can be achieved without the loss of generality by considering a rescaled LP instance~\cite{applegate2023faster}. Specifically, we consider primal-dual LP form \eqref{eq:minmax}, and under the above assumptions, PDHG for solving \eqref{eq:minmax} has the update rule as below
\begin{equation}\label{eq:pdhg-thm}
    \begin{aligned}
        & x^{k+1}\leftarrow \text{proj}_{\mathbb R^n_+}(x^k+\eta A^\top y^k-\eta c) \\ 
        & y^{k+1}\leftarrow y^k-\eta A(2x^{k+1}-x^k)+\eta b\ .
    \end{aligned}
\end{equation}

Let $\mathrm{PDHG}(\cdot)$ denote the operator induced by the PDHG iteration~\eqref{eq:pdhg-thm} for solving~\eqref{eq:minmax}, expressed as:
\begin{equation*}
    z^{k+1}=\mathrm{PDHG}(z^k) =\begin{pmatrix}
        \text{proj}_{\mathbb R^n_+}(x^k+\eta A^\top y^k-\eta c) \\ 
        y^k-\eta A(2x^{k+1}-x^k)+\eta b
    \end{pmatrix}\ ,
\end{equation*}
where $z^k=(x^k,y^k)\in \mathbb R^{m+n}$. 

We begin by discussing different progress metrics used in analysis of PDHG for LP (Section \ref{sec:metric}). In Section \ref{sec:unified}, we present a simplified perspective on PDHG for LP, which facilitates a short and intuitive proof of sublinear convergence for both average and last-iterate sequences. In Section \ref{sec:operator}, we interpret the PDHG iteration as a non-expensive operator, discussing theoretical guarantees for Halpern PDHG and its reflected variant. Sharpness conditions are introduced in Section \ref{sec:sharp}, where we explore how these conditions enable linear convergence of the last iterates. Restarted variants of PDHG are discussed in Section \ref{sec:optimal}, demonstrating their ability to achieve an accelerated and optimal complexity for FOMs in solving LP, consistent with established lower bounds. Section \ref{sec:refine} presents refined analyses of PDHG for LP, providing deeper insights into their theoretical performance. The previous subsections assume the LP is feasible and bounded, and lastly, Section \ref{sec:infeas} discusses how PDHG can detect infeasibility without additional computational overhead. 

\begin{table}[ht!]
\small
\hspace{-0.2cm}
\begin{tabular}{|c|ccc|ccc|}
\hline
\textbf{Problem setting}  & \multicolumn{3}{c|}{Convex-Concave}                                & \multicolumn{3}{c|}{Sharp and Convex-Concave}                                                             \\ \hline
\textbf{Type of iterates} & \multicolumn{1}{c|}{Last} & \multicolumn{1}{c|}{Average} & Halpern & \multicolumn{1}{c|}{Last} & \multicolumn{1}{c|}{Average + Restart} & Halpern + Restart \\ \hline
\textbf{Complexity}       & \multicolumn{1}{c|}{$O\pran{{1}/{\epsilon^2}}$}    & \multicolumn{1}{c|}{$O\pran{{1}/{\epsilon}}$}       & $O\pran{{1}/{\epsilon}}$       & \multicolumn{1}{c|}{$O\pran{\kappa^2\log\pran{1/\epsilon}}$}     & \multicolumn{1}{c|}{$O\pran{\kappa\log\pran{1/\epsilon}}$}                  &    $O\pran{\kappa\log\pran{1/\epsilon}}$               \\ \hline
\end{tabular}
\caption{Complexity of PDHG variants under convex-concave and sharp convex-concave setting. $\epsilon$ is the desired accuracy, and $\kappa$ is the condition number of sharp problems (see Section \ref{sec:sharp}).}
\label{tab:rate}
\end{table}

Table \ref{tab:rate} summarizes the complexity results of the PDHG variants discussed in this section. The primal-dual formulation of linear programming \eqref{eq:minmax} is convex-concave and satisfies a certain sharpness condition. Consequently, all the complexity rates of the PDHG variants in Table \ref{tab:rate} apply to solving LPs. In the convex-concave setting, the last-iterate sublinear rate of vanilla PDHG is slower than its averaged or Halpern version. When the primal-dual problem is further sharp, restarted average and restarted Halpern variants could also speed up the complexity. These results explain why averaging and Halpern acceleration schemes are typically employed in PDLP implementations to enhance convergence performance for LP problems.

\subsection{Progress Metrics}\label{sec:metric}

In LP literature and solver development, KKT error is perhaps the most natural and widely used termination metric for quantifying the sub-optimality of a primal-dual solution pair. It provides a comprehensive measure of the quality of a solution by considering three key aspects: primal infeasiblity, dual infeasibility and primal-dual gap. Formally, the KKT error of LP \eqref{eq:primal} is defined as follows:
    \begin{mydef}[KKT error]\label{def:kkt}
        For a primal-dual problem~\eqref{eq:minmax} and a solution $z=(x,y)\in \mathcal Z=\mathbb R^n_+ \times \mathbb R^m$, the KKT error is defined as
        \begin{equation}
            \mathrm{KKT}(z)=\mathrm{KKT}(x,y)=\left\| \begin{pmatrix}
                Ax-b \\ [A^\top y-c]^+ \\ [c^\top x-b^\top y]^+
            \end{pmatrix}\right\|_2 \ .
        \end{equation}
    \end{mydef}

For a primal-dual solution pair $z=(x,y)$, if $\operatorname{KKT}(z)=0$, then $x$ is an optimal solution to the primal LP \eqref{eq:primal}, and $y$ is an optimal solution to the dual LP \eqref{eq:dual} by strong duality of LP~\cite{bertsimas1997introduction}. Although we define the KKT error with the $\ell_2$ norm, effectively any norm can induce a valid progress metric, and commercial LP solvers often utilize $\ell_\infty$ norm in the termination check. We here choose to use $\ell_2$ norm for the ease of presentation, and the analysis presented in this section can be extended to other norms.

Note the progress metric inherent in algorithmic analysis often differs, depending on the algorithm's nature. Since PDHG is a primal-dual first-order method, a commonly used progress metric is the primal-dual gap~\cite{chambolle2011first, chambolle2016ergodic}. However, for LP problems, the feasible region of the primal-dual formulation \eqref{eq:minmax} is unbounded. Consequently, the primal-dual gap at the current solution is often infinite, rendering it uninformative.

To address this issue, \cite{applegate2023faster} introduced the normalized duality gap, a variant of the duality gap that is constrained within a ball centered at the current solution. This refinement provides a more informative progress metric for analyzing primal-dual LP algorithms.

    \begin{mydef}[Normalized duality gap {\cite{applegate2023faster}}]\label{def:ndg}
        For a primal-dual problem~\eqref{eq:minmax} and a solution $z=(x,y)\in \mathcal Z=\mathbb R^n_+ \times \mathbb R^m$, the normalized duality gap with radius $r$ is defined as
        \begin{equation}\label{eq:ndg}
        \rho_r(z)=\max_{\hat z \in W_r(z)}\frac{L(x,\hat y)-L(\hat x,y)}{r} \ ,
    \end{equation}
    where $W_r(z)=\{ \hat z\in\mathcal Z \;|\; \Vert z-\hat z\Vert_{2}\leq r \}$ is a ball centered at $z$ with radius $r$ intersected with $\mathcal Z=\mathbb R^n_+ \times \mathbb R^m$.
    \end{mydef}
    It is straight-forward to show that the normalized duality gap is nonnegative, finite, continuous, and a valid progress measurement for LP, i.e., $\rho_r(z)=0$ if and only if $z$ is an optimal solution to \eqref{eq:minmax}.

    Since \eqref{eq:minmax} represents a convex-concave primal-dual problem, a solution with a zero subdifferential is a saddle point of \eqref{eq:minmax}, corresponding to an optimal primal-dual solution pair. As such, another natural progress metric is the deviation of the subdifferential at the current solution from zero:

    \begin{mydef}[Infimal sub-differential size~\cite{lu2022infimal}]\label{def:ids}
        For a primal-dual problem~\eqref{eq:minmax} and a solution $z=(x,y)\in \mathcal Z=\mathbb R^n_+ \times \mathbb R^m$, we call $\mathcal F(z)=\mathcal F(x,y)=\begin{pmatrix}
        \partial_x L(x,y) \\ -\partial_{y}  L(x,y)
        \end{pmatrix}$ the sub-differential of the objective function $L(x,y)$ (more precisely, the sub-gradient over the primal variable $x$ and the negative sub-gradient over the dual variable $y$), and the infimal sub-differential size is defined as
        \begin{equation}
            \mathrm{dist}_2(0,\mathcal F(z)) = \min_{\hat u\in \mF(z)} \|\hat u\|_2\ .
        \end{equation}
    \end{mydef}

    Another intuitive progress metric is the distance to the optimal solution set, which is commonly used in the linear convergence analysis:

    \begin{mydef}[Distance to optimality]\label{def:dist}
        For a primal-dual problem~\eqref{eq:minmax} a solution $z=(x,y)\in \mathbb R^{m+n}$, %and a positive-semidefinte matrix $P$, 
        the distance of solution $z$ to optimality %with respect to $P$ norm 
        is defined as
        % \begin{equation}
        %     \mathrm{dist}_P(z,\mZ^*) = \min_{z^*\in \mZ^*} \|z-z^*\|_P\ ,
        % \end{equation}
        % where $\mZ^*$ is the set of the optimal primal-dual solution pairs. Particularly, we define 
        \begin{equation}
            \mathrm{dist}_2(z,\mZ^*) = \min_{z^*\in \mZ^*} \|z-z^*\|_2 \ ,
        \end{equation}
        where $\mZ^*$ is the set of the optimal primal-dual solution pairs.
    \end{mydef}

    The last progress metric we introduce here is the fixed-point residual, which measures the movement after one algorithmic iteration. For a reasonable algorithm, an optimal solution should always be a fixed point for the algorithm. For PDHG, it turns out that the inherent norm of the algorithm is associated with the positive-semidefinite matrix $P=\begin{pmatrix}
        \frac{1}{\eta}I & A^\top \\ A & \frac{1}{\eta}I
    \end{pmatrix}$, where $\eta\le \frac{1}{\|A\|_2}$ is the step-size for PDHG (see subsection \ref{sec:unified} for more intuitions and details). The fixed point residual of PDHG is defined with the $P$ norm defined as $\|v\|_P:=\sqrt{\langle v, Pv\rangle}$:
    
    \begin{mydef}[Fixed-point residual]\label{def:fixed-point}
        For a primal-dual problem~\eqref{eq:minmax} and a solution $z=(x,y)\in \mathbb R^{m+n}$, let $\widetilde z:=\mathrm{PDHG}(z)$ be one PDHG iteration from $z$. Let $P$ be a positive definite matrix. Then the fixed-point residual at solution $z$ is defined as $\|z-\widetilde z\|_P$. 
    \end{mydef}

    The next two propositions demonstrate that the KKT error can be upper-bounded by the other four metrics; thus, if any of these metrics go to 0, the KKT error also goes to 0 with the same rate, up to a constant. Notably, while the normalized duality gap and the infimal subdifferential size upper bounds KKT error at the same iterate, fixed-point residual and distance to optimality upper bounds KKT error at the next PDHG iterate.
    \begin{prop}[\cite{applegate2023faster,lu2022infimal}]\label{prop:chain}
    Let $z\in\mathcal Z=\mathbb R^n_+ \times \mathbb R^m$. Then it holds for any $R\geq \|z\|_2$ and $r\in (0,R]$ that
        \begin{equation*}
            \frac{1}{\sqrt{1+R^2}}\mathrm{KKT}(z)\leq \rho_r(z)\leq \mathrm{dist}_2(0,\mathcal F(z)) \ .
        \end{equation*}
    \end{prop}

    Let $\widetilde z := \mathrm{PDHG}(z)$ represent one PDHG iterate from $z$ with stepsize $\eta$, and then it holds from~\cite{lu2022infimal} that 
    $$P(z-\tilde z)\in \mathcal F(\tilde z)\ .$$
    By Proposition \ref{prop:chain}, the following proposition for PDHG iterates holds:
    
    \begin{prop}\label{prop:chain-pdhg}
    Suppose step-size $\eta<\frac{1}{\|A\|_2}$. Then it holds for any $z\in\mathbb R^{m+n}$, $\widetilde z = \mathrm{PDHG}(z)$, $\|\widetilde z\|_2\le R$ and $r\in (0,R]$ that
        \begin{equation*}
            \frac{1}{\sqrt{1+R^2}}\mathrm{KKT}(\widetilde z)\leq \rho_r(\widetilde z)\leq  \sqrt{\sigma_{\operatorname{max}}(P)}\|z-\widetilde z\|_P\leq  2{\sigma_{\operatorname{max}}(P)}\mathrm{dist}_2(z,\mZ^*)\ .
        \end{equation*}
    \end{prop}

\subsection{A Simplified Viewpoint of PDHG for LP}\label{sec:unified}

In this section, we present a perspective on understanding PDHG which can lead to a simplified convergence analysis of its sublinear convergence rate for both the average and last iterates. 
Additionally, we demonstrate that both PPM and ADMM can be viewed as special cases of a generic algorithm, distinguished by their choice of norm.  This unified viewpoint highlights the fundamental connections between these methods, offering principle insights into their theoretical underpinnings.

We assume the LP instances are feasible and bounded in this and the next three sections, and we will discuss infeasibility detection in Section \ref{sec:infeas}. For notational simplicity, we denote $z=(x,y)$ as the primal-dual solution pair and $\mathcal F(z)=\mathcal F(x,y)=\begin{pmatrix}
    \partial_x L(x,y) \\ -\partial_{y}  L(x,y)
\end{pmatrix}$ as the sub-differential of the objective (more precisely, the sub-gradient over the primal variable $x$ and the negative sub-gradient over the dual variable $y$). Denote $\|z\|_{P}=\sqrt{\langle z,Pz\rangle}$ for any positive semi-definite matrix $P$. Denote $\mZ^*$ is the optimal solution set and $\mathrm{dist}_P(z,\mZ^*)=\min_{z^*\in\mZ^*}\|z-z^*\|_P$ is the distance between $z$ and $\mZ^*$. For simplicity of notation, denote $\mathrm{dist}(z,\mZ^*)=\min_{z^*\in\mZ^*}\|z-z^*\|_2$.

\begin{prop}[\cite{lu2022infimal,lu2023unified}]\label{prop:pdhg}
    The update rule of PDHG iterations \eqref{eq:pdhg-thm} can be rewritten as %$P(z^k-z^{k+1})\in\mathcal F(z^{k+1})$
    \begin{equation}
        P(z^k-z^{k+1})\in\mathcal F(z^{k+1}) \ ,
    \end{equation}
    with $P=\begin{pmatrix}
        \frac{1}{\eta}I & A^\top \\ A & \frac{1}{\eta}I
    \end{pmatrix}$, $\mathcal F(z)=\mathcal F(x,y)=\begin{pmatrix}
    c-A^\top y+\partial \iota_{\mathbb R_+^n}(x) \\ Ax-b
    \end{pmatrix}$  and $\iota_{\mathbb R_+^n}(x)=\begin{cases}
        0, & x\in \mathbb R_+^n\\ +\infty, & \mathrm{otherwise}
    \end{cases}$ is the indicator of positive orthant and $\partial \iota_{\mathbb R_+^n}(x)$ is the sub-differential of indicator function of positive orthant.
\end{prop}

This reformulation provides a clearer perspective that facilitates the derivation of PDHG's sublinear convergence. The ergodic sublinear convergence of PDHG was first established in \cite{chambolle2011first}, with a simplified proof later provided in \cite{chambolle2016ergodic}. The sublinear convergence for the last iterate is detailed in works such as \cite{davis2015convergence, davis2016convergence,lu2022infimal}. Here, we summarize these results and offer simplified proofs based on the perspective provided by Proposition \ref{prop:pdhg}.

\begin{thm}[Average iterate convergence {\cite{chambolle2011first,chambolle2016ergodic,lu2023unified}}]\label{thm:average}
    Consider PDHG iterates $\{z^k=(x^k,y^k)\}_{k=1,...,\infty}$ obtained from \eqref{eq:pdhg-thm} with step-size $\eta<\frac{1}{\|A\|_2}$ and the initial solution $z^0=(x^0,y^0)$. Let $P=\begin{pmatrix}
        \frac{1}{\eta}I & A^\top \\ A & \frac{1}{\eta}I
    \end{pmatrix}$. Denote $\bar z^k=(\bar x^k,\bar y^k)=\frac{1}{k}\sum_{i=1}^k z^i$ as the average iterate. Then it holds for any $k\ge 1$ and a primal-dual solution $z=(x,y)$ with $x\geq 0$ that
    \begin{equation*}
        L(\bar x^k,y)-L(x,\bar y^k)\leq \frac{1}{2k}\|z-z^0\|^2_{P} \ .
    \end{equation*}
\end{thm}

\begin{proof}
Denote $w^{k+1}=P(z^k-z^{k+1})\in \mathcal F(z^{k+1})$. It then holds by convexity-concavity of $L(x,y)$ that
\begin{align}\label{eq:bound-gap}
    \begin{split}
        & L(x^{k+1},y)- L(x,y^{k+1}) \ = L(x^{k+1},y)-L(x^{k+1},y^{k+1})+L(x^{k+1},y^{k+1})-L(x,y^{k+1}) \\
         \le & \left\langle w^{k+1}, z^{k+1}-z \right\rangle \ = \left\langle z^{k}-z^{k+1}, z^{k+1}-z \right\rangle_{P} = \frac{1}{2}\|z^{k}-z\|_{P}^2-\frac{1}{2}\|z^{k+1}-z\|_{P}^2-\frac{1}{2}\|z^{k}-z^{k+1}\|_{P}^2 \\ 
         \ \leq & \frac{1}{2}\|z^{k}-z\|_{P}^2-\frac{1}{2}\|z^{k+1}-z\|_{P}^2 \ ,
    \end{split}
\end{align}
where the last inequality follows from the fact that $\|\cdot\|_{P}$ is a semi-norm. We finish the proof by noticing
\begin{align*}
    \begin{split}
        L(\bar x^k,y)-L(x,\bar y^k)\leq \frac 1k \sum_{i=0}^{k-1}L(x^{i+1},y)-L(x,y^{i+1}) \leq \frac{1}{2 k}\|z^{0}-z\|_{P}^2 \ ,
    \end{split}
\end{align*}
where the first inequality comes from convexity-concavity of $L(x,y)$ and the telescoping using \eqref{eq:bound-gap}.
\end{proof}

Theorem \ref{thm:average} demonstrates that the primal-dual gap at the average solution decays at the rate of $O(1/k)$, where $k$ is the number of PDHG iterations. This directly implies the $O(1/k)$ sublinear convergence of normalized duality gap~\cite{applegate2023faster}. The next theorem demonstrates that the primal-dual gap at the last iteration of PDHG decays at the rate of $O(1/\sqrt{k})$.

\begin{thm}[Last iterate convergence \cite{davis2015convergence,davis2016convergence,lu2022infimal,lu2023unified}]\label{thm:last}
Consider PDHG iterates $\{z^k=(x^k,y^k)\}_{k=1,...,\infty}$ obtained from \eqref{eq:pdhg-thm} with step-size $\eta<\frac{1}{\|A\|_2}$ and the initial solution $z^0=(x^0,y^0)$. Let $P=\begin{pmatrix}
        \frac{1}{\eta}I & A^\top \\ A & \frac{1}{\eta}I
    \end{pmatrix}$. Then it holds for any $k\ge 1$ and a primal-dual solution $z=(x,y)$ with $x\geq 0$ that
\begin{equation}\label{eq:last-iteration-residual}
    \Vert z^{k+1}-z^k \Vert_{P} \leq \frac{\Vert z^0-z^* \Vert_{P}}{\sqrt k} \ .
\end{equation}
Furthermore, it holds that
\begin{equation}\label{eq:last-iteration-gap}
    L(x^k,y)-L(x,y^k) \leq \frac{1}{\sqrt k}\pran{\Vert z^0-z^* \Vert_{P}^2+\Vert z^{0}-z^* \Vert_{P}\Vert z^*-z \Vert_{P}} \ .
\end{equation}
\end{thm}

\begin{proof}
From \eqref{eq:bound-gap} we have
    \begin{equation*}
        \Vert z^{k+1}-z^k \Vert_{P}^2 \leq \Vert z^k-z^* \Vert_{P}^2-\Vert z^{k+1}-z^* \Vert_{P}^2 \ ,
    \end{equation*}
and thus it holds by telescoping that
    \begin{equation}\label{eq:sublinear}
        \Vert z^{k+1}-z^k \Vert_{P}^2 \leq \frac{1}{k}\sum_{i=0}^{k-1}\Vert z^{i+1}-z^i \Vert_{P}^2 \leq \frac{1}{k}\sum_{i=0}^{k-1}\Vert z^i-z^* \Vert_{P}^2-\Vert z^{i+1}-z^* \Vert_{P}^2 \leq \frac{\Vert z^0-z^* \Vert_{P}^2}{k} \ .
    \end{equation}
where the first inequality follows monotonicity of $\|z^i-z^{i+1}\|_P$~\cite{lu2022infimal}. We finish the proof using convexity-concavity of $L(x,y)$ by noticing that
    \begin{align}
    \begin{split}
        & L(x^{k+1},y)- L(x,y^{k+1}) \ = L(x^{k+1},y)-L(x^{k+1},y^{k+1})+L(x^{k+1},y^{k+1})-L(x,y^{k+1}) \\
         \le & \left\langle w^{k+1}, z^{k+1}-z \right\rangle \ = \left\langle z^{k}-z^{k+1}, z^{k+1}-z \right\rangle_{P}\leq \|z^{k}-z^{k+1}\|_{P}\|z^{k+1}-z\|_P\\
         \leq & \frac{1}{\sqrt k}\Vert z^0-z^* \Vert_{P}(\Vert z^{k+1}-z^* \Vert_{P}+\Vert z^*-z \Vert_{P})\leq \frac{1}{\sqrt k}\pran{\Vert z^0-z^* \Vert_{P}^2+\Vert z^{0}-z^* \Vert_{P}\Vert z^*-z \Vert_{P}}
    \end{split}
\end{align}
where the second inequality is Cauchy-Schwarz inequality and the third inequality uses \eqref{eq:sublinear}.
\end{proof}

In the analysis of the last iteration convergence of PDHG (Theorem \ref{thm:last}), while the eventual result is on the $O(1/\sqrt{k})$ rate of primal-dual gap \eqref{eq:last-iteration-gap}, it is a byproduct of the $O(\frac{1}{\sqrt{k}})$ rate of the fixed-point residual movement.

Theorem \ref{thm:average} and Theorem \ref{thm:last} are not limited to LP, and work more generally for convex-concave primal-dual problems. In such case, the average iterate of PDHG converges faster than the last-iteration of PDHG, which is the case for many primal-dual algorithms, such as PPM, extra-gradient method (EGM), and ADMM \cite{tseng1995linear,nemirovski2004prox,boyd2011distributed,eckstein1992douglas,he20121}. This faster convergence for the average iterate can be attributed to the oscillatory and/or spiral behavior of primal-dual iterates  (see, for example, Figure \ref{fig:pdhg-gda}), where averaging cancels out oscillations, yielding a relatively faster convergence rate compared to the last iteration.

Indeed PDHG falls into a generic class of algorithms with the following iterate update rule:
\begin{equation}\label{eq:generic}
    z^{k+1}\leftarrow\mathrm{GenericALG}(z^k): P(z^k-z^{k+1})\in\mathcal F(z^{k+1}) \ ,
\end{equation}
where $P\in \mathbb R^{(m+n)\times (m+n)}$ is a positive semi-definite matrix. Proposition \ref{prop:pdhg} claims PDHG is an instance of generic algorithm \eqref{eq:generic} with a specific choice of matrix $P$. Other proper choices of positive semi-definite matrix $P$ give rise to different algorithms. For example, PPM \eqref{eq:ppm} is an instance of \eqref{eq:generic} with $P=\begin{pmatrix}
        \frac{1}{\eta}I & 0 \\ 0 & \frac{1}{\eta}I
    \end{pmatrix}$ and $\mathcal F(z)=\mathcal F(x,y)=\begin{pmatrix}
    c-A^\top y+\partial \iota_{\mathbb R_+^n}(x) \\ Ax-b
\end{pmatrix}$. 
This demonstrates that PDHG can be viewed as a preconditioned variant of PPM with the $P$-norm. The connection between PDHG and PPM was first documented in~\cite{he2012convergence}, while the equivalence between PDHG and Douglas-Rachford Splitting (DRS) was established in~\cite{o2020equivalence}. Unified perspectives on various primal-dual splitting algorithms have also been extensively studied in~\cite{davis2015convergence,davis2016convergence,lu2022infimal,lu2023unified}, providing a comprehensive framework for understanding their relationships and theoretical properties. The results of Theorem \ref{thm:average} and Theorem \ref{thm:last} can be applied to the generic algorithm \eqref{eq:generic}, thus offering a unified and simplified proof for the convergence of PPM and ADMM (see \cite{lu2023unified,lu2022infimal} for a detailed discussion).

\subsection{Firmly Non-Expansive Operator, Halpern Iteration and Reflection}\label{sec:operator}
It is often insightful to view iterative algorithms like PDHG from a non-expansive operator perspective, framing them as iterations to find a fixed point of an operator. Specifically, for PDHG applied to solving LP, recall $\mathrm{PDHG}(\cdot)$ denote the operator induced by the PDHG iteration~\eqref{eq:pdhg-thm} for solving~\eqref{eq:minmax}:
\begin{equation*}
    z^{k+1}=\mathrm{PDHG}(z^k) =\begin{pmatrix}
        \text{proj}_{\mathbb R^n_+}(x^k+\eta A^\top y^k-\eta c) \\ 
        y^k-\eta A(2x^{k+1}-x^k)+\eta b
    \end{pmatrix}\ ,
\end{equation*}
where $z^k=(x^k,y^k)\in \mathbb R^{m+n}$.

The following lemma establishes that PDHG operator for solving~\eqref{eq:minmax} is a firmly non-expansive operator~\cite{bauschke2019convex,ryu2022large}, which demonstrates the contraction property of the PDHG operator, and provides a theoretical foundation for its convergence behavior.

\begin{lem}[\cite{applegate2024infeasibility,bauschke2019convex,ryu2022large}]\label{lem:operator}
    The operator induced by PDHG iteration \eqref{eq:pdhg-thm} with step-size $\eta<\frac{1}{\|A\|_2}$ for solving \eqref{eq:minmax}, denote as $\mathrm{PDHG}(\cdot)$, is firmly non-expansive with respect to norm $\|\cdot\|_P$ where $P=\begin{pmatrix}
        \frac{1}{\eta}I & A^\top \\ A & \frac{1}{\eta}I
    \end{pmatrix}$, namely, for any $u,v\in \mathbb R^{m+n}$,
    \begin{equation*}
        \big\|\mathrm{PDHG}(u)-\mathrm{PDHG}(v)\big\|_P^2\leq \big\|u-v\big\|_P^2-\big\|(\mathrm{PDHG}(u)-u)-(\mathrm{PDHG}(v)-v)\big\|_P^2 \ .
    \end{equation*}
\end{lem}

Firm non-expansiveness immediately implies a bracket of convergence guarantees for PDHG \cite{bauschke2017correction}, for example:
\begin{itemize}
    \item Non-expansiveness: for any $u,v\in \mathbb R^{m+n}$, $\big\|\mathrm{PDHG}(u)-\mathrm{PDHG}(v)\big\|_P\leq \big\|u-v\big\|_P$;
    \item Convergence: there exists a $z^*$ such that $\lim_{k\rightarrow\infty}z^k=z^*$;
    \item Contraction: for any $z^*\in \mZ^*$, it holds for any $k\geq 0$ that $\|z^{k+1}-z^*\|_P \leq \|z^k-z^*\|_P$.
    \item Boundedness: there exists a constant $R>0$ such that for any $k\geq 0$, $\|z^k\|_2\leq R$.
\end{itemize}

Halpern iteration~\cite{halpern1967fixed,lieder2021convergence,diakonikolas2020halpern,park2022exact,kim2021accelerated} is a scheme designed to accelerate non-expansive fixed-point algorithms. Halpern PDHG anchors to the initial solution, and takes a weighted average between the PDHG step at the current iterate and the initial point, and the update rule for Halpern PDHG is given by: 
    \begin{equation*}
        z^{k+1}%=\text{H-PDHG}(z^k;z^0):
        =\frac{k+1}{k+2}\mathrm{PDHG}(z^k)+\frac{1}{k+2}z^0 \ ,
    \end{equation*}
while Figure \ref{fig:hpdhg} visualizes the update rule of Halpern PDHG.

    \begin{figure}
    \begin{subfigure}{0.45\textwidth}
    \centering
    \begin{tikzpicture}
    \tikzstyle{every node}=[font=\large]    
    \filldraw[red] (2,2) circle (2pt) node[left] {$ z^0$};
    \filldraw[red] (0.6,0.5) circle (2pt) node[left] {$z^k$};
    \filldraw[blue] (2,0) circle (2pt) node[below] {PDHG($z^k$)};
    \draw[magenta, dotted, thick] (0.6,0.5) -- (2,0);
    \draw[purple, dotted, thick] (2,2) -- (2,0);
    \filldraw[red] (2,0.3) circle (2pt) node[right] {$z^{k+1}$};
\end{tikzpicture}
\caption{Halpern PDHG}
\label{fig:hpdhg}
\end{subfigure}
    \begin{subfigure}{0.45\textwidth}
    \centering
    \begin{tikzpicture}
    \tikzstyle{every node}=[font=\large]    
    \filldraw[red] (2,2) circle (2pt) node[left] {$ z^0$};
    \filldraw[red] (0.6,0.5) circle (2pt) node[left] {$z^k$};
    \filldraw[blue] (2,0) circle (2pt) node[above] {PDHG($z^k$)};
    \draw[magenta, dotted, thick] (0.6,0.5) -- (3.4,-0.5);
    \filldraw[cyan] (3.4,-0.5) circle (2pt) node[below] {2PDHG($z^k$)-$z^k$};
    \draw[purple, dotted, thick] (2,2) -- (3.4,-0.5);
    \filldraw[red] (3.2,-0.14) circle (2pt) node[right] {$z^{k+1}$};
\end{tikzpicture}
\caption{reflected Halpern PDHG}
\label{fig:rhpdhg}
\end{subfigure}
\caption{Illustration of update schemes of Halpern PDHG and reflected Halpern PDHG}
\end{figure}
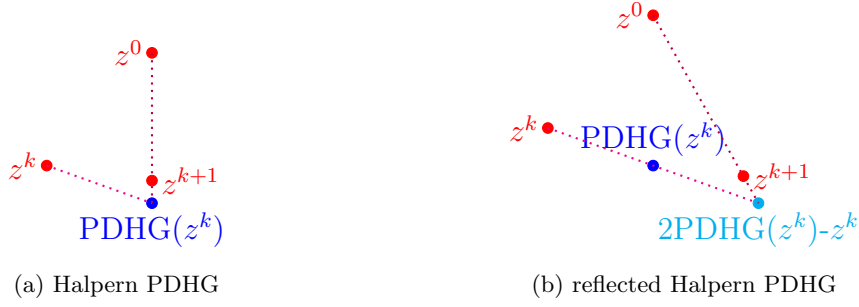

    In light of the non-expansiveness of PDHG operator (Lemma \ref{lem:operator}), the Halpern scheme can speed up the convergence behaviors of PDHG for solving \eqref{eq:minmax}. 
\begin{thm}[{\cite[Theorem 2.1]{lieder2021convergence}}]\label{thm:sublinear-h}
    Consider $\{z^{k}\}$ the Halpern PDHG iterates on solving \eqref{eq:minmax} via update rules \eqref{eq:hpdhg} with step-size $\eta<\frac{1}{\|A\|_2}$.
    Denote $z^*\in \mathcal Z^*:=\{z^*\mid \mathrm{PDHG}(z^*)-z^*=0\}$. Then it holds for any $k\geq 1$ that
    \begin{equation*}
        \left\|z^k-\mathrm{PDHG}(z^k)\right\|_P\leq \frac{2}{k+1}\mathrm{dist}_P(z^0,\mathcal Z^*) \ .
    \end{equation*}
\end{thm}

Halpern iteration has a similar effect as taking the average. For example, consider a unconstrained bilinear problem $\min_{x,y} c^\top x -y^\top Ax + b^\top y$, then the average iterate of PDHG and the last iteration of Halpern PDHG are identical~\cite{lu2024restarted}. Furthermore, both the average PDHG (Theorem \ref{thm:average}) and Halpern PDHG (Theorem \ref{thm:sublinear-h}) lead to a $O(1/k)$ sublinear convergence rate in KKT residual in the light of Proposition \ref{prop:chain} and Proposition \ref{prop:chain-pdhg}.

Another effective enhancement related to Halpern iteration is reflection. Indeed, the sublinear convergence guarantee of Halpern iteration (Theorem \ref{thm:sublinear-h}) just requires non-expansive operators, which is a weaker condition than firm non-expansiveness that is satisfied by PDHG operator. Indeed, it is well known that for a firmly non-expansive operator $T$, its reflection $2T-I$ is a non-expansive operator \cite{bauschke2019convex}, thus we can apply Halpern on reflected PDHG and provide the following update rule
\begin{equation*}
        z^{k+1}%=\text{H-PDHG}(z^k;z^0):
        =\frac{k+1}{k+2}\bigg(2\mathrm{PDHG}(z^k)-z^k\bigg)+\frac{1}{k+2}z^0 \ .
\end{equation*}
Figure \ref{fig:rhpdhg} illustrates the update scheme of reflected Halpern PDHG. Intuitively, the use of reflection takes a longer step, and it can improve the complexity result by a constant factor of 2 in theory:
\begin{thm}[\cite{lieder2021convergence,park2022exact,lu2024restarted}]\label{thm:sublinear-rh}
    Consider $\{z^{k}\}$ the reflected Halpern PDHG iterates on solving \eqref{eq:minmax} via update rules \eqref{eq:rhpdhg} with step-size $\eta<\frac{1}{\|A\|_2}$.
    Denote $z^*\in \mathcal Z^*:=\{z^*\mid \mathrm{PDHG}(z^*)-z^*=0\}$. Then it holds for any $k\geq 1$ that
    \begin{equation*}
        \left\|z^k-\mathrm{PDHG}(z^k)\right\|_P\leq \frac{1}{k+1}\mathrm{dist}_P(z^0,\mathcal Z^*) \ .
    \end{equation*}
\end{thm}

The Halpern scheme and reflection can be applied to other operator splitting algorithms beyond PDHG. Recently, ~\cite{chen2024hpr} proposes the use of Halpern Peaceman-Rachford splitting to solve LP, while the same sublinear rate is achieved therein.

\subsection{Sharpness and Linear Convergence of Last Iterates}\label{sec:sharp}
Theorem \ref{thm:last} seems to imply that the last iterates of PDHG have a slower convergence rate than average iterates (Theorem \ref{thm:average}) and Halpern variants (Theorem \ref{thm:sublinear-h} and Theorem \ref{thm:sublinear-rh}). Contradictorily, one often observes that numerically the last iterates exhibit faster convergence (even linear convergence) than the average and Halpern iterates. This is due to the structure of LP, which satisfies a certain regularity condition that we call the sharpness condition as discussed below.

Sharpness refers to the fact that a function grows at least linearly as it moves away from its minimizers. A fundamental property of LP is the sharpness of the KKT error, which was originally established in the seminal work of Hoffman \cite{hoffman1952approximate} and the constant $H$ is often referred as Hoffman constant:
    \begin{prop}[{\cite{hoffman1952approximate,pena2021new}}]\label{prop:sharp-hoffman}
        Consider an LP instance \eqref{eq:minmax}. Then there exists a constant $H>0$ and it holds for any $z=(x,y)\in \mZ=\mathbb R^n_+ \times \mathbb R^m$ that 
        \begin{equation*}
            \frac{1}{H}\mathrm{dist}_{2}(z,\mathcal Z^*) \leq \mathrm{KKT(z)}\ ,
        \end{equation*} 
        where $\mZ^*$ is the optimal solution set, $\mathrm{dist}_2(z,\mZ^*)=\min_{z^*\in\mZ^*}\|z-z^*\|_2$ is the distance between $z$ and $\mZ^*$ and $\mathrm{KKT}(z)$ is the KKT error defined in Definition \ref{def:kkt}. 
    \end{prop}

In light of the relationship of these three metrics (Proposition \ref{prop:chain}), one can directly show that the normalized duality gap and the infimal differential size are both sharp on a bounded region from the sharpness of KKT residual:

\begin{prop}[{\cite{applegate2023faster}}]\label{prop:sharp-ndg}
        The primal-dual formulation of linear programming \eqref{eq:minmax} is $\alpha$-sharp with respect to norm $\|\cdot\|_2$ on the set $\|z\|_2\leq R$ for all $r\leq R$ and $z\in \mZ=\mathbb R^n_+ \times \mathbb R^m$, i.e., there exists a constant $\alpha>0$ and it holds for any $z\in \mZ=\mathbb R^n_+ \times \mathbb R^m$ with $\|z\|_2\leq R$ and any $r\leq R$ that 
        \begin{equation*}
            \alpha\mathrm{dist}_2(z,\mathcal Z^*)\leq \rho_r(z) \ ,
        \end{equation*}
        where $\mZ^*$ is the optimal solution set, and $\mathrm{dist}_2(z,\mZ^*)=\min_{z^*\in\mZ^*}\|z-z^*\|_2$ is the distance between $z$ and $\mZ^*$.
    \end{prop}

    \begin{prop}[{\cite{zheng2014metric,robinson1981some,dontchev2009implicit}}]\label{prop:sharp-ms}
        The primal-dual formulation of linear programming \eqref{eq:minmax} satisfies $\alpha$-metric sub-regularity on the set $\|z\|_2\leq R$ with $z\in \mZ=\mathbb R^n_+ \times \mathbb R^m$, i.e., there exists a constant $\alpha>0$ and it holds for any $z\in \mZ=\mathbb R^n_+ \times \mathbb R^m$ with $\|z\|_2\leq R$ that 
        \begin{equation*}
            \alpha \mathrm{dist}_{2}(z,\mathcal Z^*) \leq \mathrm{dist}_{2}(0,\mathcal F(z)) \ ,
        \end{equation*} 
        where operator $\mathcal F(z)=\mathcal F(x,y)=\begin{pmatrix}
        \partial_x L(x,y) \\ -\partial_{y}  L(x,y)
        \end{pmatrix}$, $\mZ^*$ is the optimal solution set, and $\mathrm{dist}_2(z,\mZ^*)=\min_{z^*\in\mZ^*}\|z-z^*\|_2$ is the distance between $z$ and $\mZ^*$. 
    \end{prop}

Particularly, the sharpness of the infimal differential size is also called metric sub-regularity, which has been extensively studied in the literature of convex analysis~\cite{dontchev2004regularity,rockafellar2009variational,dontchev2009implicit,zheng2007metric,zheng2014metric,drusvyatskiy2018error,drusvyatskiy2013tilt,drusvyatskiy2013second}. For simplicity, we use the term $\alpha$-metric sub-regularity as a shorthand for the $\alpha$-metric sub-regularity for all $z^*\in\mZ^*$ of $\mathcal F$ at $z^*$ for $0$~\cite{rockafellar2009variational,dontchev2009implicit,drusvyatskiy2018error}.

With the sharpness condition, the next theorem shows that the last iterates of \eqref{eq:pdhg-thm} have global linear convergence on LP. 

\begin{thm}[{\cite{fercoq2022quadratic,lu2022infimal,lu2023unified}}]\label{thm:linear}
    Consider PDHG iterates $\{z^k=(x^k,y^k)\}_{k=1,...,\infty}$ obtained from \eqref{eq:pdhg-thm} with step-size $\eta<\frac{1}{\|A\|_2}$ and the initial solution $z^0=(x^0,y^0)$. Let $P=\begin{pmatrix}
        \frac{1}{\eta}I & A^\top \\ A & \frac{1}{\eta}I
    \end{pmatrix}$. Let $R$ be a constant such that $\|z^k\|_2\leq R$ for any $k\geq 0$. Denote $\alpha$ the corresponding sharpness constant as in Proposition \ref{prop:sharp-ms}. Then it holds for any $k\ge 1$ that    
\begin{equation*}
    \mathrm{dist}_2(z^k,\mathcal Z^*) \leq \exp\pran{1-\frac{k}{\left\lceil e^2\sigma_{\text{max}}^2(P)/\alpha^2\right\rceil}}\mathrm{dist}_2(z^{0},\mathcal Z^*) \ .
\end{equation*}
\end{thm}

Theorem \ref{thm:linear} shows that indeed the last iteration of PDHG has linear convergence to the optimal solution, thanks to the sharpness of LP. In contrast, the average and Halpern iterates, as shown in the previous two sections, always have sublinear convergence if no further modification is made. This explains the numerical observation that quite often the last iteration of PDHG converges faster than the average PDHG iterate and/or the Halpern PDHG. We will see how to speed up the average PDHG and Halpern by restarts in next subsection.

\subsection{Restart and Optimal FOM for LP}\label{sec:optimal}
Theorem \ref{thm:linear} shows that the last iterates of \eqref{eq:pdhg-thm} have linear convergence with complexity $O\pran{\pran{\frac{\|A\|_2}{\alpha}}^2\log\pran{\frac{1}{\epsilon}}}$ with $\eta=O\pran{\frac{1}{\|A\|_2}}$. A natural question is whether there exists an FOM with faster convergence for LP. The answer is yes, and it turns out restart variants of PDHG achieves faster linear convergence and it matches with the complexity lower bound~\cite{applegate2023faster,lu2024restarted}. 

Restart techniques for LP were first introduced in~\cite{applegate2023faster} for various primal-dual algorithms, accompanied by a complexity lower bound to establish the optimality of the proposed approach. The concept of adaptive restart is straightforward: the base algorithm continues iterating until a measurable progress metric, such as KKT error, normalized duality gap, or fixed-point residual, decays by a pre-specified factor between 0 and 1. Once sufficient decay is observed, the base algorithm restarts with a new initial solution, marking the beginning of a new epoch.

In~\cite{applegate2023faster}, normalized duality gap was employed as the progress metric, and the base algorithm (e.g., vanilla PDHG) restarted at the average iterates from the previous epoch. However, the computation of normalized duality gap may involve a non-trivial and sequential trust-region algorithm that is not suitable for GPU implementation. Subsequently, \cite{lu2023cupdlp} demonstrated that KKT error could also serve as a restart progress metric while achieving the same optimal linear convergence rate established in \cite{applegate2023faster}.
More recently, it has been shown that Halpern PDHG, with or without reflection, can utilize fixed-point residual as its restart metric and restart at PDHG iterates from current epoch~\cite{lu2024restarted}. This approach still achieves the optimal rate for solving LP. Moreover, note the relationship between different progress metrics (see Proposition \ref{prop:chain-pdhg}), it is straightforward to extend restarted (reflected) Halpern PDHG to use normalized duality gap or KKT error as the restart metric, while maintaining the optimal convergence rate. 

In the rest of this section, we discuss in detail the restarted Halpern PDHG (Algorithm \ref{alg:hpdhg-restart}), with fixed-point residual as metric. Algorithm \ref{alg:hpdhg-restart} formally presents this algorithm. This is a two-loop algorithm. The inner loop runs \eqref{eq:rhpdhg} until one of the restart conditions holds. At the end of each inner loop, the algorithm restarts the next outer loop at PDHG iterates of the current epoch. 

\begin{algorithm}%[H]
\caption{Restarted Halpern PDHG for \eqref{eq:minmax}}
\label{alg:hpdhg-restart}
\SetKwInOut{Input}{Input}
\Input{Initial point $(x^0,y^0)$, outer loop counter $n\leftarrow 0$.}

\Repeat{\upshape $(x^{n+1,0},y^{n+1,0})$ convergence}{
  initialize the inner loop counter $k\leftarrow0$;\\
  \Repeat{\upshape restart condition \eqref{eq:adaptive-restart} holds}{
    $(x^{n,k+1},y^{n,k+1})\leftarrow \frac{k+1}{k+2}\mathrm{PDHG}(x^{n,k},y^{n,k})+\frac{1}{k+2}(x^{n,0},y^{n,0})$ as in \eqref{eq:rhpdhg};
  }
  initialize the initial solution $(x^{n+1,0}, y^{n+1,0})\leftarrow \mathrm{PDHG}(x^{n,k},y^{n,k})$;\\
  $n\leftarrow n+1$;
}
\end{algorithm}

A crucial component of the algorithm is when to restart. In \cite{lu2024restarted}, an adaptive restart scheme is proposed: we initiate a restart when there is significant decay in the difference between iterates $z^{n,k}$ and $\mathrm{PDHG}(z^{n,k})$. This adaptive restart mechanism does not require an estimate of sharpness $\alpha$ (as defined in Proposition \ref{prop:sharp-ms}) which is almost inaccessible in practice~\cite{pena2021new}. Specifically, we restart the algorithm if
    \begin{equation}\label{eq:adaptive-restart}
    \begin{cases}
        \left\|z^{n,k}-\mathrm{PDHG}(z^{n,k})\right\|_P\leq \cfrac 1e \left\|z^{n,0}-\mathrm{PDHG}(z^{n,0})\right\|_P, & n\geq 1\\
        k>\tau^0,& n=0
    \end{cases}
\end{equation}

\begin{thm}[Adaptive restart for LP]\label{thm:adaptive-restart}
Consider $\{z^{n,k}\}$ the iterates of restarted Halpern PDHG (Algorithm \ref{alg:hpdhg-restart}) with adaptive restart scheme, namely, we restart the outer loop if \eqref{eq:adaptive-restart} holds. Denote $\mathcal Z^*$ is the set of optimal solutions. Then it holds for any $n\geq 0$ that
\begin{enumerate}
    \item[(i)] The restart length $\tau^n$ is upper bounded by $k^*$,
    \begin{equation*}
    \tau^n\leq \left\lceil{\frac{2e\sigma_{max}(P)}{\alpha}}\right\rceil=:k^* \ ,
    \end{equation*}
    \item[(ii)] The distance to optimal set decays linearly,
    \begin{equation*}
        \mathrm{dist}_2(z^{n+1,0},\mathcal Z^*) \leq \pran{\frac 1e}^{n+1}\frac{2e\sigma_{max}(P)}{ (\tau^0+1)\alpha}\mathrm{dist}_2(z^{0,0},\mathcal Z^*) \ .
    \end{equation*}
\end{enumerate}
\end{thm}

Theorem \ref{thm:adaptive-restart} shows that the iterates of Algorithm \ref{alg:hpdhg-restart} have linear convergence with complexity $O\pran{\frac{\|A\|_2}{\alpha}\log\pran{\frac{1}{\epsilon}}}$ with $\eta=O\pran{\frac{1}{\|A\|_2}}$. Compared with $O\pran{\pran{\frac{\|A\|_2}{\alpha}}^2\log\pran{\frac{1}{\epsilon}}}$ as proved in Theorem \ref{thm:linear}, restart variants of PDHG achieve an accelerated linear rate in the sense of better dependence on condition number $\frac{\|A\|_2}{\alpha}$.

Furthermore, the complexity lower bound established in \cite{applegate2023faster} shows that restarted reflected Halpern PDHG is optimal across a wide range of FOMs for LP. In particular, we consider the span-respecting FOMs:

\begin{mydef}\label{def:span-respecting}
    An algorithm is span-respecting for an unconstrained primal-dual problem $\min_x\max_y\; L(x,y)$ if
    \begin{align*}
        \begin{split}
            & \ x^k\in x^0+\mathrm{span}\{\nabla_x L(x^i,y^j): \forall i,j\in\{1,...,k-1\}\}\\
            & \ y^k\in y^0+\mathrm{span}\{\nabla_y L(x^i,y^j): \forall i\in\{1,...,k\},\forall j\in\{1,...,k-1\} \} \ .
        \end{split}
    \end{align*}
\end{mydef}
Definition \ref{def:span-respecting} is an extension of the span-respecting FOMs for minimization~\cite{nesterov2003introductory} in the primal-dual setting. 
Theorem \ref{thm:lb} provides a lower complexity bound of span-respecting primal-dual algorithms for LP.

\begin{thm}[Lower complexity bound {\cite{applegate2023faster}}]\label{thm:lb}
    Consider any iteration $k\geq 0$ and parameter value $\gamma>\alpha>0$. There exists an $\alpha$-sharp linear programming with $\|A\|_2=\gamma$ such that the iterates $z^k$ of any span-respecting algorithm  satisfies that
    \begin{equation*}
        \mathrm{dist}_2(z^k,\mathcal Z^*)\geq \pran{1-\frac{\alpha}{\gamma}}^{k} \mathrm{dist}_2(z^0,\mathcal Z^*) \ .
    \end{equation*}
\end{thm}

Together with Theorem \ref{thm:adaptive-restart}, Theorem \ref{thm:lb} shows that restarted Halpern PDHG is an optimal FOM for LP (upto log factor), i.e., the convergence rate of restarted Halpern PDHG matches lower bound complexity (upto log factor). Lastly, we want to highlight that such an ``optimal'' argument does not rule out possible better practical algorithms, because it is rooted in worst-case analysis, namely, the rate is optimal only on a constructed worst-case instance.

\subsection{Refined Analysis}\label{sec:refine}
In the previous section, we establish global linear convergence rates for vanilla and restarted PDHG that depend on the sharpness constant $\alpha$. However, the sharpness constant in Proposition \ref{prop:sharp-ndg} and Proposition \ref{prop:sharp-ms} is essentially the reciprocal of global Hoffman constant of the KKT system corresponding to the LP (Proposition \ref{prop:sharp-hoffman})~\cite{applegate2023faster,lu2022infimal}, whereas the Hoffman constant is well-known to be overly conservative and uninformative~\cite{pena2021new}. Indeed, it is highly difficult to provide a simple characterization of the sharpness constant $\alpha$. One characterization for system $Fx=g, \tilde Fx\leq \tilde g$ is described in \cite{pena2021new}:
\begin{equation}\label{eq:hoffman}
    \alpha = \min_{J\in\mathcal S(F,\tilde F)}\min_{\substack{v\in\mathbb R_+^J,\; z\in F(\mathbb R^{n})\\ \|(v,z)\|_2=1,\; \tilde F_J^Tv+Fz-u=0}}\|u\|_2 \ ,
\end{equation}
where $\mathcal S(F,\tilde F)=\{ J\subset \{1,2,...,m\}:\{(\tilde Fx+s,Fx):s\in \mathbb R^m,s_J\geq 0,x\in\mathbb R^n \} \mathrm{\; is\;a\;linear\;subspace}\}$.
Informally speaking, the inner minimization in \eqref{eq:hoffman} computes an extension of minimal positive singular value for a certain matrix, which is specified by an ``active set'' $J$ of constraints. The outer minimization takes the minimum of these extended minimal positive singular values over all possible ``active sets'' (intuitively, it goes through every extreme point, edge, face, etc., of solution set $\mathcal X^*=\{x\in \mathbb R^n: F x= g, \tilde F x \le \tilde g\}$). While the inner minimization is expected when characterizing the behaviors of an algorithm, similar to the strong convexity in a minimization problem, there are usually exponentially many ``active sets'' for a polytope $\mathcal{X}^*$, thus $\alpha$ defined in \eqref{eq:hoffman} is known to be a loose bound and it is generally NP-hard to compute its exact value or even just a reasonable bound.

On the other hand, it is evident that the numerical performance of first-order algorithms does not depend on the overly-conservative global Hoffman constant. The rate derived in Section \ref{sec:theory} can be too loose to interpret the successful empirical behavior of the algorithm. A more refined analysis is favorable to characterize the actual behavior of the algorithm. In this section, we briefly overview three lines of research that aim to overcome this drawback.

\subsubsection{Two-stage Trajectory-based Analysis}
In the context of non-smooth optimization, there has been significant progress in studying the identification properties of first-order methods and achieving fast local linear convergence under partial smoothness~\cite{wright1993identifiable,lewis2016proximal,liang2017activity,liang2017local,liang2018local,poon2020geometry}. However, these results rely on the non-degeneracy condition, which, in the context of LP, implies that the algorithm converges to an optimal solution satisfying strict complementary slackness. Unfortunately, this condition is rarely met in real-world LP instances when using PDHG, limiting the applicability of these theoretical guarantees to LP.

Despite the lack of a non-degeneracy condition, it is still observed that PDHG on LP often exhibits a two-stage behavior~\cite{lu2023geometry}: an initial sublinear convergence followed by a linear convergence. Figure \ref{fig:pdhg} shows the representative behaviors of PDHG on four LP instances from MIPLIB and Netlib. The convergence patterns of the algorithm differ dramatically across these LP instances. The instance \texttt{mad} converges linearly to optimality within only few thousands of steps. Instances \texttt{neos16} and \texttt{gsvm2r13} eventually reach the linear convergence stage but beforehand there exists a relatively flat slow stage. The instance \texttt{sc105} does not exhibit linear rate within twenty thousand iterations. Furthermore, the eventual linear rates are not equivalent: \texttt{gsvm2r13} exhibits much faster linear rate than \texttt{neos16}. Similar observation holds for the ``warm-up" sublinear stage: \texttt{neos16} has much shorter sublinear period than \texttt{gsvm2r13}. The global linear convergence with the conservative Hoffman constant~\cite{lu2022infimal,applegate2023faster} is clearly not enough to interpret the diverse empirical behaviors. 
\begin{figure}[ht]
	\centering
	\hspace{-1Cm}
	\includegraphics[width=0.5\textwidth]{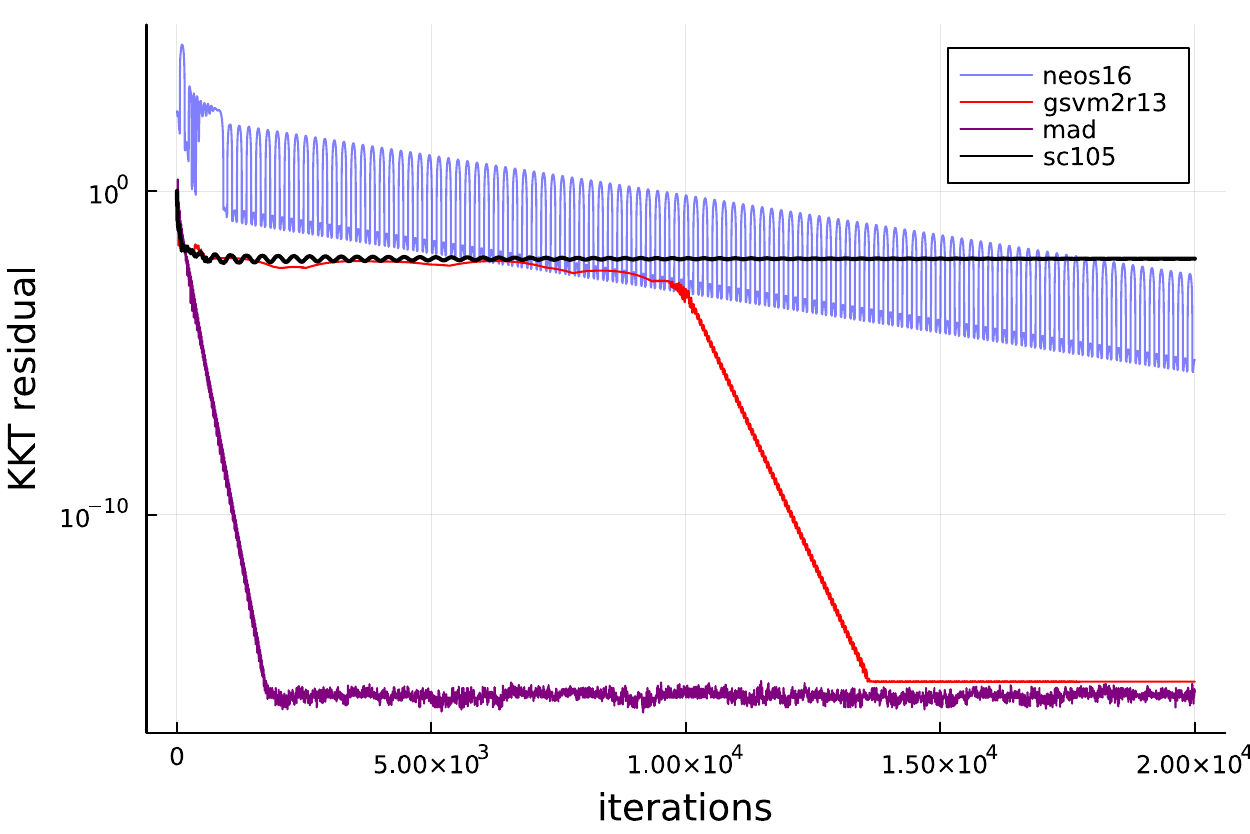}
	\caption{PDHG on four LP relaxation of instances from benchmark sets}
	\label{fig:pdhg}
\end{figure}

The results in \cite{lu2023geometry} demonstrate a two-stage convergence behavior of vanilla PDHG on solving LP:
\begin{itemize}
    \item In the first stage, the algorithm identifies the non-degenerate variables.  This stage finishes within a finite number of iterations, and the convergence rate in this stage is sublinear. The driving force of the first stage is how close the non-degenerate part of the LP is to degeneracy.
    \item In the second stage, the algorithm effectively solves a homogeneous linear inequality system. The driving force of the linear convergence rate is a local sharpness constant for the homogeneous linear inequality system. This local sharpness constant is much better than the global Hoffman constant, and it is a generalization of the minimal non-zero singular value of a certain matrix. Intuitively, this happens due to the structure of the homogeneous linear inequality system so that one just needs to focus on the local geometry around the origin (See~\cite{pena2024easily} for a detailed characterization), avoiding going through the exponentially large boundary set as in the calculation of global Hoffman constant.
\end{itemize}
Such two-stage behavior of restarted PDHG is also documented and analyzed in more recent works~\cite{xiong2024accessible,xiong2025high} under more specific settings such as LP with unique optima and random LP. 

\subsubsection{Geometric Quantity and Average Analysis}
Mostly motivated by the limitations of Hoffman-based complexity, a series of works~\cite{xiong2023computational,xiong2023relation,xiong2024role,xiong2024accessible,xiong2025high} develops novel conditioning measures for LP instances to refine the complexity of (restarted) PDHG. In contrast to the algebraic nature of Hoffman constant, the newly proposed conditioning measures are based on the geometry of the LP instances and enables a tighter and more transparent characterization of the actual behaviour, in line with fruitful condition number theories for (conic) LP~\cite{epelman2000condition,freund1999some,belloni2009geometric,freund2004complexity,freund2003primal,renegar2001mathematical,renegar1988polynomial,renegar1995linear,renegar1993some,renegar1995incorporating,pena2000understanding,vera2007primal,cheung2003unifying}.

In \cite{xiong2023computational,xiong2023relation}, two purely geometric condition measures are introduced: the ``limiting error ratio'' and LP sharpness, and develop new computational guarantees for the restarted PDHG method applied to linear programming (LP). The limiting error ratio captures instance-specific feasibility properties, while LP sharpness quantifies solution stability under objective perturbations. These intrinsic measures of LP instances result in tighter theoretical iteration bounds for PDHG compared to the original Hoffman-based complexity, as validated through constructed test cases and experiments on the MIPLIB dataset.

Despite the tighter characterization provided by the ``limiting error ratio'' and LP sharpness, these measures can still take extreme values, leading to poor theoretical performance of PDHG. To address this, \cite{xiong2024role} introduces the use of level-set geometry to analyze PDHG behavior. Specifically, three geometric measures for level sets—diameter, radius and center, and the Hausdorff distance to optimality—are proposed to evaluate the convergence properties of restarted PDHG and establish computational guarantees. The theoretical results in \cite{xiong2024role} extend to more general conic linear programming problems. Furthermore, the paper demonstrates how rescaling transformations, guided by the central path, can enhance these geometric measures and accelerate convergence.

In \cite{xiong2024accessible}, a computationally accessible iteration bound for restarted PDHG is established for LP with unique optimal solution. The proposed bound depends on intrinsic properties of the LP instance and admits a closed-form expression, enabling practical evaluation and deeper insights into PDHG’s behavior. A two-stage convergence phenomenon related to~\cite{lu2023geometry} is also analyzed for restarted PDHG in this context.

A more recent work~\cite{xiong2025high} is among the research line of average analysis to addresses the theoretical gap between the practical efficiency and worst-case iteration bounds. By employing probabilistic analysis, the study establishes high-probability polynomial-time complexity bounds for restarted PDHG under sub-Gaussian and Gaussian input data models, in contrast to the worst-case data-dependent complexity. More precesily, \cite{xiong2025high} shows that restarted PDHG can find a solution with $\mathrm{dist}(z,\mathcal Z^*)\leq \epsilon$ within
\begin{equation*}
    \widetilde{O}\pran{\pran{n^{2.5}m^{0.5}+n^{0.5}m^{0.5}\log\frac{1}{\epsilon}}\frac{1}{\delta}}
\end{equation*}
iterations with probability at least $1-\delta$. The results provide theoretical justification for the strong empirical performance of restarted PDHG.

\subsubsection{Use of Problem Structure}
As shown in \eqref{eq:hoffman}, the global sharpness constant is often overly conservative in worst-case scenarios. Beyond the analyses based on two-stage approaches and geometric quantities, another active research direction leverages the structure of specialized problems, such as totally unimodular linear programming~\cite{hinder2024worst} and optimal transport~\cite{lu2024pdot}.

A totally unimodular linear program is a special class of LP in which the constraint matrix is totally unimodular, i.e., every square submatrix has a determinant of \( 0 \), \( 1 \), or \(-1\), and has integer right-hand sides and an integer objective vector~\cite{schrijver1998theory}. These LPs are particularly useful in fields such as network flows and combinatorial optimization~\cite{ahuja1988network,schrijver1998theory,korte2011combinatorial,bazaraa2011linear}. In~\cite{hinder2024worst}, the complexity of restarted average PDHG~\cite{applegate2023faster} for solving totally unimodular LPs is analyzed by providing a refined characterization of the sharpness constant for the KKT system. Notably, the sharpness constant for totally unimodular LPs with $m$ constraints can be lower bounded as follows:
\begin{equation*}
    \alpha\geq \Omega\pran{\frac{1}{m^{2.5}H}}
\end{equation*}
where $H\geq \max\{\|b\|_\infty,\|c\|_\infty\}$, and combined with complexity results in \cite{applegate2023faster}, this immediately establishes that the complexity of restarted PDHG for solving totally unimodular LPs is polynomial. More precisely, restarted PDHG can achieve a solution with $\mathrm{dist}_2(z,\mZ^*)\leq \epsilon$ within
\begin{equation*}
    O\pran{Hm^{2.5}\sqrt{\mathrm{nnz}(A)}\log\pran{\frac{mH}{\epsilon}}}
\end{equation*}
iterations, where $\mathrm{nnz}(A)$ represents the number of nonzeros in constraint matrix $A$.

Another example is \cite{lu2024pdot}, which examined the complexity of restarted PDHG for solving optimal transport (OT) problems. Optimal transport, also known as mass transportation, the earth mover's distance, or the Wasserstein-1 (\(W_{1}\)) distance, is a fundamental class of optimization problems that measures the distance between probability distributions~\cite{villani2009optimal,villani2021topics}. First introduced in the 18th century through Monge's pioneering work~\cite{monge1781founding}, OT has since become a cornerstone in various domains, including operations research, economics, machine learning, and image processing~\cite{villani2009optimal,villani2021topics,peyre2019computational}. The discrete OT problem~\cite{kantarovich1939mathematical} is mathematically formulated as:
\begin{equation}\label{eq:ot}
    \begin{aligned}
        \min_{X \geq 0} & \ \langle C, X \rangle \\ 
        \mathrm{s.t.} & \ X \mathbf{1}_n = f \\ 
        & \ X^\top \mathbf{1}_m = g \ ,
    \end{aligned}
\end{equation}
where $C = [C_{ij}]_{1 \leq i \leq m, 1 \leq j \leq n}$ denotes the given non-negative cost matrix, and $f = [f_i]_{i=1}^m$ and $g = [g_j]_{j=1}^n$ are probability vectors representing the row and column marginals, respectively. Here, $\mathbf{1}_n$ and $\mathbf{1}_m$ are vectors of ones with dimensions $n$ and $m$. The objective of the discrete optimal transport (OT) problem is to find a non-negative matrix \( X \), referred to as the transportation plan, that minimizes the total transportation cost \( \langle C, X \rangle \), subject to the constraint that the marginals of \( X \) align with the given vectors \( f \) and \( g \), which represent distributions. In~\cite{lu2024pdot}, it is shown that restarted PDHG achieves a data-independent local complexity of floating-point operations:
\begin{equation*}
    \widetilde{O}\pran{mn(m+n)^{1.5}\log\pran{\frac{1}{\epsilon}}} \ .
\end{equation*}
Furthermore, a data-dependent global floating-point-operations count of 
\begin{equation*}
    \widetilde O\pran{mn(m+n)^{3.5}\Delta + mn(m+n)^{1.5}\log\pran{\frac{1}{\epsilon}}}
\end{equation*}
is established, where \( \Delta \) represents the precision of the data.

In summary, restarted PDHG demonstrates polynomial complexity for specific classes of LPs, such as totally unimodular LPs and optimal transport problems. These polynomial rates can be directly extended to other variants, including restarted (reflected) Halpern PDHG. Such refined characterizations for specialized problems provide theoretical justification for the observed strong numerical performance of PDLP in solving these problems.

\subsection{Infeasibility Detection}\label{sec:infeas}
The convergence results in the previous section require the LP to be feasible and bounded. In practice, it is occasionally the case that the LP instance is infeasible or unbounded, thus infeasibility detection is a necessary feature for any LP solver. In this section, we investigate the behavior of PDHG on infeasible/unbounded LPs, and claim that the PDHG iterates encode infeasibility information automatically.

Consider the primal LP \eqref{eq:primal} and its dual form \eqref{eq:dual}. Farkas’ Lemma~\cite{farkas1902theorie,gale1951linear} states that a feasible solution of \eqref{eq:primal} exists if, and only if, the following set is empty
\begin{equation}\label{eq:primal-infeas}
    \{y\in \mathbb R^m\ |\ b^\top y<0, A^\top y \geq0\} \ .
\end{equation}
We refer to the elements of this set as certificates of primal infeasibility, as their existence provides a guarantee that the primal problem is infeasible. Similarly, the certificates of infeasibility for the dual problem~\eqref{eq:dual} are defined analogously, ensuring that the dual problem is also identified as infeasible under corresponding conditions:
\begin{equation}\label{eq:dual-infeas}
    \{x\in \mathbb R^n\ |\ c^\top x<0, Ax=0, x\geq0\} \ .
\end{equation}

The easiest way to describe the infeasiblity detection property of PDHG, i.e., ability to find elements in either \eqref{eq:primal-infeas} and \eqref{eq:dual-infeas}, is perhaps to look at it from an operator perspective. More formally, we use $\mathrm{PDHG}(\cdot)$ to represent the operator for one step of PDHG iteration for solving \eqref{eq:minmax}, i.e., $z^{k+1}=\mathrm{PDHG}(z^k)$ where $\mathrm{PDHG}$ is specified by \eqref{eq:pdhg}. Next, we introduce the infimal displacement vector of the operator $\mathrm{PDHG}$, which plays a central role in the infeasibility detection of PDHG:
\begin{mydef}
    For PDHG operator $\mathrm{PDHG}(\cdot)$ induced by PDHG on \eqref{eq:minmax},  we call $$v:=\argmin_{z\in\mathrm{range}(\mathrm{PDHG}-I)} \|z\|_2^2$$ its infimal displacement vector (which is uniquely defined~\cite{pazy1971asymptotic}).
\end{mydef}
It turns out that if the LP instance is primal (or dual) infeasible, then the dual (or primal) variables diverge like a ray with direction $v$. Furthermore, the corresponding dual (or primal) part of $v$ provides an infeasibility certificate for the primal. Table \ref{tab:infeas} summaries such results. More formally, 

\begin{thm}[Behaviors of PDHG for infeasible LP {\cite{applegate2024infeasibility}}]\label{thm:full-infeas}
    Consider the primal problem \eqref{eq:primal} and dual problem \eqref{eq:dual}. Assume $\eta< \frac{1}{\|A\|}$, let $\mathrm{PDHG}(\cdot)$ be the operator induced by PDHG on \eqref{eq:minmax}, and let $\{z^k=(x^k,y^k)\}_{k=0}^{\infty}$ be a sequence generated by the fixed-point iteration from an arbitrary starting point $z^0$. Then, one of the following holds:

    (a). If both primal and dual are feasible, then the iterates $(x^k,y^k)$ converge to a primal-dual solution $z^*=(x^*,y^*)$ and $v=(\mathrm{PDHG}-I)(z^*)=0$.

    (b). If both primal and dual are infeasible, then both primal and dual iterates diverge to infinity. Moreover, the primal and dual components of the infimal displacement vector $v=(v_x,v_y)$ give certificates of dual and primal infeasibility, respectively.

    (c). If the primal is infeasible and the dual is feasible, then the dual iterates diverge to infinity, while the primal iterates converge to a vector $x^*$. The dual-component $v_y$ is a certificate of primal infeasibility. Furthermore, there exists a vector $y^*$ such that $v=(\mathrm{PDHG}-I)(x^*,y^*)$.

    (d). If the primal is feasible and the dual is infeasible, then the same conclusions as in the previous item hold by swapping primal with dual.
\end{thm}

\begin{table}[h]
\centering
{\large
\begin{tabular}{|c|c|c|}
\hline
\diagbox[width=10em]{\textbf{Primal}}{\textbf{Dual}}                   & \textbf{Feasible} & \textbf{Infeasible} \\ \hline
\textbf{Feasible}   &  $x^k,y^k$ both converge              &  $x^k$ diverges, $y^k$ converges    \\ \hline
\textbf{Infeasible} &  $x^k$ converges, $y^k$ diverges                 &  $x^k,y^k$ both diverge                    \\ \hline
\end{tabular}
}
\caption{Behavior of PDHG for solving \eqref{eq:minmax} under different feasibility assumptions.}
\label{tab:infeas}
\end{table}

Furthermore, one can show that the difference of iterates and the normalized iterates converge to the infimal displacement vector $v$ with sublinear rate:
\begin{thm}[{\cite{applegate2024infeasibility,davis2016convergence}}]\label{thm:infeas}
    Let $\mathrm{PDHG}(\cdot)$ be the operator induced by PDHG on \eqref{eq:minmax} with step-size $\eta<\frac{1}{\|A\|_2}$. Then there exists a finite $z^*$ such that $\mathrm{PDHG}(z^*)=z^*+v$ and for any such $z^*$ and all $k$:

    (a) (Difference of iterates)
    \begin{equation*}
        \min_{j\leq k}\|v-(z^{j+1}-z^j)\|_{P}\leq \frac{1}{\sqrt k}\|z^0-z^*\|_{P}\ ,
    \end{equation*}

    (b) (Normalized iterates)
    \begin{equation*}
        \left\|v-\frac{1}{k}(z^k-z^0)\right\|_{P}\leq \frac{2}{k}\|z^0-z^*\|_{P}\ .
    \end{equation*}
\end{thm}

Theorem \ref{thm:full-infeas} and Theorem \ref{thm:infeas} show that the difference of iterates and the normalized iterates of PDHG can recover the infeasibility certificates with sublinear rate. Notice that the first result in Theorem \ref{thm:infeas} is consistent with the feasible and bounded case in Theorem \ref{thm:last}, in which case the infimal displacement vector $v=0$, and the PDHG movement decays at the rate of $O(1/\sqrt{k})$ (see \eqref{eq:last-iteration-gap}). 

While the normalized iterates have faster sub-linear convergence than the difference of iterates, one can show that the difference of iterates exhibits linear convergence in the local regime to the infimal displacement vector under additional non-degeneracy conditions~\cite{applegate2024infeasibility}. Thus, both the difference of iterates and normalized iterates are checked in infeasiblity detection of PDLP~\cite{applegate2024infeasibility,lu2023cupdlp,lu2023cupdlpc}. More recently, \cite{lu2024restarted} shows that restarted (reflected) Halpern PDHG achieves stronger theoretical guarantee for infeasibility detection of LP in the sense that
\begin{itemize}
    \item Linear convergence is achieved without any additional regularity assumptions;
    \item The linear rate achieved is an acceleration over the results in~\cite{applegate2024infeasibility}, as it eliminates a square term in the condition number.
\end{itemize}

\section{GPU-based mathematical programming beyond LP}\label{sec:beyondlp}
This section provides a brief overview of the recent advancements in GPU-based optimization solvers beyond linear programming. Extensions to convex quadratic programming (QP) are detailed in Section \ref{sec:qp}, while developments in semi-definite programming (SDP) are summarized in Section \ref{sec:sdp}. The section concludes with discussion on leveraging GPUs to accelerate solvers based on interior-point methods (IPMs) for solving conic programming and nonlinear programming problems in Section \ref{sec:ipm}.

\subsection{Convex Quadratic Programming}\label{sec:qp}
Convex quadratic programming (QP) is a direct extension of LP to minimize a quadratic objective over linear constraints. More specifically, a form of QP is 
\begin{equation}\label{eq:qp}
    \begin{aligned}[c]
    \min_{x\in \mathbb R^n}~~ &~ \frac12 x^\top Qx+c^\top x \\
\text{s.t.}~~ &~ Ax\leq b \ ,% \\
% & ~ x\geq 0 \ ,
    \end{aligned}
\end{equation}
where $Q\in \mathbb R^{n\times n}$ is a positive semi-definite matrix, $A\in\mathbb R^{m\times n}$, $b\in \mathbb R^m$ and $c\in \mathbb R^n$.

{OSQP is a first-order method solver designed for convex QP based on the Alternating Direction Method of Multipliers (ADMM)~\cite{stellato2020osqp}. The GPU implementation of OSQP incorporates a conjugate gradient (CG) method on GPUs for solving the linear systems that arise in each iteration. However, OSQP's GPU performance does not always surpass its CPU counterpart. The primary reason lies in the frequent CPU-GPU communications during the solving of linear systems. The communication latency is likely to diminish the expected speedup.}

In \cite{lu2023practical}, the authors design and analyze a first-order method for QP, dubbed restarted accelerated primal-dual hybrid gradient (rAPDHG) to solve the primal-dual form of \eqref{eq:qp}:
\begin{equation*}
    \min_x\max_{y\geq 0}\ \frac{1}{2}x^\top Qx+c^\top x+y^\top Ax-b^\top y \ .
\end{equation*}
rAPDHG adds two major enhancements, restart and momentum, upon vanilla PDHG. Hereby we elaborate the intuition of these two enhancements: the matrix $Q$ is usually not full rank, and there are two orthogonal subspaces in the primal space: a linear subspace $\text{ker}(Q)$, and a quadratic subspace $\text{range}(Q)$. Along the linear subspace, the problem is essentially an LP, and restarted PDHG achieves the optimal rate. Along the quadratic subspace, the problem is quadratic; an optimal algorithm should utilize momentum/acceleration, for example, accelerated PDHG~\cite{chen2014optimal}. Thus, both restarting and acceleration are essential.

A GPU-based QP solver \href{https://github.com/jinwen-yang/PDQP.jl}{PDQP} is developed based on rAPDHG~\cite{lu2023practical}. Numerical experiments on standard benchmark datasets and large-scale synthetic instances demonstate superior performance of PDQP over SCS and OSQP on larger instances.

In \cite{huang2024restarted}, the authors introduce a restarted primal-dual hybrid conjugate gradient (\href{https://github.com/Huangyc98/PDHCG.jl}{PDHCG}) method for solving \eqref{eq:qp}, which incorporates conjugate gradient (CG) techniques to address the primal subproblems inexactly. It is demonstrated that PDHCG maintains a linear convergence rate with an improved convergence constant while its GPU implementation shows superior empirical performances compared to PDQP.

\subsection{Semi-Definite Programming}\label{sec:sdp}
Semi-definite programming (SDP) is another fundamental class of optimization problems with wide-ranging applications. Despite its versatility, solving large-scale SDP problems has long been hindered by significant computational and memory demands, especially for dense matrix formulations and expensive computations. More formally, consider SDP of the form
\begin{equation}\label{eq:sdp}
    \begin{aligned}[c]
    \min_{X\in \mathbb R^{n\times n}}~~ &~ \mathrm{tr}(CX) \\
\text{s.t.}~~ &~ \mathcal A(X)= b \\
& ~ X\succeq 0 \ ,
    \end{aligned}
\end{equation}
where $C\in\mathbb R^{n\times n}$ is a symmetric matrix, $b\in \mathbb R^m$ and $\mathcal A:\mathbb R^{n\times n}\rightarrow \mathbb R^m$ is a linear functional, mapping a matrix variable to a vector.

One breakthrough on solving large-scale SDP is the introduction of low-rank framework, i.e., Burer-Monteiro factorization~\cite{burer2003nonlinear,burer2005local}. This approach factorizes the matrix variable $X=FF^\top$, where $F\in\mathbb R^{n\times k}$ while rank $k$ is significantly smaller than the number of variables $n$. Despite the introduction of non-convexity, the adoption of Burer-Monteiro approach significantly reduces memory usage by storing only the low-rank factor, and also reduces computational cost, eliminating the need of projections onto positive semi-definite cones.

In~\cite{han2024accelerating}, a new GPU-based SDP solver, \href{https://github.com/COPT-Public/cuLoRADS}{cuLoRADS}, is introduced. cuLoRADS uses the Burer-Monteiro factorization together with ALM~\cite{burer2006computational}/ADMM~\cite{han2024low} frameworks. The solver runs low-rank ALM during initial stage and switches to low-rank ADMM once the solution is near-optimal. Leveraging GPU parallelism, cuLoRADS achieves remarkable scalability, solving certain SDPs with matrix dimensions as large as 170 million $\times$ 170 million within minutes. More recently, there emerges other GPU-based solvers such as cuHALLaR~\cite{aguirre2025cuhallar} and ALORA~\cite{ding2025new} for solving low-rank SDP via ALM framework.

\subsection{Conic Programming and Nonlinear Programming}\label{sec:ipm}
In this section, we review recent advancements in GPU-based solvers for solving conic programming and nonlinear programming. We begin with FOM-based solver PDCS for conic linear programming, followed by discussion of IPM-based solvers CuClarabel for conic quadratic programming, and MadNLP for nonlinear programming. 

{\href{https://www.cvxgrp.org/scs/}{SCS}~\cite{o2016conic,o2021operator} is a first-order method solver for large-scale convex cone programs:
\begin{equation}\label{eq:qlp}
    \begin{aligned}[c]
    \min_{x\in \mathbb R^n}~~ &~ \frac{1}{2}x^\top Qx+c^\top x \\
\text{s.t.}~~ &~ Ax=b \\
& ~ x\in\mathcal K \ ,
    \end{aligned}
\end{equation}
SCS is based on the ADMM applied to the homogeneous self-dual embedding of the primal-dual pair. This embedding allows SCS to provide not only primal-dual solutions but also certificates of infeasibility or unboundedness. Its GPU implementation supports acceleration via iterative methods such as conjugate gradient when solving the linear systems arising in each ADMM iteration. However, the performance gains from GPU acceleration can often be limited by data transfer overhead between the CPU and GPU.}

{\href{https://github.com/ZikaiXiong/PDCS}{PDCS}~\cite{lin2025pdcs} is also a GPU-based solver tailored for large-scale conic linear programming, i.e., with $Q=0$ in \eqref{eq:qlp}.
Built upon PDHG, PDCS incorporates several practical enhancements, including adaptive Halpern restarts, diagonal rescaling, and efficient projection algorithms based on bisection. It supports a broad range of cones such as nonnegative, second-order, exponential, and their Cartesian products. The GPU implementation, cuPDCS, leverages customized parallelism strategies at the grid, block, and thread levels to accelerate matrix-vector operations and cone projections. cuPDCS achieves strong scalability and performance on diverse conic applications, with numerical experiments demonstrating its efficiency and robustness across benchmarks such as Fisher markets, Lasso regression, and portfolio optimization.

Additionally, there are recent advancements in GPU-based solvers based on interior-point methods (IPMs). While IPMs are not classified as first-order methods and thus fall outside the primary scope of this work, we believe that providing an overview here can offer valuable insights and inspire future developments in the broader field of GPU-based solvers.}

\href{https://github.com/cvxgrp/CuClarabel}{CuClarabel}~\cite{chen2024cuclarabel} is a GPU implementation of \href{https://clarabel.org/stable/}{Clarabel}~\cite{goulart2024clarabel}, an open-source IPM-based solver for conic programming.
(Cu)Clarabel supports solving multiple cones, including second-order cones, positive semi-definite cones and exponential and power cones. To efficiently handle the linear systems arising in the IPM iterations, CuClarabel leverages \href{https://docs.nvidia.com/cuda/cudss/index.html}{cuDSS}, a recently released CUDA direct sparse system solver. This GPU-based approach results in significant performance improvements, with CuClarabel demonstrating remarkable speedups compared to the CPU-based Clarabel.

\href{https://github.com/MadNLP/MadNLP.jl}{MadNLP}~\cite{shin2023accelerating,shin2020graph} is a GPU-based nonlinear programming (NLP) solver designed to address the challenges of large-scale nonlinear optimization problems:
\begin{equation}\label{eq:nlp}
    \begin{aligned}[c]
    \min_{x\in \mathbb R^n}~~ &~ f(x) \\
\text{s.t.}~~ &~ g(x)\leq 0  \ ,
    \end{aligned}
\end{equation}
MadNLP utilizes a condensed-space interior-point method~\cite{pacaud2024accelerating} to reduce the reliance on numerical pivoting, which is often a computational bottleneck in traditional IPM implementations. By streamlining the handling of equality and inequality constraints, MadNLP ensures robust and efficient numerical performance. Leveraging the parallel computing capabilities of modern GPUs, MadNLP delivers substantial speedups compared to CPU-based solvers such as IPOPT on large-scale applications such as power systems optimization, where the solver demonstrates an order-of-magnitude improvement in efficiency. By maintaining critical computational workflows entirely within GPU memory and optimizing data movement, MadNLP excels in solving complex, large-scale problems with remarkable speed and scalability.

\section{Conclusions and Open Questions}
This survey provides an overview of recent advancements in GPU-based solvers for mathematical programming, with a primary focus on linear programming and its extensions. We begin by comparing CPU and GPU architectures, emphasizing the significant speedups that GPUs offer for matrix-vector multiplications.
We then delve into the design of first-order methods (FOMs) for LP, weighing the advantages and disadvantages of various algorithms, ultimately leading to the PDHG algorithm. Next, we discuss practical enhancements beyond vanilla PDHG, followed by empirical comparisons of cuPDLP, its CPU counterpart PDLP, and the commercial solver Gurobi, demonstrating how GPUs can serve as powerful engines for large-scale LPs.
Additionally, we provide an overview of the theoretical foundations of PDHG for LP, drawing from the authors’ own insights into these problems. Finally, we explore recent developments in GPU-based solvers for quadratic, conic, semidefinite, and nonlinear programming, showcasing the broader landscape of GPU-based mathematical optimization. Just as GPUs have revolutionized deep learning by efficiently handling vast amounts of parallel computations, they are now reshaping the field of optimization by unlocking new levels of speed and scalability.

Despite the significant recent advancements, challenges and open questions persist in the development of GPU-based first-order methods for linear programming and beyond. Some of the key challenges include:
\begin{itemize}
    \item {\bf Adaptive step-size.} Step-size selection plays a crucial role in accelerating the convergence of first-order methods for large-scale linear programming. While constant step-size ensures theoretical convergence guarantees, it often underperforms compared to adaptive heuristics in practice. Conversely, existing adaptive heuristics, such as those used in PDLP, exhibit strong empirical performance but lack rigorous worst-case complexity guarantees. Developing a new adaptive step-size strategy that enjoys provable guarantees and practical efficiency could bridge this gap, offering both theoretical soundness and robust performance.
    \item {\bf Adaptive diagonal preconditioning.} Diagonal preconditioning is essential for improving the practical convergence of first-order methods by rescaling variables and constraints to mitigate ill-conditioning. However, most existing techniques, such as Ruiz equilibration and Chambolle-Pock scaling in PDLP, are applied only once before optimization begins, without adapting to changes in problem structure during iterations. An open question is how to develop an adaptive diagonal preconditioning scheme that dynamically adjusts throughout optimization to further enhance convergence. Key challenges include designing efficient and stable update rules that avoid excessive computational overhead, ensuring that adaptive scaling does not introduce oscillations or instability. A well-designed adaptive preconditioner could lead to faster convergence and improved numerical stability for FOM-based LP solvers, particularly in highly ill-conditioned or imbalanced problem instances.
    \item {\bf Presolve tailored for FOMs.} As an essential component of mature LP solvers, presolving simplifies problem instances by reducing dimensionality, detecting infeasibility, and improving numerical stability. Traditional presolve techniques are designed for simplex and interior-point methods, where factorization-based operations can efficiently eliminate redundant constraints and variables. However, first-order methods (FOMs) rely solely on matrix-vector multiplications and are highly sensitive to conditioning and sparsity structure, making standard presolve techniques less directly applicable. An open question is how to design presolve strategies specifically tailored for FOMs, ensuring that reductions do not introduce costly projections or complicate the iterative updates. A principled approach to presolving for FOMs could significantly enhance their efficiency and scalability for large-scale LP problems.
    \item {\bf GPU-based crossover.} In traditional LP solvers, the crossover procedure converts the approximate solution obtained by first-order or interior-point methods into a basic feasible solution (BFS), which is crucial for warm-starting simplex-based refinement or integrating with mixed-integer programming (MIP) solvers. However, existing crossover techniques rely heavily on factorization and pivoting, which are inherently sequential and memory-intensive, making them inefficient for GPU architectures. An open question is how to design a GPU-friendly crossover procedure that efficiently transitions solutions from first-order methods to a BFS while leveraging parallel computation. A scalable, GPU-optimized crossover strategy would enable seamless integration of FOMs with traditional MIP solvers, improving both solution accuracy and solver interoperability in large-scale optimization tasks.
\end{itemize}

As GPU technology continues to evolve, further research and innovation will be essential to fully unlock its potential for large-scale mathematical programming. Advancements in algorithm design and computing techniques will play a crucial role in overcoming existing challenges and expanding the applicability of GPU-based solvers. By bridging the gap between hardware architectures and mathematical programming needs, it is promising that GPUs could become a transformative force in large-scale optimization, enabling the solution of increasingly complex problems across diverse domains.

\section*{Acknowledgements}

The authors would like to thank many colleagues for their valuable feedback that shapes this paper. In particular, we are deeply grateful to Robert M. Freund for carefully proofreading the manuscript and offering numerous insightful comments and suggestions across various sections. We are  indebted to Miles Lubin, who essentially provided a ``referee report'', highlighting many ways to improve the rigor and overall quality of the paper. We also thank Ed Rothberg for his substantial feedback on modern LP solvers and for the stimulating discussion on the fundamental advantages of GPUs for different LP algorithms, which ultimately led to a complete rewrite of the introduction. 

We would additionally like to thank David Applegate, Burcin Bozkaya, Mateo Diaz, Dongdong Ge, Lin Gui, Alex Fender, Oliver Hinder, Qi Huangfu, Tianhao Liu, Chris Maes, Brendan O'Donoghue, Zedong Peng, Warren Schudy, Zikai Xiong, and Yinyu Ye for their support and helpful thoughts and feedback at various stages of developing this survey.

As a rapidly evolving area, many of the references cited and discussed in this survey are available as preprints or as research blogs, including several by the authors. Most of these preprints are in the review process for formal publication.

Haihao Lu is supported by AFOSR Grant No. FA9550-24-1-0051 and ONR Grant No. N000142412735. Jinwen Yang is supported by AFOSR Grant No. FA9550-24-1-0051.

\bibliographystyle{amsplain}
\bibliography{ref-papers}

\providecommand{\bysame}{\leavevmode\hbox to3em{\hrulefill}\thinspace}
\providecommand{\MR}{\relax\ifhmode\unskip\space\fi MR }
\providecommand{\MRhref}[2]{%
  \href{http://www.ams.org/mathscinet-getitem?mr=#1}{#2}
}
\providecommand{\href}[2]{#2}
\begin{thebibliography}{100}

\bibitem{aguirre2025cuhallar}
Jacob~M Aguirre, Diego Cifuentes, Vincent Guigues, Renato~DC Monteiro, Victor~Hugo Nascimento, and Arnesh Sujanani, \emph{cuhallar: A gpu accelerated low-rank augmented lagrangian method for large-scale semidefinite programming}, arXiv preprint arXiv:2505.13719 (2025).

\bibitem{ahuja1988network}
Ravindra~K Ahuja, Thomas~L Magnanti, and James~B Orlin, \emph{Network flows},  (1988).

\bibitem{andersen2003implementing}
Erling~D Andersen, Cees Roos, and Tamas Terlaky, \emph{On implementing a primal-dual interior-point method for conic quadratic optimization}, Mathematical Programming \textbf{95} (2003), 249--277.

\bibitem{applegate2021practical}
David Applegate, Mateo D{\'\i}az, Oliver Hinder, Haihao Lu, Miles Lubin, Brendan O'Donoghue, and Warren Schudy, \emph{Practical large-scale linear programming using primal-dual hybrid gradient}, Advances in Neural Information Processing Systems \textbf{34} (2021), 20243--20257.

\bibitem{applegate2025pdlp}
\bysame, \emph{Pdlp: A practical first-order method for large-scale linear programming}, arXiv preprint arXiv:2501.07018 (2025).

\bibitem{applegate2024infeasibility}
David Applegate, Mateo D{\'\i}az, Haihao Lu, and Miles Lubin, \emph{Infeasibility detection with primal-dual hybrid gradient for large-scale linear programming}, SIAM Journal on Optimization \textbf{34} (2024), no.~1, 459--484.

\bibitem{applegate2023faster}
David Applegate, Oliver Hinder, Haihao Lu, and Miles Lubin, \emph{Faster first-order primal-dual methods for linear programming using restarts and sharpness}, Mathematical Programming \textbf{201} (2023), no.~1, 133--184.

\bibitem{bauschke2017correction}
Heinz~H Bauschke, Patrick~L Combettes, Heinz~H Bauschke, and Patrick~L Combettes, \emph{Correction to: convex analysis and monotone operator theory in hilbert spaces}, Springer, 2017.

\bibitem{bauschke2019convex}
HH~Bauschke and PL~Combettes, \emph{Convex analysis and monotone operator theory in hilbert spaces, corrected printing}, 2019.

\bibitem{bazaraa2011linear}
Mokhtar~S Bazaraa, John~J Jarvis, and Hanif~D Sherali, \emph{Linear programming and network flows}, John Wiley \& Sons, 2011.

\bibitem{beck2017first}
Amir Beck, \emph{First-order methods in optimization}, SIAM, 2017.

\bibitem{belloni2009geometric}
Alexandre Belloni and Robert~M Freund, \emph{A geometric analysis of renegar’s condition number, and its interplay with conic curvature}, Mathematical programming \textbf{119} (2009), no.~1, 95--107.

\bibitem{bertsimas1997introduction}
Dimitris Bertsimas and John~N Tsitsiklis, \emph{Introduction to linear optimization}, vol.~6, Athena Scientific Belmont, MA, 1997.

\bibitem{besard2018effective}
Tim Besard, Christophe Foket, and Bjorn De~Sutter, \emph{Effective extensible programming: unleashing julia on gpus}, IEEE Transactions on Parallel and Distributed Systems \textbf{30} (2018), no.~4, 827--841.

\bibitem{bezanson2017julia}
Jeff Bezanson, Alan Edelman, Stefan Karpinski, and Viral~B Shah, \emph{Julia: A fresh approach to numerical computing}, SIAM review \textbf{59} (2017), no.~1, 65--98.

\bibitem{bland1981ellipsoid}
Robert~G Bland, Donald Goldfarb, and Michael~J Todd, \emph{The ellipsoid method: A survey}, Operations research \textbf{29} (1981), no.~6, 1039--1091.

\bibitem{nvidianews}
Nicolas Blin, \emph{Accelerate large linear programming problems with nvidia cuopt}, 2024, \url{https://developer.nvidia.com/blog/accelerate-large-linear-programming-problems-with-nvidia-cuopt/}.

\bibitem{bnnobrs1962partitioning}
J~BnnoBRs, \emph{Partitioning procedures for solving mixed-variables programming problems}, Numer. Math \textbf{4} (1962), no.~1, 238--252.

\bibitem{boyd2011distributed}
Stephen Boyd, Neal Parikh, Eric Chu, Borja Peleato, Jonathan Eckstein, et~al., \emph{Distributed optimization and statistical learning via the alternating direction method of multipliers}, Foundations and Trends{\textregistered} in Machine learning \textbf{3} (2011), no.~1, 1--122.

\bibitem{boyd2004convex}
Stephen Boyd and Lieven Vandenberghe, \emph{Convex optimization}, Cambridge university press, 2004.

\bibitem{brown1951computational}
George Brown and Tjalling Koopmans, \emph{Computational suggestions for maximizing a linear function subject to linear inequalities}, Activity Analysis of Production and Allocation (1951), 377--380.

\bibitem{burer2006computational}
Samuel Burer and Changhui Choi, \emph{Computational enhancements in low-rank semidefinite programming}, Optimisation Methods and Software \textbf{21} (2006), no.~3, 493--512.

\bibitem{burer2003nonlinear}
Samuel Burer and Renato~DC Monteiro, \emph{A nonlinear programming algorithm for solving semidefinite programs via low-rank factorization}, Mathematical programming \textbf{95} (2003), no.~2, 329--357.

\bibitem{burer2005local}
\bysame, \emph{Local minima and convergence in low-rank semidefinite programming}, Mathematical programming \textbf{103} (2005), no.~3, 427--444.

\bibitem{chambolle2011first}
Antonin Chambolle and Thomas Pock, \emph{A first-order primal-dual algorithm for convex problems with applications to imaging}, Journal of mathematical imaging and vision \textbf{40} (2011), 120--145.

\bibitem{chambolle2016ergodic}
\bysame, \emph{On the ergodic convergence rates of a first-order primal--dual algorithm}, Mathematical Programming \textbf{159} (2016), no.~1-2, 253--287.

\bibitem{chang1989steepest}
Soo Chang and Katta Murty, \emph{The steepest descent gravitational method for linear programming}, Discrete Applied Mathematics \textbf{25} (1989), no.~3, 211--239.

\bibitem{charnes1959application}
Abraham Charnes, William~W Cooper, and Merton~H Miller, \emph{Application of linear programming to financial budgeting and the costing of funds}, The Journal of Business \textbf{32} (1959), no.~1, 20--46.

\bibitem{chen2024hpr}
Kaihuang Chen, Defeng Sun, Yancheng Yuan, Guojun Zhang, and Xinyuan Zhao, \emph{Hpr-lp: An implementation of an hpr method for solving linear programming}, arXiv preprint arXiv:2408.12179 (2024).

\bibitem{chen2014optimal}
Yunmei Chen, Guanghui Lan, and Yuyuan Ouyang, \emph{Optimal primal-dual methods for a class of saddle point problems}, SIAM Journal on Optimization \textbf{24} (2014), no.~4, 1779--1814.

\bibitem{chen2024cuclarabel}
Yuwen Chen, Danny Tse, Parth Nobel, Paul Goulart, and Stephen Boyd, \emph{Cuclarabel: Gpu acceleration for a conic optimization solver}, arXiv preprint arXiv:2412.19027 (2024).

\bibitem{cheung2003unifying}
Dennis Cheung, Felipe Cucker, and Javier Pena, \emph{Unifying condition numbers for linear programming}, Mathematics of Operations Research \textbf{28} (2003), no.~4, 609--624.

\bibitem{cipra2000best}
Barry~A Cipra, \emph{The best of the 20th century: Editors name top 10 algorithms}, SIAM news \textbf{33} (2000), no.~4, 1--2.

\bibitem{dahl2022primal}
Joachim Dahl and Erling~D Andersen, \emph{A primal-dual interior-point algorithm for nonsymmetric exponential-cone optimization}, Mathematical Programming \textbf{194} (2022), no.~1, 341--370.

\bibitem{dahleh1994control}
Munther Dahleh and Ignacio Diaz-Bobillo, \emph{Control of uncertain systems: a linear programming approach}, Prentice-Hall, Inc., 1994.

\bibitem{dantzig1963linear}
George Dantzig, \emph{Linear programming and extensions}, Princeton university press, 1963.

\bibitem{dantzig1948programming}
George~B Dantzig, \emph{Programming in a linear structure}, Washington, DC (1948).

\bibitem{dantzig1960decomposition}
George~B Dantzig and Philip Wolfe, \emph{Decomposition principle for linear programs}, Operations research \textbf{8} (1960), no.~1, 101--111.

\bibitem{davis2015convergence}
Damek Davis, \emph{Convergence rate analysis of primal-dual splitting schemes}, SIAM Journal on Optimization \textbf{25} (2015), no.~3, 1912--1943.

\bibitem{davis2016convergence}
Damek Davis and Wotao Yin, \emph{Convergence rate analysis of several splitting schemes}, Splitting methods in communication, imaging, science, and engineering (2016), 115--163.

\bibitem{de2024power}
Antonio De~Rosa, Aida Khajavirad, and Yakun Wang, \emph{On the power of linear programming for k-means clustering}, arXiv preprint arXiv:2402.01061 (2024).

\bibitem{delson1992linear}
Jerome Delson and Mohammad Shahidehpour, \emph{Linear programming applications to power system economics, planning and operations}, IEEE Transactions on Power Systems \textbf{7} (1992), no.~3, 1155--1163.

\bibitem{diakonikolas2020halpern}
Jelena Diakonikolas, \emph{Halpern iteration for near-optimal and parameter-free monotone inclusion and strong solutions to variational inequalities}, Conference on Learning Theory, PMLR, 2020, pp.~1428--1451.

\bibitem{dikin1967iterative}
Iliya~Iosiphovich Dikin, \emph{Iterative solution of problems of linear and quadratic programming}, Soviet Math. Dokl., vol.~8, 1967, pp.~674--675.

\bibitem{ding2025new}
Lijun Ding, Haihao Lu, and Jinwen Yang, \emph{New understandings and computation on augmented lagrangian methods for low-rank semidefinite programming}, arXiv preprint arXiv:2505.15775 (2025).

\bibitem{dontchev2004regularity}
Asen~L Dontchev and R~Tyrrell Rockafellar, \emph{Regularity and conditioning of solution mappings in variational analysis}, Set-Valued Analysis \textbf{12} (2004), 79--109.

\bibitem{dontchev2009implicit}
\bysame, \emph{Implicit functions and solution mappings}, vol. 543, Springer, 2009.

\bibitem{drusvyatskiy2013tilt}
Dmitriy Drusvyatskiy and Adrian~S Lewis, \emph{Tilt stability, uniform quadratic growth, and strong metric regularity of the subdifferential}, SIAM Journal on Optimization \textbf{23} (2013), no.~1, 256--267.

\bibitem{drusvyatskiy2018error}
\bysame, \emph{Error bounds, quadratic growth, and linear convergence of proximal methods}, Mathematics of Operations Research \textbf{43} (2018), no.~3, 919--948.

\bibitem{drusvyatskiy2013second}
Dmitriy Drusvyatskiy, Boris~S Mordukhovich, and Tran~TA Nghia, \emph{Second-order growth, tilt stability, and metric regularity of the subdifferential}, arXiv preprint arXiv:1304.7385 (2013).

\bibitem{eckstein1992douglas}
Jonathan Eckstein and Dimitri~P Bertsekas, \emph{On the douglas—rachford splitting method and the proximal point algorithm for maximal monotone operators}, Mathematical programming \textbf{55} (1992), 293--318.

\bibitem{epelman2000condition}
Marina Epelman and Robert~M Freund, \emph{Condition number complexity of an elementary algorithm for computing a reliable solution of a conic linear system}, Mathematical Programming \textbf{88} (2000), no.~3, 451--485.

\bibitem{farkas1902theorie}
Julius Farkas, \emph{Theorie der einfachen ungleichungen.}, Journal f{\"u}r die reine und angewandte Mathematik (Crelles Journal) \textbf{1902} (1902), no.~124, 1--27.

\bibitem{fercoq2022quadratic}
Olivier Fercoq, \emph{Quadratic error bound of the smoothed gap and the restarted averaged primal-dual hybrid gradient}, arXiv preprint arXiv:2206.03041 (2022).

\bibitem{fiacco1964computational}
Anthony~V Fiacco and Garth~P McCormick, \emph{Computational algorithm for the sequential unconstrained minimization technique for nonlinear programming}, Management Science \textbf{10} (1964), no.~4, 601--617.

\bibitem{fiacco1964sequential}
\bysame, \emph{The sequential unconstrained minimization technique for nonlinear programing, a primal-dual method}, Management Science \textbf{10} (1964), no.~2, 360--366.

\bibitem{forrest1992steepest}
John~J Forrest and Donald Goldfarb, \emph{Steepest-edge simplex algorithms for linear programming}, Mathematical programming \textbf{57} (1992), no.~1, 341--374.

\bibitem{freund1994professor}
Robert Freund, \emph{Professor george dantzig: Linear programming founder turns 80}, SIAM News, November (1994).

\bibitem{freund2003primal}
Robert~M Freund, \emph{On the primal-dual geometry of level sets in linear and conic optimization}, SIAM Journal on Optimization \textbf{13} (2003), no.~4, 1004--1013.

\bibitem{freund2004complexity}
\bysame, \emph{Complexity of convex optimization using geometry-based measures and a reference point}, Mathematical Programming \textbf{99} (2004), 197--221.

\bibitem{freund1999some}
Robert~M Freund and Jorge~R Vera, \emph{Some characterizations and properties of the “distance to ill-posedness” and the condition measure of a conic linear system}, Mathematical Programming \textbf{86} (1999), no.~2, 225--260.

\bibitem{gale1951linear}
David Gale, Harold~W Kuhn, and Albert~W Tucker, \emph{Linear programming and the theory of games}, Activity analysis of production and allocation \textbf{13} (1951), 317--335.

\bibitem{gleixner2021miplib}
Ambros Gleixner, Gregor Hendel, Gerald Gamrath, Tobias Achterberg, Michael Bastubbe, Timo Berthold, Philipp~M. Christophel, Kati Jarck, Thorsten Koch, Jeff Linderoth, Marco L\"ubbecke, Hans~D. Mittelmann, Derya Ozyurt, Ted~K. Ralphs, Domenico Salvagnin, and Yuji Shinano, \emph{{MIPLIB 2017: Data-Driven Compilation of the 6th Mixed-Integer Programming Library}}, Mathematical Programming Computation (2021).

\bibitem{goulart2024clarabel}
Paul~J Goulart and Yuwen Chen, \emph{Clarabel: An interior-point solver for conic programs with quadratic objectives}, arXiv preprint arXiv:2405.12762 (2024).

\bibitem{googleblog}
David~Applegate Haihao~Lu, \emph{Scaling up linear programming with pdlp}, \url{https://research.google/blog/scaling-up-linear-programming-with-pdlp/}, 2024-09-20.

\bibitem{halpern1967fixed}
Benjamin Halpern, \emph{Fixed points of nonexpanding maps},  (1967).

\bibitem{han2024low}
Qiushi Han, Chenxi Li, Zhenwei Lin, Caihua Chen, Qi~Deng, Dongdong Ge, Huikang Liu, and Yinyu Ye, \emph{A low-rank admm splitting approach for semidefinite programming}, arXiv preprint arXiv:2403.09133 (2024).

\bibitem{han2024accelerating}
Qiushi Han, Zhenwei Lin, Hanwen Liu, Caihua Chen, Qi~Deng, Dongdong Ge, and Yinyu Ye, \emph{Accelerating low-rank factorization-based semidefinite programming algorithms on gpu}, arXiv preprint arXiv:2407.15049 (2024).

\bibitem{hazell1974competitive}
Peter Hazell and Pasquale Scandizzo, \emph{Competitive demand structures under risk in agricultural linear programming models}, American Journal of Agricultural Economics \textbf{56} (1974), no.~2, 235--244.

\bibitem{he2012convergence}
Bingsheng He and Xiaoming Yuan, \emph{Convergence analysis of primal-dual algorithms for a saddle-point problem: from contraction perspective}, SIAM Journal on Imaging Sciences \textbf{5} (2012), no.~1, 119--149.

\bibitem{he20121}
\bysame, \emph{On the o(1/n) convergence rate of the douglas--rachford alternating direction method}, SIAM Journal on Numerical Analysis \textbf{50} (2012), no.~2, 700--709.

\bibitem{hinder2024worst}
Oliver Hinder, \emph{Worst-case analysis of restarted primal-dual hybrid gradient on totally unimodular linear programs}, Operations Research Letters \textbf{57} (2024), 107199.

\bibitem{hoffman1952approximate}
Alan~J Hoffman, \emph{On approximate solutions of systems of linear inequalities}, Journal of Research of the National Bureau of Standards \textbf{49} (1952), 263--265.

\bibitem{huang2024restarted}
Yicheng Huang, Wanyu Zhang, Hongpei Li, Dongdong Ge, Huikang Liu, and Yinyu Ye, \emph{Restarted primal-dual hybrid conjugate gradient method for large-scale quadratic programming}, arXiv preprint arXiv:2405.16160 (2024).

\bibitem{kantarovich1939mathematical}
LV~Kantarovich, \emph{Mathematical methods in the organization and planning of production}, Publication House of the Leningrad State University.[Translated in Management Sc. vol 66, 366-422] (1939).

\bibitem{karmarkar1984new}
Narendra Karmarkar, \emph{A new polynomial-time algorithm for linear programming}, Proceedings of the sixteenth annual ACM symposium on Theory of computing, 1984, pp.~302--311.

\bibitem{kempke2025low}
Nils-Christian Kempke and Thorsten Koch, \emph{Low-precision first-order method-based fix-and-propagate heuristics for large-scale mixed-integer linear optimization}, arXiv preprint arXiv:2503.10344 (2025).

\bibitem{khachiyan1980polynomial}
Leonid~G Khachiyan, \emph{Polynomial algorithms in linear programming}, USSR Computational Mathematics and Mathematical Physics \textbf{20} (1980), no.~1, 53--72.

\bibitem{kim2021accelerated}
Donghwan Kim, \emph{Accelerated proximal point method for maximally monotone operators}, Mathematical Programming \textbf{190} (2021), no.~1, 57--87.

\bibitem{klee1972good}
Victor Klee and George~J Minty, \emph{How good is the simplex algorithm}, Inequalities \textbf{3} (1972), no.~3, 159--175.

\bibitem{koberstein2005dual}
Achim Koberstein, \emph{The dual simplex method, techniques for a fast and stable implementation}, Ph.D. thesis, Paderborn University, Germany, 2005.

\bibitem{korte2011combinatorial}
Bernhard~H Korte, Jens Vygen, B~Korte, and J~Vygen, \emph{Combinatorial optimization}, vol.~1, Springer, 2011.

\bibitem{lemke1961constrained}
Carlton Lemke, \emph{The constrained gradient method of linear programming}, Journal of the Society for Industrial and Applied Mathematics \textbf{9} (1961), no.~1, 1--17.

\bibitem{lewis2016proximal}
Adrian~S Lewis and Stephen~J Wright, \emph{A proximal method for composite minimization}, Mathematical Programming \textbf{158} (2016), no.~1, 501--546.

\bibitem{liang2017activity}
Jingwei Liang, Jalal Fadili, and Gabriel Peyr{\'e}, \emph{Activity identification and local linear convergence of forward--backward-type methods}, SIAM Journal on Optimization \textbf{27} (2017), no.~1, 408--437.

\bibitem{liang2017local}
\bysame, \emph{Local convergence properties of {D}ouglas--{R}achford and alternating direction method of multipliers}, Journal of Optimization Theory and Applications \textbf{172} (2017), no.~3, 874--913.

\bibitem{liang2018local}
\bysame, \emph{Local linear convergence analysis of primal--dual splitting methods}, Optimization \textbf{67} (2018), no.~6, 821--853.

\bibitem{lieder2021convergence}
Felix Lieder, \emph{On the convergence rate of the halpern-iteration}, Optimization letters \textbf{15} (2021), no.~2, 405--418.

\bibitem{lin2025pdcs}
Zhenwei Lin, Zikai Xiong, Dongdong Ge, and Yinyu Ye, \emph{Pdcs: A primal-dual large-scale conic programming solver with gpu enhancements}, arXiv preprint arXiv:2505.00311 (2025).

\bibitem{liu2008choice}
Qian Liu and Garrett Van~Ryzin, \emph{On the choice-based linear programming model for network revenue management}, Manufacturing \& Service Operations Management \textbf{10} (2008), no.~2, 288--310.

\bibitem{liu2024new}
Tianhao Liu and Haihao Lu, \emph{A new crossover algorithm for lp inspired by the spiral dynamic of pdhg}, arXiv preprint arXiv:2409.14715 (2024).

\bibitem{lu2024mpax}
Haihao Lu, Zedong Peng, and Jinwen Yang, \emph{Mpax: Mathematical programming in jax}, arXiv preprint arXiv:2412.09734 (2024).

\bibitem{lu2023optimizing}
Haihao Lu, Duncan Simester, and Yuting Zhu, \emph{Optimizing scalable targeted marketing policies with constraints}, arXiv preprint arXiv:2312.01035 (2023).

\bibitem{lu2022infimal}
Haihao Lu and Jinwen Yang, \emph{On the infimal sub-differential size of primal-dual hybrid gradient method}, arXiv preprint arXiv:2206.12061 (2022).

\bibitem{lu2023cupdlp}
\bysame, \emph{cupdlp. jl: A gpu implementation of restarted primal-dual hybrid gradient for linear programming in julia}, arXiv preprint arXiv:2311.12180 (2023).

\bibitem{lu2023unified}
\bysame, \emph{On a unified and simplified proof for the ergodic convergence rates of ppm, pdhg and admm}, arXiv preprint arXiv:2305.02165 (2023).

\bibitem{lu2023geometry}
\bysame, \emph{On the geometry and refined rate of primal-dual hybrid gradient for linear programming}, arXiv preprint arXiv:2307.03664 (2023).

\bibitem{lu2023practical}
\bysame, \emph{A practical and optimal first-order method for large-scale convex quadratic programming}, arXiv preprint arXiv:2311.07710 (2023).

\bibitem{lu2024pdot}
\bysame, \emph{Pdot: A practical primal-dual algorithm and a gpu-based solver for optimal transport}, arXiv preprint arXiv:2407.19689 (2024).

\bibitem{lu2024restarted}
\bysame, \emph{Restarted halpern pdhg for linear programming}, arXiv preprint arXiv:2407.16144 (2024).

\bibitem{lu2023cupdlpc}
Haihao Lu, Jinwen Yang, Haodong Hu, Qi~Huangfu, Jinsong Liu, Tianhao Liu, Yinyu Ye, Chuwen Zhang, and Dongdong Ge, \emph{cupdlp-c: A strengthened implementation of cupdlp for linear programming by c language}, arXiv preprint arXiv:2312.14832 (2023).

\bibitem{lu2024power}
Haihao Lu and Luyang Zhang, \emph{The power of linear programming in sponsored listings ranking: Evidence from field experiments}, arXiv preprint arXiv:2403.14862 (2024).

\bibitem{maros2002computational}
Istv{\'a}n Maros, \emph{Computational techniques of the simplex method}, vol.~61, Springer Science \& Business Media, 2002.

\bibitem{mehrotra1992implementation}
Sanjay Mehrotra, \emph{On the implementation of a primal-dual interior point method}, SIAM Journal on optimization \textbf{2} (1992), no.~4, 575--601.

\bibitem{monge1781founding}
Gaspard Monge, \emph{The founding fathers of optimal transport}, 1781.

\bibitem{nemirovski2004prox}
Arkadi Nemirovski, \emph{Prox-method with rate of convergence o (1/t) for variational inequalities with lipschitz continuous monotone operators and smooth convex-concave saddle point problems}, SIAM Journal on Optimization \textbf{15} (2004), no.~1, 229--251.

\bibitem{nesterov1997self}
Yu~E Nesterov and Michael~J Todd, \emph{Self-scaled barriers and interior-point methods for convex programming}, Mathematics of Operations research \textbf{22} (1997), no.~1, 1--42.

\bibitem{nesterov1998primal}
\bysame, \emph{Primal-dual interior-point methods for self-scaled cones}, SIAM Journal on optimization \textbf{8} (1998), no.~2, 324--364.

\bibitem{nesterov2003introductory}
Yurii Nesterov, \emph{Introductory lectures on convex optimization: A basic course}, vol.~87, Springer Science \& Business Media, 2003.

\bibitem{nesterov2018lectures}
Yurii Nesterov et~al., \emph{Lectures on convex optimization}, vol. 137, Springer, 2018.

\bibitem{nesterov1994interior}
Yurii Nesterov and Arkadi Nemirovskii, \emph{Interior-point polynomial algorithms in convex programming}, SIAM, 1994.

\bibitem{nvidiahpc}
NVIDIA, \emph{High-performance computing: Accelerating the rate of scientific discovery.}, \url{https://www.nvidia.com/en-us/high-performance-computing/}.

\bibitem{nvidiaai}
\bysame, \emph{Why gpus are great for ai.}, \url{https://blogs.nvidia.com/blog/why-gpus-are-great-for-ai/}.

\bibitem{o2021operator}
Brendan O'Donoghue, \emph{Operator splitting for a homogeneous embedding of the linear complementarity problem}, SIAM Journal on Optimization \textbf{31} (2021), no.~3, 1999--2023.

\bibitem{o2020equivalence}
Daniel O’Connor and Lieven Vandenberghe, \emph{On the equivalence of the primal-dual hybrid gradient method and douglas--rachford splitting}, Mathematical Programming \textbf{179} (2020), no.~1, 85--108.

\bibitem{o2016conic}
Brendan O’donoghue, Eric Chu, Neal Parikh, and Stephen Boyd, \emph{Conic optimization via operator splitting and homogeneous self-dual embedding}, Journal of Optimization Theory and Applications \textbf{169} (2016), 1042--1068.

\bibitem{pacaud2024accelerating}
Fran{\c{c}}ois Pacaud, Sungho Shin, Michel Schanen, Daniel~Adrian Maldonado, and Mihai Anitescu, \emph{Accelerating condensed interior-point methods on simd/gpu architectures}, Journal of Optimization Theory and Applications \textbf{202} (2024), no.~1, 184--203.

\bibitem{park2022exact}
Jisun Park and Ernest~K Ryu, \emph{Exact optimal accelerated complexity for fixed-point iterations}, International Conference on Machine Learning, PMLR, 2022, pp.~17420--17457.

\bibitem{pazy1971asymptotic}
Amnon Pazy, \emph{Asymptotic behavior of contractions in hilbert space}, Israel Journal of Mathematics \textbf{9} (1971), 235--240.

\bibitem{pena2000understanding}
Javier Pena, \emph{Understanding the geometry of infeasible perturbations of a conic linear system}, SIAM Journal on Optimization \textbf{10} (2000), no.~2, 534--550.

\bibitem{pena2021new}
Javier Pena, Juan~C Vera, and Luis~F Zuluaga, \emph{New characterizations of hoffman constants for systems of linear constraints}, Mathematical Programming \textbf{187} (2021), 79--109.

\bibitem{pena2024easily}
Javier~F Pe{\~n}a, \emph{An easily computable upper bound on the hoffman constant for homogeneous inequality systems}, Computational Optimization and Applications \textbf{87} (2024), no.~1, 323--335.

\bibitem{peyre2019computational}
Gabriel Peyr{\'e}, Marco Cuturi, et~al., \emph{Computational optimal transport: With applications to data science}, Foundations and Trends{\textregistered} in Machine Learning \textbf{11} (2019), no.~5-6, 355--607.

\bibitem{pock2011diagonal}
Thomas Pock and Antonin Chambolle, \emph{Diagonal preconditioning for first order primal-dual algorithms in convex optimization}, 2011 International Conference on Computer Vision, IEEE, 2011, pp.~1762--1769.

\bibitem{poon2020geometry}
Clarice Poon and Jingwei Liang, \emph{Geometry of first-order methods and adaptive acceleration}, arXiv preprint arXiv:2003.03910 (2020).

\bibitem{renegar1988polynomial}
James Renegar, \emph{A polynomial-time algorithm, based on newton's method, for linear programming}, Mathematical programming \textbf{40} (1988), no.~1-3, 59--93.

\bibitem{renegar1993some}
\bysame, \emph{Some perturbation theory for linear programming}, Tech. report, Cornell University Operations Research and Industrial Engineering, 1993.

\bibitem{renegar1995incorporating}
\bysame, \emph{Incorporating condition measures into the complexity theory of linear programming}, SIAM Journal on Optimization \textbf{5} (1995), no.~3, 506--524.

\bibitem{renegar1995linear}
\bysame, \emph{Linear programming, complexity theory and elementary functional analysis}, Mathematical Programming \textbf{70} (1995), no.~1, 279--351.

\bibitem{renegar2001mathematical}
\bysame, \emph{A mathematical view of interior-point methods in convex optimization}, SIAM, 2001.

\bibitem{robinson1981some}
Stephen~M Robinson, \emph{Some continuity properties of polyhedral multifunctions}, Springer, 1981.

\bibitem{rockafellar1976monotone}
R~Tyrrell Rockafellar, \emph{Monotone operators and the proximal point algorithm}, SIAM journal on control and optimization \textbf{14} (1976), no.~5, 877--898.

\bibitem{rockafellar2009variational}
R~Tyrrell Rockafellar and Roger J-B Wets, \emph{Variational analysis}, vol. 317, Springer Science \& Business Media, 2009.

\bibitem{rosen1961gradient}
J.~Ben Rosen, \emph{The gradient projection method for nonlinear programming. part ii. nonlinear constraints}, Journal of the Society for Industrial and Applied Mathematics \textbf{9} (1961), no.~4, 514--532.

\bibitem{gurobinews}
Ed~Rothberg, \emph{New options for solving giant lps}, 2024, \url{https://cdn.gurobi.com/wp-content/uploads/New-Options-for-Solving-Giant-LPs.pdf}.

\bibitem{ruiz2001scaling}
Daniel Ruiz, \emph{A scaling algorithm to equilibrate both rows and columns norms in matrices}, Tech. report, CM-P00040415, 2001.

\bibitem{ryu2022large}
Ernest~K Ryu and Wotao Yin, \emph{Large-scale convex optimization: algorithms \& analyses via monotone operators}, Cambridge University Press, 2022.

\bibitem{sarker2007optimization}
Ruhul~Amin Sarker and Charles~S Newton, \emph{Optimization modelling: a practical approach}, CRC press, 2007.

\bibitem{schrijver1998theory}
Alexander Schrijver, \emph{Theory of linear and integer programming}, John Wiley \& Sons, 1998.

\bibitem{shin2020graph}
Sungho Shin, Carleton Coffrin, Kaarthik Sundar, and Victor~M Zavala, \emph{Graph-based modeling and decomposition of energy infrastructures}, arXiv preprint arXiv:2010.02404 (2020).

\bibitem{shin2023accelerating}
Sungho Shin, Fran{\c{c}}ois Pacaud, and Mihai Anitescu, \emph{Accelerating optimal power flow with {GPU}s: {SIMD} abstraction of nonlinear programs and condensed-space interior-point methods}, arXiv preprint arXiv:2307.16830 (2023).

\bibitem{stellato2020osqp}
Bartolomeo Stellato, Goran Banjac, Paul Goulart, Alberto Bemporad, and Stephen Boyd, \emph{Osqp: An operator splitting solver for quadratic programs}, Mathematical Programming Computation \textbf{12} (2020), no.~4, 637--672.

\bibitem{coptlink}
COPT team, \emph{Copt breaks barriers in linear programming with nvidia cuopt}, 2025, \url{https://www.shanshu.ai/news/breaking-barriers-in-linear-programming.html}.

\bibitem{toh1999sdpt3}
Kim-Chuan Toh, Michael~J Todd, and Reha~H T{\"u}t{\"u}nc{\"u}, \emph{Sdpt3—a matlab software package for semidefinite programming, version 1.3}, Optimization methods and software \textbf{11} (1999), no.~1-4, 545--581.

\bibitem{gurobilink}
Cara Touretzky, Robert Luce, and David~Torres Sanchez, \emph{Using gpus to solve lps: What's in it for me?}, 2025, \url{https://www.gurobi.com/resources/using-gpus-to-solve-lps-whats-in-it-for-me/}.

\bibitem{tseng1995linear}
Paul Tseng, \emph{On linear convergence of iterative methods for the variational inequality problem}, Journal of Computational and Applied Mathematics \textbf{60} (1995), no.~1-2, 237--252.

\bibitem{vandenberghe2010cvxopt}
Lieven Vandenberghe, \emph{The cvxopt linear and quadratic cone program solvers}, Online: http://cvxopt. org/documentation/coneprog. pdf \textbf{53} (2010).

\bibitem{vera2007primal}
Juan~Carlos Vera, Juan~Carlos Rivera, Javier Pena, and Yao Hui, \emph{A primal--dual symmetric relaxation for homogeneous conic systems}, Journal of Complexity \textbf{23} (2007), no.~2, 245--261.

\bibitem{villani2021topics}
C{\'e}dric Villani, \emph{Topics in optimal transportation}, vol.~58, American Mathematical Soc., 2021.

\bibitem{villani2009optimal}
C{\'e}dric Villani et~al., \emph{Optimal transport: old and new}, vol. 338, Springer, 2009.

\bibitem{wachter2006implementation}
Andreas W{\"a}chter and Lorenz~T Biegler, \emph{On the implementation of an interior-point filter line-search algorithm for large-scale nonlinear programming}, Mathematical programming \textbf{106} (2006), 25--57.

\bibitem{wright1997primal}
Stephen Wright, \emph{Primal-dual interior-point methods}, SIAM, 1997.

\bibitem{wright1993identifiable}
Stephen~J Wright, \emph{Identifiable surfaces in constrained optimization}, SIAM Journal on Control and Optimization \textbf{31} (1993), no.~4, 1063--1079.

\bibitem{xiong2024accessible}
Zikai Xiong, \emph{Accessible theoretical complexity of the restarted primal-dual hybrid gradient method for linear programs with unique optima}, arXiv preprint arXiv:2410.04043 (2024).

\bibitem{xiong2025high}
\bysame, \emph{High-probability polynomial-time complexity of restarted pdhg for linear programming}, arXiv preprint arXiv:2501.00728 (2025).

\bibitem{xiong2023relation}
Zikai Xiong and Robert~M Freund, \emph{On the relation between lp sharpness and limiting error ratio and complexity implications for restarted pdhg}, arXiv preprint arXiv:2312.13773 (2023).

\bibitem{xiong2024role}
\bysame, \emph{The role of level-set geometry on the performance of pdhg for conic linear optimization}, arXiv preprint arXiv:2406.01942 (2024).

\bibitem{xiong2023computational}
Zikai Xiong and Robert~Michael Freund, \emph{Computational guarantees for restarted pdhg for lp based on" limiting error ratios" and lp sharpness}, arXiv preprint arXiv:2312.14774 (2023).

\bibitem{ye1994nl}
Yinyu Ye, Michael~J Todd, and Shinji Mizuno, \emph{An $o (\sqrt{n}l)$-iteration homogeneous and self-dual linear programming algorithm}, Mathematics of operations research \textbf{19} (1994), no.~1, 53--67.

\bibitem{yudin1976informational}
David~B Yudin and Arkadi~S Nemirovskii, \emph{Informational complexity and efficient methods for the solution of convex extremal problems}, Matekon \textbf{13} (1976), no.~2, 22--45.

\bibitem{zheng2007metric}
Xi~Yin Zheng and Kung Fu~Ng, \emph{Metric subregularity and constraint qualifications for convex generalized equations in banach spaces}, SIAM Journal on Optimization \textbf{18} (2007), no.~2, 437--460.

\bibitem{zheng2014metric}
Xi~Yin Zheng and Kung~Fu Ng, \emph{Metric subregularity of piecewise linear multifunctions and applications to piecewise linear multiobjective optimization}, SIAM Journal on Optimization \textbf{24} (2014), no.~1, 154--174.

\bibitem{zhou2008linear}
Peng Zhou and Beng~Wah Ang, \emph{Linear programming models for measuring economy-wide energy efficiency performance}, Energy Policy \textbf{36} (2008), no.~8, 2911--2916.

\bibitem{zhu2008efficient}
Mingqiang Zhu and Tony Chan, \emph{An efficient primal-dual hybrid gradient algorithm for total variation image restoration}, Ucla Cam Report \textbf{34} (2008), 8--34.

\bibitem{zoutendijk1960methods}
Guus Zoutendijk, \emph{Methods of feasible directions}, 1960 (1960).

\bibitem{zoutendijk1970some}
\bysame, \emph{Some algorithms based on the principle of feasible directions}, Nonlinear programming, Elsevier, 1970, pp.~93--121.

\end{thebibliography}

\end{document}